\documentclass[12pt,a4paper]{article}

\usepackage{amsfonts,amsmath,amsthm}
\usepackage{graphicx}
\usepackage{amssymb}
\usepackage{mathrsfs}
\usepackage[top=2cm, bottom=2.8cm, left=1.5cm, right=1.5cm]{geometry}
\usepackage[all]{xy}
\usepackage{tikz-cd}
\usepackage{hyperref}
\usepackage{authblk}
\usepackage{enumitem}

\newtheorem{theorem}{Theorem}[section]
\newtheorem{proposition}[theorem]{Proposition}
\newtheorem{corollary}[theorem]{Corollary}
\newtheorem{definition}[theorem]{Definition}
\newtheorem{lemma}[theorem]{Lemma}

\newtheorem{thIntro}{Theorem}

\theoremstyle{remark}
\newtheorem{remark}[theorem]{Remark}
\newtheorem{example}[theorem]{Example}

\newcommand{\flecheIso}[4]{                     
            \begin{array}{rcl} #1 & \overset{\sim}{\to} & #2 \\   %
                         #3 &\mapsto & #4          %
            \end{array}}

\newcommand{\fonc}[5]{                     
            \begin{array}{crll}#1 :& #2 & \rightarrow & #3 \\   %
                         &#4 &\mapsto & #5          %
            \end{array}}

\newcommand{\foncIso}[5]{                     
            \begin{array}{crll}#1 :& #2 & \overset{\sim}{\rightarrow} & #3 \\   %
                         &#4 &\mapsto & #5          %
            \end{array}}

\DeclareMathOperator{\Ext}{Ext}
\DeclareMathOperator{\YExt}{YExt}
\DeclareMathOperator{\Hom}{Hom}

\title{\bf Davydov--Yetter cohomology \\and relative homological algebra}
\author[1]{M. Faitg}
\author[2]{A.M. Gainutdinov}
\author[1]{C. Schweigert}
\affil[1]{\small \textit{Fachbereich Mathematik, Universität Hamburg, Bundesstra{\ss}e 55, 20146 Hamburg, Germany}}
\affil[2]{\small \textit{Institut Denis Poisson, CNRS, Université de Tours, Universit\'e d’Orl\'eans, Parc de Grandmont, 37200 Tours, France}}

\date{}

\begin{document}

\begin{flushright}
\textsf{ZMP-HH/22-6
\\Hamburger Beitr\"age zur Mathematik Nr. 914}
\end{flushright}

{\let\newpage\relax\maketitle}

\maketitle

\vspace{-2em}

\noindent {\small {\em E-mail adresses:} matthieu.faitg@univ-tlse3.fr, azat.gainutdinov@lmpt.univ-tours.fr,
\\\hphantom{E-mail adresses: }christoph.schweigert@uni-hamburg.de}

\bigskip

\bigskip

\begin{abstract}
Davydov--Yetter (DY) cohomology classifies infinitesimal deformations of the monoidal structure of tensor functors and tensor categories. 
In this paper we provide new tools for the computation of the DY cohomology 
for finite tensor categories and exact functors between them. The key point
is to realize DY cohomology as relative Ext groups.
In particular, we prove that the infinitesimal deformations of a tensor category $\mathcal{C}$ are classified 
by the 3-rd self-extension group of the tensor unit 
of the Drinfeld center $\mathcal{Z}(\mathcal{C})$ relative to $\mathcal{C}$.
From classical results on relative homological algebra we get 
a long exact sequence for DY cohomology and a Yoneda product for which we provide an explicit formula.
Using the long exact sequence and duality, we obtain a
dimension formula for the cohomology groups based solely on relatively projective covers
which reduces a problem in homological algebra to a problem in representation theory, e.g. calculating the space of invariants in a certain object of $\mathcal{Z}(\mathcal{C})$.
Thanks to the Yoneda product, we also develop a method for computing DY cocycles 
explicitly which are needed for applications in the deformation theory.
We apply these tools to the category of finite-dimensional modules over a finite-dimensional Hopf algebra. We study in detail the examples of the bosonization of exterior algebras $\Lambda\mathbb{C}^k \rtimes \mathbb{C}[\mathbb{Z}_2]$, the Taft algebras and the small quantum group of $\mathfrak{sl}_2$ at a root of unity.
\end{abstract}

\bigskip

\bigskip


\newpage

\tableofcontents

\section{Introduction}

\indent A deformation theory of monoidal structures has been introduced and studied by Davydov, Crane and Yetter \cite{davydov, CY, yetter1, yetter2}; it describes deformations of the monoidal structure of a $k$-linear monoidal functor or the associator of a $k$-linear monoidal category (where $k$ is a field), without changing the underlying functor and categories. This theory is the first step to the classification problem of monoidal structures \cite{davydov} but is also related to quantum algebra and low-dimensional topology. Within this deformation theory, it was shown in \cite{DE} that the category of all modules over the enveloping algebra $U(\mathfrak{g})$ of a simple Lie algebra $\mathfrak{g}$ admits a one-parameter family of non-trivial deformations. One therefore expects to recover the category of modules over the quantum group $U_q(\mathfrak{g})$, see e.g.\,\cite{CP} for the definition of $U_q(\mathfrak{g})$. Also, this theory allows to deform the braiding of a tensor category and this can be used to produce link invariants; see \cite{yetter1} where a relation with Vassiliev invariants was established.

\smallskip

\indent We first recall a bit more precisely this deformation theory, which is often called Davydov--Yetter theory as e.g.\,\,in \cite[\S7.22]{EGNO}. Let $\mathcal{C}, \mathcal{D}$ be $k$-linear monoidal categories, assumed strict for simplicity, and let $F : \mathcal{C} \to \mathcal{D}$ be a tensor functor, \textit{i.e.}\,a $k$-linear monoidal functor. By definition, $F$ comes with a natural isomorphism $\theta_{X,Y} : F(X \otimes Y) \overset{\sim}{\to} F(X) \otimes F(Y)$ such that the diagram
\begin{equation}\label{diagramTensorStructureTheta}
\xymatrix@C=6em{
F(X \otimes Y \otimes Z) \ar[r]^{\theta_{X \otimes Y, Z}} \ar[d]_{\theta_{X, Y \otimes Z}} & F(X \otimes Y) \otimes F(Z) \ar[d]^{\theta_{X,Y} \otimes \mathrm{id}_{F(Z)}}\\
F(X) \otimes F(Y \otimes Z) \ar[r]_{\!\!\!\!\!\!\!\!\mathrm{id}_{F(X)} \otimes \theta_{Y,Z}} & F(X) \otimes F(Y) \otimes F(Z)
} \end{equation}
is commutative. Let $F$ be strict, \textit{i.e.}\,$\theta = \mathrm{id}$; this can be assumed without loss of generality because a tensor functor can always be strictified \cite{JS}. In Davydov--Yetter theory we consider infinitesimal deformations $\theta_h = \mathrm{id} + hf$ with $h^2 = 0$, where $f$ is a natural transformation $f_{X,Y} : F(X \otimes Y) \to F(X) \otimes F(Y)$, such that the diagram \eqref{diagramTensorStructureTheta} remains commutative with $\theta_h$ instead of $\theta$. Then the condition \eqref{diagramTensorStructureTheta} on $\theta_h$ implies
\begin{equation}\label{cocycleCondition}
\mathrm{id}_{F(X_1)} \otimes f_{X_2, X_3} - f_{X_1 \otimes X_2, X_3} + f_{X_1, X_2 \otimes X_3} - f_{X_1,X_2} \otimes \mathrm{id}_{F(X_3)} = 0.
\end{equation}
The space of $2$-cochains $C^2_{\mathrm{DY}}(F)$ is defined as the space of all natural endo-transformations of $F \otimes F$ while the subspace of $2$-cocycles contains the solutions of \eqref{cocycleCondition}. In other words, the differential $\delta^2(f)_{X_1, X_2, X_3}$ is defined as the expression at the left hand-side of \eqref{cocycleCondition}. The other cochain spaces are $C^n_{\mathrm{DY}}(F) = \mathrm{End}(F^{\otimes n})$, for the differentials $\delta^n$ see \eqref{defDYdifferential}. The resulting cohomology, denoted by $H^{\bullet}_{\mathrm{DY}}(F)$, is called the Davydov--Yetter (DY) cohomology of $F$. The infinitesimal deformations of the monoidal structure of $F$ are classified up to equivalence by $H^2_{\mathrm{DY}}(F)$, and it was shown in \cite{yetter1} that the obstructions are contained in $H^3_{\mathrm{DY}}(F)$. This in particular means that if $H^3_{\mathrm{DY}}(F) = 0$ an infinitesimal deformation can be extended to all orders.

\smallskip

\indent We note that the identity functor $F = \mathrm{Id}_{\mathcal{C}}$ deserves a special attention because $H^3_{\mathrm{DY}}(\mathrm{Id}_{\mathcal{C}})$ classifies the infinitesimal deformations of the (trivial) associator of $\mathcal{C}$. Such a deformation is an  expansion $a_h = \mathrm{id} + hg$ over $k[h]/\langle h^2 \rangle$ which satisfies the pentagon equation, where $g$ is a natural transformation $g_{X,Y,Z} : X \otimes Y \otimes Z \to X \otimes Y \otimes Z$. The obstructions are contained in $H^4_{\mathrm{DY}}(\mathrm{Id}_{\mathcal{C}})$, at least for the extension of an infinitesimal deformation to the order $2$ \cite[Prop.\,3.21]{BD}. We will denote $H^n_{\mathrm{DY}}(\mathcal{C})$ instead of $H^n_{\mathrm{DY}}(\mathrm{Id}_{\mathcal{C}})$.

\smallskip

\indent In this paper we study DY cohomology for finite tensor categories and exact tensor functors between them (see our conventions at the beginning of \S \ref{relativeExtTensorCategories}). Such categories are equivalent to $A$-mod for a finite-dimensional $k$-algebra $A$ and are equipped with a rigid monoidal structure (see \S \ref{relativeExtTensorCategories}). They are ubiquitous in quantum algebra and mathematical physics, for instance from logarithmic CFTs \cite{GR17,FGR,CGR}, small quantum groups~\cite{lusztig, GLO18,N18} or more generally finite-dimensional Hopf algebras.

\smallskip

\indent In general, cohomology theories admit coefficients; for instance in Hochschild cohomology the coefficients are bimodules. Coefficients for DY cohomology have been introduced in \cite{GHS} and are objects in $\mathcal{Z}(F)$, namely the centralizer of the tensor functor $F : \mathcal{C} \to \mathcal{D}$. The category $\mathcal{Z}(F)$ is rigid and monoidal and plays a crucial role in this paper; its objects are pairs $\mathsf{V} = (V, \rho^V)$ where $V \in \mathcal{D}$ and $\rho^V : V \otimes F(-) \to F(-) \otimes V$ is like a half-braiding for $V$ (see more details in \S \ref{sectionAdjunctionCentralizerFunctor}). In particular $\mathcal{Z}(\mathrm{Id}_ {\mathcal{C}})$ is just the Drinfeld center $\mathcal{Z}(\mathcal{C})$. The DY cohomology with coefficients $\mathsf{V}, \mathsf{W} \in \mathcal{Z}(F)$ is denoted by $H^{\bullet}_{\mathrm{DY}}(F;\mathsf{V},\mathsf{W})$ and the cohomology without coefficients $H^{\bullet}_{\mathrm{DY}}(F)$ is recovered as $H^{\bullet}_{\mathrm{DY}}(F; \mathbf{1}, \mathbf{1})$, where $\mathbf{1}$ is the tensor unit of $\mathcal{Z}(F)$ (trivial coefficients).

\smallskip

\indent An obvious question is: how to compute Davydov--Yetter cohomology? This can be divided into two problems:
\begin{itemize}
\item[(P1)] Compute the dimension of the Davydov--Yetter cohomology groups.
\item[(P2)] Determine explicit cocycles. This question is especially relevant for $2$-cocycles (or $3$-cocycles for the identity functor) since they give rise to infinitesimal deformations.
\end{itemize}
In general it is difficult to solve these problems using only the definition of the DY complex $C^{\bullet}_{\mathrm{DY}}(F)$ because the cochain spaces grow fast and natural transformations are in general not suitable for explicit computations. A more successful strategy is to use methods from homological algebra. For instance in \cite{DE} the DY cohomology of $U(\mathfrak{g})$-Mod for a simple Lie algebra $\mathfrak{g}$ has been computed thanks to a natural filtration on this DY complex and using the associated spectral sequence. But in general the DY complex of a tensor category does not admit a natural filtration.

\smallskip

\indent In the case of finite tensor categories $\mathcal{C}, \mathcal{D}$ and an exact tensor functor $F : \mathcal{C} \to \mathcal{D}$, the forgetful functor $\mathcal{U} : \mathcal{Z}(F) \to \mathcal{D}$ is also exact and thus admits a left adjoint $\mathcal{F}$. We relate the DY cohomology to the relative Ext groups of this adjunction:
\begin{thIntro}\label{thIntroIsoDYRelExt}
The Davydov--Yetter cohomology of an exact tensor functor $F : \mathcal{C} \to \mathcal{D}$ with coefficients $\mathsf{V}, \mathsf{W} \in \mathcal{Z}(F)$ is isomorphic to the relative Ext groups of the adjunction defined by the forgetful functor $\mathcal{Z}(F) \to \mathcal{D}$:
\[ H^n_{\mathrm{DY}}(F; \mathsf{V}, \mathsf{W}) \cong \mathrm{Ext}^n_{\mathcal{Z}(F), \mathcal{D}}(\mathsf{V}, \mathsf{W}) \]
for all $n \geq 0$.
\end{thIntro}
\noindent Relative Ext groups are cohomology groups associated to a resolvent pair
$\xymatrix{
\mathcal{A} \ar@<1ex>[r]^{\mathcal{U}} & \ar[l]^{\mathcal{F}} \mathcal{B}
}$, 
which is an adjunction between abelian categories such that the right adjoint $\mathcal{U}$ is additive, exact and faithful. The definition and computation of these abelian groups is similar to that of the usual Ext groups, but with ``relative versions'' of projective objects and projective resolutions; this classical subject \cite{hochschild, macLane} is reviewed in \S \ref{relExtGroups}. We first prove in general that the relative Ext groups are isomorphic to the cohomology of the comonad $G = \mathcal{F}\mathcal{U}$ (a new result to the best of our knowledge); comonad cohomology \cite{BB} is reviewed in \S\ref{sectionRelExtGroupsAsComonadCohomology}. Then, to obtain Theorem \ref{thIntroIsoDYRelExt} we use the isomorphism discovered in \cite{GHS} between DY cohomology and the comonad cohomology of the adjunction $\mathcal{F} \dashv \mathcal{U}$ discussed above Theorem \ref{thIntroIsoDYRelExt}. We thus obtain the following equivalence relations between the three cohomology theories:
\begin{equation}\label{diagramIntro}
\begingroup%
  \makeatletter%
  \providecommand\color[2][]{%
    \errmessage{(Inkscape) Color is used for the text in Inkscape, but the package 'color.sty' is not loaded}%
    \renewcommand\color[2][]{}%
  }%
  \providecommand\transparent[1]{%
    \errmessage{(Inkscape) Transparency is used (non-zero) for the text in Inkscape, but the package 'transparent.sty' is not loaded}%
    \renewcommand\transparent[1]{}%
  }%
  \providecommand\rotatebox[2]{#2}%
  \newcommand*\fsize{\dimexpr\f@size pt\relax}%
  \newcommand*\lineheight[1]{\fontsize{\fsize}{#1\fsize}\selectfont}%
  \ifx\svgwidth\undefined%
    \setlength{\unitlength}{423.34139454bp}%
    \ifx\svgscale\undefined%
      \relax%
    \else%
      \setlength{\unitlength}{\unitlength * \real{\svgscale}}%
    \fi%
  \else%
    \setlength{\unitlength}{\svgwidth}%
  \fi%
  \global\let\svgwidth\undefined%
  \global\let\svgscale\undefined%
  \makeatother%
  \begin{picture}(1,0.24707552)%
    \lineheight{1}%
    \setlength\tabcolsep{0pt}%
    \put(-0.00186861,0.21470007){\color[rgb]{0,0,0}\makebox(0,0)[lt]{\lineheight{1.25}\smash{\begin{tabular}[t]{l}Davydov--Yetter cohomology\end{tabular}}}}%
    \put(0.0835469,0.17831855){\color[rgb]{0,0,0}\makebox(0,0)[lt]{\lineheight{1.25}\smash{\begin{tabular}[t]{l}$H^{\bullet}_{\mathrm{DY}}(F;\mathsf{V},\mathsf{W})$\end{tabular}}}}%
    \put(0.69467787,0.21343458){\color[rgb]{0,0,0}\makebox(0,0)[lt]{\lineheight{1.25}\smash{\begin{tabular}[t]{l}Comonad cohomology\end{tabular}}}}%
    \put(0.6899306,0.17863488){\color[rgb]{0,0,0}\makebox(0,0)[lt]{\lineheight{1.25}\smash{\begin{tabular}[t]{l}$H^{\bullet}_G\bigl( \mathsf{V}, \Hom_{\mathcal{Z}(F)}(?,\mathsf{W}) \bigr)$\end{tabular}}}}%
    \put(0.39703047,0.04224189){\color[rgb]{0,0,0}\makebox(0,0)[lt]{\lineheight{1.25}\smash{\begin{tabular}[t]{l}Relative Ext groups\end{tabular}}}}%
    \put(0.41790838,0.00488975){\color[rgb]{0,0,0}\makebox(0,0)[lt]{\lineheight{1.25}\smash{\begin{tabular}[t]{l}$\mathrm{Ext}^{\bullet}_{\mathcal{Z}(F),\mathcal{D}}(\mathsf{V},\mathsf{W})$\end{tabular}}}}%
    \put(0,0){\includegraphics[width=\unitlength,page=1]{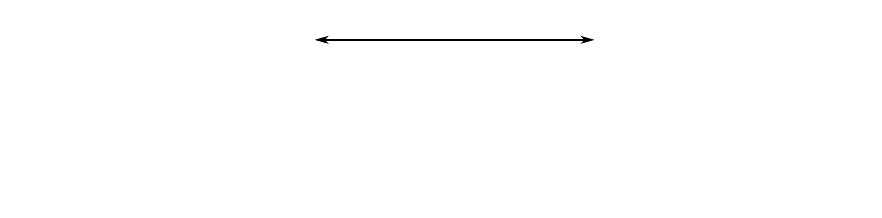}}%
    \put(0.39773427,0.21965115){\color[rgb]{0,0,0}\makebox(0,0)[lt]{\lineheight{1.25}\smash{\begin{tabular}[t]{l}\cite[Thm.\,3.11]{GHS}\end{tabular}}}}%
    \put(0,0){\includegraphics[width=\unitlength,page=2]{diagram_Intro.pdf}}%
    \put(0.76286105,0.09034732){\color[rgb]{0,0,0}\makebox(0,0)[lt]{\lineheight{1.25}\smash{\begin{tabular}[t]{l}Special case\\of Prop. \ref{relExtAndComonadCohom}\end{tabular}}}}%
    \put(0.19659011,0.08234352){\color[rgb]{0,0,0}\makebox(0,0)[lt]{\lineheight{1.25}\smash{\begin{tabular}[t]{l}Thm.\,\ref{thIntroIsoDYRelExt}\end{tabular}}}}%
    \put(0,0){\includegraphics[width=\unitlength,page=3]{diagram_Intro.pdf}}%
  \end{picture}%
\endgroup%

\end{equation}
\noindent The isomorphisms in the diagram \eqref{diagramIntro} also hold on the level of complexes, with all relations made explicit in \S \ref{sectionRelationDYCohomology}. In particular, the cochain complex isomorphism for the dashed arrow is explained below Corollary \ref{DYrelExt}.

\smallskip

\indent Working with relative Ext groups instead of comonad cohomology allows us to use powerful results from relative homological algebra, which is a well-established subject \cite{macLane}: existence of a long exact sequence of relative Ext groups associated to a short exact sequence, Yoneda product, Yoneda description of relative $\Ext^n$ groups in terms of certain $n$-fold exact sequences. These results are reviewed in \S \ref{subsectionPropRelExtGroups}. Our goal here is to use these properties in the case of the resolvent pair $\mathcal{Z}(F) \leftrightarrows \mathcal{D}$ and to obtain the corresponding properties for Davydov--Yetter cohomology through the isomorphism in Theorem \ref{thIntroIsoDYRelExt}. In that way we obtain the following results:
\begin{enumerate}
\item Proposition \ref{propH1Equals0}: $H^1_{\mathrm{DY}}(F) = 0$ provided that the tensor functor $F$ is exact and the ground field $k$ has characteristic $0$ and is algebraically closed.
\item Corollary \ref{coroLongExactSequenceDYCohomology}: long exact sequence for Davydov--Yetter cohomology with coefficients.
\item Corollary \ref{corollaryDimensionDYGroups}: dimension formulas for the Davydov--Yetter cohomology groups.
\item Theorem \ref{propYonedaProductDY}: explicit formula of the Yoneda product on Davydov--Yetter cohomology.
\item Proposition \ref{propositionDYGroupsAreModules}: graded module structure of $H^{\bullet}_{\mathrm{DY}}(F;\mathsf{V},\mathsf{W})$ over the Yoneda product algebra $H^{\bullet}_{\mathrm{DY}}(F)$, for any coefficients $\mathsf{V},\mathsf{W} \in \mathcal{Z}(F)$.
\item The method in \S \ref{methodDYCocyclesThanksToSequences} for constructing explicit DY cocycles for $\mathcal{C} = H\text{-}\mathrm{mod}$ and $F = \mathrm{Id}_{\mathcal{C}}$, where $H\text{-}\mathrm{mod}$ is the category of finite-dimensional modules over a finite-dimensional Hopf $k$-algebra $H$.
\end{enumerate}
These results uncover the importance of DY coefficients; in particular every $n$-cocycle can be written as a product of $n$ $1$-cocyles  with non-trivial coefficients (Lemma \ref{factorization1CocyclesDY}). All these properties are helpful to compute the DY cohomology on examples, as demonstrated in \S \ref{sectionExamples} for certain finite-dimensional Hopf algebras. Let us now give more details on how these results are obtained.

\smallskip

\indent The first item follows from the fact that $\mathcal{Z}(F)$ is a finite tensor category when $F$ is exact; therefore if $k$ has characteristic $0$ and is algebraically closed then $\Ext^1_{\mathcal{Z}(F)}(\mathbf{1}, \mathbf{1}) = 0$ by \cite[Thm.\,4.4.1]{EGNO} from which we derive $\Ext^1_{\mathcal{Z}(F),\mathcal{D}}(\mathbf{1}, \mathbf{1}) = 0$. The second item is immediate from the corresponding statement on relative Ext groups recalled in \S \ref{sectionLongExactSequenceExtGroups}.

\smallskip

\indent Before explaining the third item, we note that the resolvent pair $\mathcal{Z}(F) \leftrightarrows \mathcal{D}$ is monoidal, \textit{i.e.} the forgetful functor $\mathcal{U} : \mathcal{Z}(F) \to \mathcal{D}$ is monoidal. We show in \S \ref{sectionMonoidalResolventPairs} that the relative Ext groups of a monoidal resolvent pair have properties similar to those of usual Ext groups for tensor categories: the relatively projective objects form a tensor ideal and relative Ext groups are compatible with the duality. From these properties and the long exact sequence, we express the relative Ext groups in terms of Hom spaces in Proposition \ref{propositionExtWithRelProjCovers}. The associated dimension formula in Corollary \ref{CoroDimExtWithHom} requires only to find a relatively projective cover (see definition in \S \ref{subsectionRelProjCover}) and to determine the dimension of certain Hom spaces. For example, if the ground field $k$ has characteristic $0$ and is algebraically closed, the formula for the second cohomology is
\begin{equation}\label{formulaDimH2Intro}
\dim H_{\mathrm{DY}}^2(F) = \dim \Hom_{\mathcal{Z}(F)}\bigl(\mathsf{K},\mathsf{K}^{\vee}\bigr) - \dim \Hom_{\mathcal{Z}(F)}\bigl(\mathsf{P},\mathsf{K}^{\vee}\bigr)
\end{equation}
where $\mathsf{P}$ is the relatively projective cover of the unit object and $\mathsf{K} = \mathrm{ker}(\mathsf{P} \twoheadrightarrow \boldsymbol{1})$. This is especially efficient in practice and gives the dimension of the space of infinitesimal deformations of $F$. We give the general formula for all degrees in  Corollary \ref{corollaryDimensionDYGroups}, which is the third item of the above list. In \S \ref{subsubsectionDYcohomologyBarUi} we apply it to the example of $\mathcal{C} = \bar U_{\mathbf{i}}(\mathfrak{sl}_2)\text{-}\mathrm{mod}$ with $F = \mathrm{Id}_{\mathcal{C}}$ and we compute $\dim\bigl( H^n_{\mathrm{DY}}(F) \bigr)$ for $n \leq 4$; this calculation was not accessible with previous methods. Therefore we have a new method to address the problem (P1).

\smallskip

\indent To address the problem (P2) and to explain the items 4 and 5, we recall in \S \ref{sectionYonedaDescriptionOfRelativeExtGroups} the Yoneda description $\YExt^n_{\mathcal{A},\mathcal{B}}$ of the relative Ext groups $\Ext^n_{\mathcal{A},\mathcal{B}}$, based on exact sequences in $\mathcal{A}$ of length $n$ that are split in $\mathcal{B}$. The point is that these $n$-fold exact sequences are not so difficult to construct (in particular they are easier than constructing relatively projective resolutions) and that we have isomorphisms
\begin{equation}\label{isoYExtExtDYIntro}
\YExt^n_{\mathcal{Z}(F),\mathcal{D}}(\mathsf{V}, \mathsf{W}) \overset{\sim}{\to} \Ext^n_{\mathcal{Z}(F),\mathcal{D}}(\mathsf{V}, \mathsf{W}) \overset{\sim}{\to} H^n_{\mathrm{DY}}(F;\mathsf{V},\mathsf{W})
\end{equation}
which associates DY cocycles to exact sequences, for all $n \geq 0$ and $\mathsf{V}, \mathsf{W} \in \mathcal{Z}(F)$.

\smallskip

\indent There is a natural gluing operation $\circ : \YExt_{\mathcal{A},\mathcal{B}}^n(V,W) \times \YExt_{\mathcal{A},\mathcal{B}}^m(U,V) \to \YExt_{\mathcal{A},\mathcal{B}}^{m+n}(U,W)$ for $n$-fold exact sequences called Yoneda product, see \S \ref{sectionTheYonedaProduct}. We transport the Yoneda product through the isomorphism \eqref{isoYExtExtDYIntro}, in order to get an associative product on DY cohomology:
\[ \circ : \:\:H^n_{\mathrm{DY}}(F;\mathsf{V},\mathsf{W}) \times H^m_{\mathrm{DY}}(F;\mathsf{U},\mathsf{V}) \to H^{m+n}_{\mathrm{DY}}(F;\mathsf{U},\mathsf{W}). \]
The explicit expression of this product is determined in Theorem \ref{propYonedaProductDY} and turns out to be very simple; it is an important tool in practice because every cocycle can be factored as a product of $1$-cocycles. The fifth item in the above list follows from the fact that the product $\circ$ endows $H^{\bullet}(F)$ with the structure of a graded associative algebra and that $H_{\mathrm{DY}}^{\bullet}(F; \mathsf{U}, \mathsf{V})$ can be naturally endowed with the structure of a graded $H^{\bullet}(F)$-module; we also determine the formula of the corresponding action.

\smallskip

\indent In section \ref{sectionFinDimHopfAlgebras} we turn to the case $\mathcal{C} = H\text{-}\mathrm{mod}$ where $H$ is a finite-dimensional Hopf algebra over a field $k$ and we obtain explicit results for $F = \mathrm{Id}_{\mathcal{C}}$. Note that $\mathcal{Z}(\mathrm{Id}_{H\text{-}\mathrm{mod}}) \cong D(H)\text{-}\mathrm{mod}$, where $D(H)$ is the Drinfeld double of $H$ and $\mathcal{U} : D(H)\text{-}\mathrm{mod} \to H\text{-}\mathrm{mod}$ is the restriction functor for the standard embedding $H \hookrightarrow D(H) = (H^*)^{\mathrm{op}} \otimes H$. The isomorphism of Theorem \ref{thIntroIsoDYRelExt} then becomes
\begin{equation}\label{DYHmodIntro}
H^n_{\mathrm{DY}}(H\text{-}\mathrm{mod}; V, W) \cong \Ext^n_{D(H),H}(V, W)
\end{equation}
for all $n \geq 0$ and $V, W \in D(H)\text{-}\mathrm{mod}$. The Yoneda Ext group $\YExt^n_{D(H),H}(V,W)$ consists of exact sequences of $D(H)$-modules of length $n$ from $W$ to $V$ which split as sequences in $H$-mod. The method of \S \ref{methodDYCocyclesThanksToSequences}, based on the isomorphism \eqref{isoYExtExtDYIntro} and on its compatibility with the Yoneda product, associates to such an exact sequence an explicit DY $n$-cocycle. A key feature in this case is that the Davydov--Yetter cochains can be encoded in a very explicit manner in terms of $H$-linear maps, suitable for computations. Finally, we note in \S \ref{sectionFactoHopfAlg} that when $H$ is factorizable the isomorphism $D(H) \cong H^{\otimes 2}$ as algebras \cite{schneider} implies a convenient reformulation of the groups $\Ext_{D(H),H}$ as $\Ext_{H \otimes H, \Delta(H)}$ where $\Delta : H \to H \otimes H$ is the coproduct of $H$. As a consequence, when $k$ is algebraically closed, we show that $\mathsf{P}$ and $\mathsf{K}$ from \eqref{formulaDimH2Intro} are just the principal block of $H$ (the one containing the cointegral) and the kernel of the counit, respectively. Therefore, one can calculate the DY cohomologies in this case without even calculating the Drinfeld double. This is a quite useful fact that we employ in the case of small quantum groups at odd orders of roots of $1$.

\smallskip

\indent We apply these results and methods to examples in section \ref{sectionExamples}:
\begin{itemize}
\item In \S \ref{DYBk} we consider the bosonization $H = \Lambda(\mathbb{C}^k) \rtimes \mathbb{C}[\mathbb{Z}_2]$ of the exterior algebra of $\mathbb{C}^k$. We give explicit $2$-cocycles and show that $H^{\bullet}_{\mathrm{DY}}\bigl( \Lambda(\mathbb{C}^k) \rtimes \mathbb{C}[\mathbb{Z}_2]\text{-}\mathrm{mod} \bigr)$ is the polynomial algebra generated by them, thanks to the Yoneda product. This example was considered in \cite{GHS}, where only the dimensions of the DY cohomology groups were determined. In Proposition \ref{lemmaCongruenceBk}, we obtain the $2$-cocycles by constructing $2$-fold exact sequences in $D(H)$-mod and computing their images through the isomorphism \eqref{isoYExtExtDYIntro} with the method of \S \ref{methodDYCocyclesThanksToSequences}.
\item In \S \ref{sectionExampleTaftAlgebra} we consider the Taft algebras $H = T_q$ for every root of unity $q \in \mathbb{C}$. We show that $H^{\bullet}_{\mathrm{DY}}(T_q\text{-}\mathrm{mod}) \cong \mathbb{C}[X^2]$ as a graded algebra (where $X$ is of degree $1$) and we compute the explicit Davydov--Yetter cocycle in each even degree.

\item In \S \ref{sectionExampleBarUiSl2} we consider $H = \bar U_{\mathbf{i}}(\mathfrak{sl}_2)$, namely the restricted quantum group of $\mathfrak{sl}_2$ at $q = \mathbf{i} = \sqrt{-1}$, that appeared in a relation to vertex-operators algebras in \cite{FGST}. This example is much more involved: it is too difficult to construct a relatively projective resolution of the trivial $D(\bar U_{\mathbf{i}}(\mathfrak{sl}_2))$-module $\mathbb{C}$. Instead we determine only the relatively projective cover of $\mathbb{C}$ and we apply the dimension formulas of Corollary \ref{corollaryDimensionDYGroups}. We find:
\[ H^2_{\mathrm{DY}}(\bar U_{\mathbf{i}}(\mathfrak{sl}_2)\text{-}\mathrm{mod}) = 0, \quad \dim\bigl(H^3_{\mathrm{DY}}(\bar U_{\mathbf{i}}(\mathfrak{sl}_2)\text{-}\mathrm{mod}) \big) = 3, \quad H^4_{\mathrm{DY}}(\bar U_{\mathbf{i}}(\mathfrak{sl}_2)\text{-}\mathrm{mod}) = 0. \]
The first equality is easily obtained by \eqref{formulaDimH2Intro} while the two last are computationally harder. To overcome this, we note that computing the dimension of a Hom space is equivalent to counting invariants, which is a linear-algebra problem easily solved by a computer. We also construct two $3$-fold exact sequences and determine the corresponding DY $3$-cocycles through the isomorphism \eqref{isoYExtExtDYIntro} by the method of \S \ref{methodDYCocyclesThanksToSequences}. Then using a GAP program we prove that they are linearly independent in the corresponding cohomology and we find a third basis element of $H^3_{\mathrm{DY}}(\bar U_{\mathbf{i}}(\mathfrak{sl}_2)\text{-}\mathrm{mod})$. This gives a complete picture of the infinitesimal associators on $\bar U_{\mathbf{i}}$-mod in Proposition \ref{propDYbarUi}, that can be extended to order $2$ because of trivial $H^4_{\mathrm{DY}}$.

\item In \S \ref{sectionSmallQuantumGroupH2}, we apply our reformulation of the groups $\Ext_{D(H),H}$ for factorizable Hopf algebras and results from \cite{FGST} to prove in a few lines that $H^2_{\mathrm{DY}}\bigl(u_q(\mathfrak{sl}_2)\text{-}\mathrm{mod}\bigr) = 0$, where  $u_q(\mathfrak{sl}_2)$ is the small quantum group at a root of unity $q \in \mathbb{C}$ of odd order. For the case $q^3 = 1$ we also calculate $H^3_{\mathrm{DY}}(u_q(\mathfrak{sl}_2)\text{-}\mathrm{mod}) \cong \mathbb{C}$ using the same strategy as for $\bar U_{\mathbf{i}}$.
\end{itemize}

\smallskip

\indent The paper is organized as follows. In \S \ref{sectionRelativeHomAlg} we review the main facts of relative homological algebra, provide some results on relatively projective covers and observe that relative Ext groups are isomorphic to comonad cohomology. In \S \ref{sectionMonoidalResolventPairs} we give the special properties of relative Ext groups associated to a monoidal adjunction which allow us to derive a dimension formula for these groups in terms of relatively projective covers (see Corollary \ref{CoroDimExtWithHom}). In \S \ref{sectionAdjunctionCentralizerFunctor} we collect facts about the adjunction $\mathcal{Z}(F) \leftrightarrows \mathcal{D}$ while \S\S \ref{sectionRelationDYCohomology}--\ref{subsectionLongExactSequenceDYCohom} contain the main results of this paper on DY cohomology. In \S \ref{sectionFinDimHopfAlgebras} we specialize our results to the case of the identity functor on $H\text{-}\mathrm{mod}$ where $H$ is a finite-dimensional Hopf algebra; moreover we give a method to construct explicit DY cocycles in this case and provide further results for factorizable Hopf algebras in \S \ref{sectionFactoHopfAlg}. In \S\ref{sectionExamples} we apply the results of \S\ref{sectionFinDimHopfAlgebras} to a series of examples mentioned above.

\medskip

\noindent \textbf{Acknowledgments.} We thank P. Etingof for correspondence and T. D\'ecoppet for pointing out a gap in the first version of Proposition \ref{propATensBrelB}. We also thank the anonymous referee for careful reading and valuable comments. M.F. and C.S. are partially supported by the Deutsche Forschungsgemeinschaft (DFG, German Research Foundation) under Germany's Excellence Strategy - EXC 2121 “Quantum Universe”- 390833306. 
The work of A.M.G. was supported by the CNRS, and partially by the ANR grant JCJC ANR-18-CE40-0001 and the RSF Grant No.\ 20-61-46005. 
A.M.G. is also grateful to Hamburg University for its kind hospitality in 2021 and 2022.

\section{Relative homological algebra}\label{sectionRelativeHomAlg}
This section collects results on resolvent pairs, relatively projective objects, relative Ext groups, Yoneda product, \textit{etc}.\ following \cite{macLane}. It also contains results that we have not been able to find in the literature: for instance the existence and properties of relatively projective covers for quite general resolvent pairs in \S \ref{subsectionRelProjCover} or the relation between the notions from relative homological algebra and those from comonad cohomology in \S \ref{sectionRelExtGroupsAsComonadCohomology}.

\subsection{Resolvent pairs and relative $\Ext$ groups}\label{relExtGroups}
\indent Let $\mathcal{A}, \mathcal{B}$ be abelian categories and
\begin{equation}\label{adjunction}
\xymatrix@R=.7em{
\mathcal{A}\ar@/^.7em/[dd]^{\mathcal{U}}\\
\dashv\\
\ar@/^.7em/[uu]^{\mathcal{F}}\mathcal{B}
}
\end{equation}
be a pair of adjoint functors, where $\mathcal{F}$ is left adjoint to $\mathcal{U}$. We say that this adjunction is a {\em resolvent pair} if the functor $\mathcal{U} : \mathcal{A} \to \mathcal{B}$ is additive, exact and faithful. In all the sequel we assume that \eqref{adjunction} is a resolvent pair. Before defining the relative Ext groups associated to \eqref{adjunction}, we recall some notions of relative homological algebra from \cite[Chap.\,IX]{macLane}.

\begin{itemize}[topsep=.3em, itemsep=.1em]
\item A morphism $f \in \Hom_{\mathcal{A}}(V,W)$ is called {\em allowable} if there exists $s \in \Hom_{\mathcal{B}}(\mathcal{U}(W),\mathcal{U}(V))$ such that $\mathcal{U}(f) s \, \mathcal{U}(f) = \mathcal{U}(f)$. This condition is equivalent to the requirement that $\mathcal{U}(\ker(f))$ is a direct summand\footnote{Recall that in an additive category an object $X$ is a direct summand of $Y$ if there exist morphisms $\iota : X \to Y$ and $\pi : Y \to X$ such that $\pi\iota = \mathrm{id}_X$.} of $\mathcal{U}(V)$ and $\mathcal{U}(\mathrm{im}(f))$ is a direct summand of $\mathcal{U}(W)$.
\item An object $P$ of $\mathcal{A}$ is called {\em relatively projective} if the diagram (with exact row)
\[ \xymatrix{
& P \ar@{-->}[ld] \ar[d]^g & \\
V \ar[r]_{f} & W \ar[r] & 0
}\]
can always be filled, whenever $f$ is allowable and $g$ is any morphism in $\mathcal{A}$.

\item A {\em relatively projective resolution} of $V \in \mathcal{A}$ is an exact sequence
\begin{equation}\label{relativeProjectiveResolution}
0 \longleftarrow V \overset{d_0}{\longleftarrow} P_0 \overset{d_1}{\longleftarrow} P_1 \overset{d_2}{\longleftarrow} \ldots
 \end{equation}
such that each $P_i \in \mathcal{A}$ is relatively projective and each $d_i$ is allowable. The latter is equivalent to the statement that the image by $\mathcal{U}$ of the exact sequence splits, \textit{i.e.} $\mathcal{U}(P_i) \cong \mathrm{im}\bigl( \mathcal{U}(d_i) \bigr) \oplus \ker\bigl( \mathcal{U}(d_i) \bigr)$ for all $i$.
\end{itemize}

\noindent Note that the definitions are analogous to the ones in usual homological algebra, the difference being the notion of allowable morphism. 

\smallskip

\indent Let $\varepsilon : \mathcal{F}\mathcal{U} \to \mathrm{Id}_{\mathcal{A}}$ be the counit of the adjunction \eqref{adjunction}. Then for all $V \in \mathcal{A}$, $\varepsilon_V : \mathcal{F}\mathcal{U}(V) \to V$ is an epimorphism and is allowable since we have $\mathcal{U}(\varepsilon_V)\eta_{\mathcal{U}(V)} = \mathrm{id}_{\mathcal{U}(V)}$, where $\eta : \mathrm{Id}_{\mathcal{B}} \to \mathcal{U}\mathcal{F}$ is the unit of the adjunction. Thanks to $\varepsilon_V$ one can construct an important relatively projective resolution of $V$, called the bar resolution \cite[Thm.\,6.3]{macLane}:
\begin{equation}\label{relativeBarResolution}
\mathrm{Bar}^{\bullet}_{\mathcal{A},\mathcal{B}}(V) = \left( 0 \longleftarrow V \overset{\varepsilon_V}{\longleftarrow} \mathcal{F}\mathcal{U}(V) \overset{d_1}{\longleftarrow} (\mathcal{F}\mathcal{U})^2(V) \overset{d_2}{\longleftarrow} \ldots \right)
\end{equation}
where
\begin{equation}\label{differentialRelativeBarResolution}
d_n = \sum_{i=0}^n (-1)^i(\mathcal{F}\mathcal{U})^{n-i}\bigl(\varepsilon_{(\mathcal{F}\mathcal{U})^i(V)}\bigr).
\end{equation}
This shows in particular that any $V \in \mathcal{A}$ admits at least one relatively projective resolution. However, bar resolutions are not suitable for concrete computations because the objects $(\mathcal{F}\mathcal{U})^n(V)$ grow too fast. One can obtain smaller relatively projective objects thanks to the following basic properties:

\begin{lemma}\label{lemmaBasicPropertiesRelProj}
1. Let $M = V \oplus W$ be a direct sum in $\mathcal{A}$. Then $M$ is relatively projective if and only if $V$ and $W$ are relatively projective.
\\2. An object $P \in \mathcal{A}$ is relatively projective if and only if it is a direct summand of $\mathcal{F}(X)$ for some $X \in \mathcal{B}$.
\end{lemma}
\begin{proof}
1. This is the same statement as in usual homological algebra. The easy proof is left to the reader.
\\2. We know by \cite[Thm.\,6.1]{macLane} that $\mathcal{F}(X)$ is relatively projective. Hence, direct summands of $\mathcal{F}(X)$ are relatively projective as well thanks to the previous item. Conversely, assume that $P$ is relatively projective. Then since $\varepsilon_P : \mathcal{F}\mathcal{U}(P) \to P$ is an allowable epimorphism we can fill the diagram
\[ \xymatrix{
& P \ar@{-->}[ld]_{\iota} \ar[d]^{\mathrm{id}_P} & \\
\mathcal{F}\mathcal{U}(P) \ar[r]_{\quad\varepsilon_P} & P \ar[r] & 0
}\]
Hence $\varepsilon_P \iota = \mathrm{id}_P$ and $P$ is a direct summand of $\mathcal{F}(\mathcal{U}(P))$.
\end{proof}

\begin{definition}\label{defRelExt}
Given a resolvent pair as in \eqref{adjunction} and a relatively projective resolution as in \eqref{relativeProjectiveResolution}, the $n$-th cohomology group of the complex of abelian groups
\[ 0 \longrightarrow \Hom_{\mathcal{A}}(P_0,W) \overset{d^*_1}{\longrightarrow} \Hom_{\mathcal{A}}(P_1,W) \overset{d^*_2}{\longrightarrow} \Hom_{\mathcal{A}}(P_2,W) \overset{d^*_3}{\longrightarrow} \ldots \]
is denoted by $\Ext^n_{\mathcal{A},\mathcal{B}}(V,W)$ and is called the $n$-th relative $\Ext$ group of $V$ and $W$. \\We denote by $[\alpha] \in \Ext_{\mathcal{A},\mathcal{B}}^n(V,W)$ the equivalence class of a cocycle $\alpha \in \Hom_{\mathcal{A}}(P_n,W)$.
\end{definition}
\noindent Thanks to the fundamental lemma of relative homological algebra \cite[Chap.\,IX, Thm.\,4.3]{macLane} (which is the generalization of the well-known statement in usual homological algebra), $\Ext^{\bullet}_{\mathcal{A},\mathcal{B}}(V,W)$ does not depend on the choice of a relatively projective resolution. For the particular choice of the bar resolution \eqref{relativeBarResolution}, we denote the complex in Definition \ref{defRelExt} by
\[ \mathrm{Bar}^{\bullet}_{\mathcal{A},\mathcal{B}}(V,W) = \left( 0 \longrightarrow \Hom_{\mathcal{A}}\bigl(\mathcal{F}\mathcal{U}(V),W\bigr) \overset{d_1^*}{\longrightarrow} \Hom_{\mathcal{A}}\bigr((\mathcal{F}\mathcal{U})^2(V),W\bigl) \overset{d_2^*}{\longrightarrow} \ldots \right). \]
Note that though the adjunction is not apparent in the notation $\Ext^n_{\mathcal{A},\mathcal{B}}(V,W)$ it is always assumed to be fixed, as different adjunctions yield in general different cohomologies since they yield different classes of allowable morphisms and of relatively projective objects.

\begin{example}\label{exampleSemisimpleBottom}
If the category $\mathcal{B}$ in \eqref{adjunction} is semisimple, then $\Ext^{\bullet}_{\mathcal{A},\mathcal{B}}(V,W) = \Ext^{\bullet}_{\mathcal{A}}(V,W)$. Indeed, any sub-object in $\mathcal{B}$ is a direct summand. In particular for any $f \in \Hom_{\mathcal{A}}(V,W)$, the objects $\mathcal{U}(\ker(f))$ and $\mathcal{U}(\mathrm{im}(f))$ are direct summands of $\mathcal{U}(V)$ and $\mathcal{U}(W)$, respectively, and thus $f$ is allowable. It follows that a relatively projective object is projective in $\mathcal{A}$ and that any projective resolution in $\mathcal{A}$ is relatively projective.
\end{example}

\indent Let $A$ be a finite-dimensional algebra over a field $k$ and $R \subset A$ be a subalgebra. The forgetful functor $\mathcal{U} : A\text{-}\mathrm{mod} \to R\text{-}\mathrm{mod}$ has induction as left adjoint: $\mathcal{F}(X) = A \otimes_R X$. This adjunction is an example of a resolvent pair. We write $\Ext_{A,R}$ instead of $\Ext_{A\text{-}\mathrm{mod}, R\text{-}\mathrm{mod}}$.  For instance if $R=A$ then every object is relatively projective and $\mathrm{Ext}^n_{A,A} = 0$ for all $n > 0$; if $R=k$ then an object is relatively projective if and only if it is projective and $\mathrm{Ext}^n_{A,k} = \mathrm{Ext}^n_A$ for all $n \geq 0$. The next example is another case of this type of resolvent pair:
\begin{example}\label{exampleAtensBRelBInVect}
Let $A,B$ be two finite-dimensional $k$-algebras. Take $\mathcal{A} = (A \otimes B)\text{-mod}, \mathcal{B} = B\text{-mod}$, where $B$ is identified with the subalgebra $1 \otimes B \subset A \otimes B$. If $V, V'$ are (finite-dimensional) $A$-modules and $W, W'$ are (finite-dimensional) $B$-modules we have
\begin{equation}\label{exampleExtABB}
\Ext_{A\otimes B,B}^n(V \boxtimes W, V' \boxtimes W') \cong \Ext_A^n(V,V') \otimes \Hom_B(W,W')
\end{equation}
where $V \boxtimes W$ is $V \otimes_k W$ with the action $(a \otimes b)\cdot(v \otimes w) = (a\cdot v) \otimes (b \cdot w)$. To see this isomorphism, note first that the induction functor is $\mathcal{F}(W) = A \boxtimes W$. If $P$ is a projective $A$-module, then the $(A \otimes B)$-module $P \boxtimes W$ is relatively projective for any $W \in B\text{-}\mathrm{mod}$. Indeed, as a projective $A$-module, $P$ is a direct summand of a free module $A^{\oplus n}$ for some $n \geq 0$ so that $P \boxtimes W$ is a direct summand of 
\[ A^{\oplus n} \boxtimes W \cong (A \boxtimes W)^{\oplus n} \cong A \boxtimes W^{\oplus n} = \mathcal{F}(W^{\oplus n}). \]
Now, take a projective resolution $0 \leftarrow V \overset{d_0}{\longleftarrow} P_0 \overset{d_1}{\longleftarrow} P_1 \overset{d_2}{\longleftarrow} \ldots$ in $A\text{-mod}$ and consider 
\begin{equation}\label{relProjResolutionABB}
0 \longleftarrow V \boxtimes W \xleftarrow{d_0 \otimes \mathrm{id}} P_0 \boxtimes W \xleftarrow{d_1 \otimes \mathrm{id}} P_1 \boxtimes W \xleftarrow{d_2 \otimes \mathrm{id}} \ldots.
\end{equation}
It is an exact sequence of relatively projective modules. Moreover each $d_i$ is allowable: indeed, since the $B$-action on $\mathcal{U}(P_i \boxtimes S)$ is simply $b\cdot (x \otimes w) = x \otimes (b\cdot w)$, $\mathcal{U}\bigl(\ker(d_i \otimes \mathrm{id})\bigr) = \mathcal{U}\big( \ker(d_i) \boxtimes W \big)$ is a direct summand of $\mathcal{U}(P_i \boxtimes W)$ and similarly $\mathcal{U}\bigl(\mathrm{im}(d_i \otimes \mathrm{id})\bigr) = \mathcal{U}\big( (\mathrm{im}(d_i) \boxtimes W \big)$ is a direct summand of $\mathcal{U}(P_{i-1} \boxtimes W)$). Hence \eqref{relProjResolutionABB} is a relatively projective resolution.  Then \eqref{exampleExtABB} follows from the equality $\Hom_{A \otimes B}(P_i \boxtimes W, V' \boxtimes W') = \Hom_{A}(P_i,V') \otimes \Hom_{B}(W,W')$. In section \ref{sectionAlgebrasInBraidedTensorCategories} we will consider an analogous statement for algebras in braided finite tensor categories.
\end{example}

\indent The next lemma is an immediate generalization of a well-known fact in homological algebra and will be useful later; it can be found in \cite[Chap.\,IX, Prop.\,4.2]{macLane} but we recall the proof for convenience.

\begin{lemma}\label{relProjAndHom}
An object $P \in \mathcal{A}$ is relatively projective if and only if the functor $\Hom_{\mathcal{A}}(P,-)$ sends allowable short exact sequences in $\mathcal{A}$ to short exact sequences (of abelian groups).
\end{lemma}
\begin{proof}
Let $0 \to X \overset{j}{\to} Y \overset{\sigma}{\to} Z \to 0$ be an allowable short exact sequence in $\mathcal{A}$. Recall that the functor $\Hom_{\mathcal{A}}(P,-)$ is left exact for all $P \in \mathcal{A}$ which means that
\[ 0 \longrightarrow \Hom_{\mathcal{A}}(P,X) \overset{j_*}{\longrightarrow} \Hom_{\mathcal{A}}(P,Y) \overset{\sigma_*}{\longrightarrow} \Hom_{\mathcal{A}}(P,Z) \]
is automatically exact. Thus the sequence 
\[ 0 \longrightarrow \Hom_{\mathcal{A}}(P,X) \overset{j_*}{\longrightarrow} \Hom_{\mathcal{A}}(P,Y) \overset{\sigma_*}{\longrightarrow} \Hom_{\mathcal{A}}(P,Z) \longrightarrow 0 \]
is exact if and only if $\sigma_* = \Hom_{\mathcal{A}}(P,\sigma)$ is surjective. But note that the definition of a relatively projective object is equivalent to the statement that if $\sigma$ is an allowable epimorphism then $\Hom_{\mathcal{A}}(P,\sigma)$ is surjective. Hence the equivalence.
\end{proof}

\subsection{Properties of relative $\Ext$ groups}\label{subsectionPropRelExtGroups}
In this section we consider a resolvent pair of abelian categories as in \eqref{adjunction} and we discuss properties of the associated relative Ext groups: the Yoneda description of these groups, the Yoneda product and the long exact sequence of Ext groups associated to a short exact sequence of coefficients.

\subsubsection{Yoneda description of relative $\Ext$ groups}\label{sectionYonedaDescriptionOfRelativeExtGroups}
A $n$-fold exact sequence from $W$ to $V$ is an exact sequence in $\mathcal{A}$ of the form
\begin{equation}\label{nFoldExactSequence}
0 \longrightarrow W \overset{f_n}{\longrightarrow} X_n \overset{f_{n-1}}{\longrightarrow} \ldots \overset{f_1}{\longrightarrow} X_1 \overset{f_0}{\longrightarrow} V \longrightarrow 0.
\end{equation}
It is called allowable if each $f_i$ is allowable. This condition is equivalent to the property that the exact sequence
\[ 0 \longrightarrow \mathcal{U}(W) \xrightarrow{\mathcal{U}(f_n)} \mathcal{U}(X_n) \xrightarrow{\mathcal{U}(f_{n-1})} \ldots \xrightarrow{\mathcal{U}(f_1)} \mathcal{U}(X_1) \xrightarrow{\mathcal{U}(f_0)} \mathcal{U}(V) \longrightarrow 0\]
splits in $\mathcal{B}$.
\begin{definition}
The $n$-th relative Yoneda Ext group of $V$ and $W$, denoted by $\YExt^n_{\mathcal{A},\mathcal{B}}(V,W)$ is the set of all allowable $n$-fold exact sequences from $W$ to $V$, modulo congruence relations. We denote by $[S]$ the congruence class of a sequence $S$.
\end{definition}
\noindent We do not explain what are the congruence relations in $\YExt^n_{\mathcal{A},\mathcal{B}}(V,W)$ and refer to \cite{macLane}, Chap. III (\S 1 and \S 5) and Chap. XII (\S 4).

\smallskip

\indent For further use, let us recall from \cite[Chap.\,III, Thm.\,6.4]{macLane}\footnote{In \cite{macLane} both Ext groups are denote by $\Ext^n$. In Theorem 6.4 of Chapter III, $\YExt^n$ is denoted $\Ext^n$ while $\Ext^n$ is denoted $H^n$. Moreover this theorem is stated for usual $\Ext$ groups, but in \S 5, Chap. XII it is said that it holds for relative $\Ext$ groups.} the construction of the isomorphism
\begin{equation}\label{defBarEta}
\overline{\eta} : \YExt^n_{\mathcal{A},\mathcal{B}}(V,W) \overset{\sim}{\to} \Ext^n_{\mathcal{A},\mathcal{B}}(V,W).
\end{equation}
We first define a map $\eta$ at the level of cocycles. Let $S$ be a $n$-fold allowable exact sequence as in \eqref{nFoldExactSequence} and let $ \ldots \overset{d_1}{\longrightarrow} P_0 \overset{d_0}{\longrightarrow} V \longrightarrow 0$ be a relatively projective resolution of $V$. Fill the dashed arrows in the diagram (whose bottom row is $S$)
\begin{equation}\label{IsoYExtExtOnCocycles}
\xymatrix{
P_n \ar[r]^{d_n} \ar@{-->}[d]^{\eta_n = \eta(S)} & P_{n-1} \ar[r]^{d_{n-1}} \ar@{-->}[d]^{\eta_{n-1}} & \ldots \ar[r]^{d_1} & P_0 \ar[r]^{d_0} \ar@{-->}[d]^{\eta_0} & V \ar[r] \ar@{=}[d] & 0\\
W \ar[r]_{f_n} & X_n \ar[r]_{f_{n-1}} & \ldots \ar[r]_{f_1} & X_1 \ar[r]_{f_0} & V \ar[r] & 0
}\end{equation}
such that it becomes commutative. This is always possible by the fundamental lemma of relative homological algebra \cite[Chap.\,IX, Thm.\,4.3]{macLane}, since the top row is a relatively projective resolution. Then the morphism $\eta_n \in \Hom_{\mathcal{A}}(P_n,W)$ is a cocycle of the complex from Definition \ref{defRelExt} which we denote by $\eta(S)$. Indeed, $f_n \eta_n d_{n+1} = \eta_{n-1} d_n d_{n+1} = 0$ and since $f_n$ is a monomorphism we get $\eta_n d_{n+1} = 0$. This cocycle  depends on the choice of the relatively projective resolution of $V$ but one shows that $\eta$ sends congruent sequences to cohomologous cocycles, yielding a morphism $\overline{\eta} : \YExt^n_{\mathcal{A},\mathcal{B}}(V,W) \to \Ext^n_{\mathcal{A},\mathcal{B}}(V,W)$  defined by $\overline{\eta}([S]) = [\eta(S)]$; in particular $\overline{\eta}([S])$ does not longer depend on the choice of the relatively projective resolution. Finally, $\overline{\eta}$ turns out to be an isomorphism.

\smallskip

\indent With the Yoneda description of $\Ext$ groups, it is easy to see that there is an injective morphism
\begin{equation}\label{injectionRelExt1FullExt1}
\Ext^1_{\mathcal{A},\mathcal{B}}(V,W) \hookrightarrow \Ext^1_{\mathcal{A}}(V,W).
\end{equation}
Indeed, $\YExt^1_{\mathcal{A},\mathcal{B}}(V,W)$ (resp. $\YExt^1_{\mathcal{A}}(V,W)$) consists of congruence classes of allowable short exact sequences (resp. congruence classes of short exact sequences) from $W$ to $V$. In both cases, the zero element is the split exact sequence $0 \to W \to V \oplus W \to V \to 0$ and two sequences $0 \to W \to Z \to V \to 0$, $0 \to W \to Z' \to V \to 0$ are congruent if there exists a morphism in $\Hom_{\mathcal{A}}(Z,Z')$ such that the obvious diagram is commutative (see pages 64 and 368 in \cite{macLane}). Hence, if a sequence is equal to $0$ in $\YExt^1_{\mathcal{A}}(V,W)$, then it is equal to $0$ in $\YExt^1_{\mathcal{A},\mathcal{B}}(V,W)$. This property is not true for higher $\Ext$ groups; indeed, two $n$-fold exact sequences (resp. allowable exact sequences) $S, T$ are equal in $\YExt^n_{\mathcal{A}}(V,W)$ (resp. in $\YExt^n_{\mathcal{A},\mathcal{B}}(V,W)$) if there exists a chain of morphisms like for instance $S \to C_1 \leftarrow C_2 \to C_3 \leftarrow T$, where the $C_i$ are $n$-fold exact sequences (resp. allowable exact sequences). Hence it might happen that two allowable exact sequences are not equal $\YExt^n_{\mathcal{A},\mathcal{B}}(V,W)$ but they are equal in $\YExt^n_{\mathcal{A}}(V,W)$ because they can be related by a chain of morphisms using non-allowable exact sequences. For an example, see item 2 in Remark \ref{remarkRelExtNot0UsualExt0BarUi} below.

\smallskip

\indent In general, it is not easy to prove that a $n$-fold allowable exact sequence is not congruent to $0$ (\textit{i.e.}\,that it is not equal to $0$ in the corresponding Yoneda Ext group), except for $n=1$ where congruence is just isomorphism of short exact sequences. For $n=2$ we have this useful criterion \cite[Chap.\,XII, Lem.\,5.3]{macLane}:

\begin{lemma}\label{lemma2FoldSequences}
Let $S = \bigl( 0 \to W \to X_2 \overset{\text{\footnotesize $\pi$}}{\longrightarrow} K \to 0 \bigr)$ and $T = \bigl( 0 \to K \overset{\text{\footnotesize $j$}}{\longrightarrow} X_1 \to V \to 0 \bigr)$ be allowable short exact sequences in $\mathcal{A}$ and let $S \circ T = \bigl( 0 \to W \to X_2 \overset{\text{\footnotesize $j\pi$}}{\longrightarrow} X_1 \to V \to 0 \bigr)$ be their Yoneda product (see \S \ref{sectionTheYonedaProduct} below). The following are equivalent:
\begin{enumerate}[itemsep=0em, topsep=.3em]
\item $S \circ T \equiv 0$.
\item There exists an allowable short exact sequence $P = \bigl(0 \to W \to M \to X_1 \to 0\bigr)$ such that $S = Pj$, where $Pj$ denotes the pullback of $P$ by $j$.
\item There exists an allowable short exact sequence $Q = \bigl(0 \to X_2 \to N \to V \to 0\bigr)$ such that $T = \pi Q$, where $\pi Q$ denotes the push-forward of $Q$ by $\pi$.
\end{enumerate}
\end{lemma}
\noindent By definition of the pullback and of the push-forward of a short exact sequence, items 2 and 3 mean respectively that there exist commutative diagrams (with exact rows):
\[\xymatrix{
S = \big(0 \ar[r] & W \ar[r] \ar@{=}[d] & X_2 \ar[r] \ar@{-->}[d] & K \ar[r] \ar[d]^{j} & 0\big)\\
P = \big(0 \ar[r] & W \ar[r] & M \ar[r] & X_1 \ar[r] & 0\big)
} \qquad \text{ and } \qquad
\xymatrix{
Q= \big(0 \ar[r] & X_2 \ar[r] \ar[d]^{\pi} & N \ar[r] \ar@{-->}[d] & V \ar[r] \ar@{=}[d] & 0\big)\\
T = \big(0 \ar[r] & K \ar[r] & X_1 \ar[r] & V \ar[r] & 0\big)
}\]
This lemma can be extended to $n$-fold exact sequences (see \cite[Chap.\,XII, Lem.\,5.5]{macLane}) but it becomes much more difficult to apply in practice.
\begin{corollary}\label{remarkCriterion2FoldSequences}
Let $S,T$ be allowable short exact sequences as above.
\begin{enumerate}[itemsep=0em, topsep=.3em]
\item If $S \not\equiv 0$ and $\Ext^1_{\mathcal{A},\mathcal{B}}(X_1,W) = 0$, then $S \circ T \not \equiv 0$.
\item If $T \not\equiv 0$ and $\Ext^1_{\mathcal{A},\mathcal{B}}(V,X_2) = 0$, then $S \circ T \not \equiv 0$
\end{enumerate}
\end{corollary}
\begin{proof}
We show the first item, the proof of the second being analogous. Any allowable short exact sequence $P$ as in Lemma \ref{lemma2FoldSequences} belongs to $\YExt^1_{\mathcal{A},\mathcal{B}}(X_1,W) \cong \Ext^1_{\mathcal{A},\mathcal{B}}(X_1,W) = 0$ and thus is congruent to $0$. It follows that $Pj \equiv 0$ as well. But then $S = Pj$ is impossible since $S \not \equiv 0$. Hence, by the equivalence in Lemma \ref{lemma2FoldSequences}, $S \circ T \not \equiv 0$.
\end{proof}
\noindent See \S \ref{explicitCocyclesTaft} for a use of this corollary on an example.

\subsubsection{The Yoneda product}\label{sectionTheYonedaProduct}
\indent For two allowable exact sequences from $W$ to $V$ and from $V$ to $U$, their Yoneda product $\circ$ is an allowable exact sequence from $W$ to $U$ defined by
\begin{align}
&\bigl(0 \longrightarrow W \overset{f_n}{\longrightarrow} Y_n \overset{f_{n-1}}{\longrightarrow} \ldots \overset{f_1}{\longrightarrow} Y_1 \overset{f_0}{\longrightarrow} V \longrightarrow 0\bigr) \circ \bigl(0 \longrightarrow V \overset{g_m}{\longrightarrow} X_m \overset{g_{m-1}}{\longrightarrow} \ldots \overset{g_1}{\longrightarrow} X_1 \overset{g_0}{\longrightarrow} U \longrightarrow 0\bigr) \nonumber\\
&= 0 \longrightarrow W \overset{f_n}{\longrightarrow} Y_n \overset{f_{n-1}}{\longrightarrow} \ldots \overset{f_1}{\longrightarrow} Y_1 \xrightarrow{g_mf_0}  X_m \overset{g_{m-1}}{\longrightarrow} \ldots \overset{g_1}{\longrightarrow} X_1 \overset{g_0}{\longrightarrow} U \longrightarrow 0 \label{defYonedaProduct}
\end{align}
Any $n$-fold allowable exact sequence can be written as a product of $n$ allowable short exact sequences \cite[Chap.\,III, \S 5]{macLane}.

\smallskip

\indent The Yoneda product is compatible with the congruence relations and yields a bilinear map 
\[ \fonc{\circ}{\YExt_{\mathcal{A},\mathcal{B}}^n(V,W) \times \YExt_{\mathcal{A},\mathcal{B}}^m(U,V)}{\YExt_{\mathcal{A},\mathcal{B}}^{m+n}(U,W)}{([S],[S'])}{[S \circ S']} \]
We recall how to compute the corresponding product (also called Yoneda product) on cocycles:
\begin{equation}\label{yonedaProductOnExt}
\circ : \Ext_{\mathcal{A},\mathcal{B}}^n(V,W) \times \Ext_{\mathcal{A},\mathcal{B}}^m(U,V) \to \Ext_{\mathcal{A},\mathcal{B}}^{m+n}(U,W)
\end{equation}
along the isomorphism $\overline{\eta}$ from \eqref{defBarEta}. Let $\ldots \overset{d^U_1}{\longrightarrow} P_0 \overset{d^U_0}{\longrightarrow} U \longrightarrow 0$ and $\ldots \overset{d^V_1}{\longrightarrow} Q_0 \overset{d^V_0}{\longrightarrow} V \longrightarrow 0$ be relatively projective resolutions of $U$ and $V$ respectively and let $\alpha \in \Hom_{\mathcal{A}}(Q_n,W)$ and $\beta \in \Hom_{\mathcal{A}}(P_m,V)$ be cocycles with respect to these resolutions. Fill the dashed arrows in the diagram
\begin{equation}\label{diagramLiftYoneda}
\xymatrix@C=2.5em@R=2.5em{
P_{m+n} \ar@{-->}[d]^{\widetilde{\beta}_n} \ar[r]^{d^U_{m+n}} & \: \ldots \: \ar[r]^{d^U_{m+2}} & P_{m+1} \ar@{-->}[d]^{\widetilde{\beta}_1} \ar[r]^{d^U_{m+1}} & P_m \ar@{-->}[d]^{\widetilde{\beta}_0} \ar[rd]^{\beta} & & \\
Q_n \ar[r]^{d^V_n} & \: \ldots \: \ar[r]^{d^V_2} & Q_1 \ar[r]^{d^V_1} & Q_0 \ar[r]^{d_0^V} & V \ar[r] & 0
}
\end{equation}
such that it becomes commutative (this is always possible thanks to the fundamental lemma of relative homological algebra \cite[Chap.\,IX, Thm.\,4.3]{macLane}). Then $\alpha \circ \beta$ is defined to be the composition $\alpha\widetilde{\beta}_n \in \Hom_{\mathcal{A}}(P_{m+n}, W)$ and $[\alpha] \circ [\beta]$ is defined to be $[\alpha \circ \beta] \in \Ext_{\mathcal{A},\mathcal{B}}^{m+n}(V,W)$. The definition is such that the map $\eta$ from \eqref{IsoYExtExtOnCocycles} satisfies $\eta(S \circ S') = \eta(S) \circ \eta(S')$, and thus in particular $\overline{\eta}([S] \circ [S']) = \overline{\eta}([S]) \circ \overline{\eta}([S'])$.

\subsubsection{The long exact sequence for relative $\Ext$ groups}\label{sectionLongExactSequenceExtGroups}
\indent Note that $\Ext^n_{\mathcal{A},\mathcal{B}}(-,-)$ is a bifunctor $\mathcal{A}^{\mathrm{op}} \times \mathcal{A} \to \mathbf{A\!b}$ for all $n$, because it is an instance of a derived functor \cite[Chap.\,XII, \S 9]{macLane}.  Let $U,V,W$ be objects in $\mathcal{A}$ and let $f : U \to V$ be a morphism in $\mathcal{A}$; here we only need to recall how to define the morphism of abelian groups
\[ f^* = \Ext_{\mathcal{A},\mathcal{B}}^n(f,W) : \Ext_{\mathcal{A},\mathcal{B}}^n(V,W) \to \Ext_{\mathcal{A},\mathcal{B}}^n(U,W). \]
We first define $f^*$ on cochains. Take relative projective resolutions $\ldots \to P_1 \to P_0 \to U \to 0$, $\ldots \to Q_1 \to Q_0 \to V \to 0$ of $U$, $V$ respectively, and let $\alpha \in \mathrm{Hom}_{\mathcal{A}}(Q_n, W)$ be a cocycle representing a class $[\alpha] \in \Ext^n_{\mathcal{A},\mathcal{B}}(V,W)$ (see Definition \ref{defRelExt}). Fill the dashed arrows in the diagram
\[ \xymatrix{
P_n \ar[r] \ar@{-->}[d]^{\widetilde{f}_n} & P_{n-1} \ar[r] \ar@{-->}[d]^{\widetilde{f}_{n-1}} & \ldots \ar[r] & P_0 \ar[r] \ar@{-->}[d]^{\widetilde{f}_0}& U \ar[r] \ar[d]^f & 0\\
Q_n \ar[r] & Q_{n-1} \ar[r] & \ldots \ar[r] & Q_0 \ar[r] & V \ar[r] & 0
} \]
such that it becomes commutative (again this is always possible thanks to the fundamental lemma of relative homological algebra). Then $\alpha \widetilde{f}_n \in \mathrm{Hom}_{\mathcal{A}}(P_n, W)$ is a cocycle which we denote by $f^*(\alpha)$ and we put $f^*([\alpha]) = [f^*(\alpha)] \in \Ext_{\mathcal{A},\mathcal{B}}^n(U,W)$.

\begin{theorem}\label{thLongExactSequenceExtGroups} {\em \cite[Chap.\,XII, Thm.\,5.1]{macLane}\footnote{Note that in \cite{macLane} this theorem is stated in a more general setting based on the notion of proper class of short exact sequences. Here we take the class of all allowable short exact sequences in $\mathcal{A}$ (\textit{i.e.}\,the short exact sequences which arrows are allowable morphisms), which is a proper class as remarked in \cite[p.\,368]{macLane}.}}
Let $S = \big(0 \longrightarrow U \overset{j}{\longrightarrow} V \overset{\pi}{\longrightarrow} W \longrightarrow 0 \big)$
be an allowable short exact sequence in $\mathcal{A}$ and let $N$ be any object in $\mathcal{A}$. Then we have the long exact sequence of abelian groups
\[\xymatrix{
0 \ar[r] & \Hom_{\mathcal{A}}(W,N) \ar[r]^{\pi^*} & \Hom_{\mathcal{A}}(V,N) \ar[r]^{j^*} & \Hom_{\mathcal{A}}(U,N) \ar[lld]_{c^0} &\\
& \Ext^1_{\mathcal{A},\mathcal{B}}(W,N) \ar[r]^{\pi^*} & \Ext^1_{\mathcal{A},\mathcal{B}}(V,N) \ar[r]^{j^*} & \Ext^1_{\mathcal{A},\mathcal{B}}(U,N) \ar[lld]_{c^1} & \\
& \Ext^2_{\mathcal{A},\mathcal{B}}(W,N) \ar[r]^{\pi^*} & \Ext^2_{\mathcal{A},\mathcal{B}}(V,N) \ar[r]^{j^*} & \Ext^2_{\mathcal{A},\mathcal{B}}(U,N) & \!\!\!\!\!\!\!\!\!\!\!\!\!\!\!\!\!\!\!\ldots
} \]
\end{theorem}
\noindent For each $n$, $\pi^*$ and $j^*$ are the pullbacks $\Ext^n_{\mathcal{A},\mathcal{B}}(\pi,N)$ and $\Ext^n_{\mathcal{A},\mathcal{B}}(j,N)$. The maps $c^n$ are called connecting morphisms and are defined by 
\[ c^n(\alpha) = (-1)^n \alpha \circ \overline{\eta}(S) \]
where $\circ$ is the Yoneda product defined after \eqref{yonedaProductOnExt} and $\overline{\eta}([S]) \in \Ext^1_{\mathcal{A},\mathcal{B}}(W,U)$ is the element associated to $[S] \in \YExt^1_{\mathcal{A},\mathcal{B}}(W,U)$ by the isomorphism from \eqref{defBarEta}.

\smallskip

\indent The long exact sequence from Theorem \ref{thLongExactSequenceExtGroups} gives in particular an inductive formula for the relative $\Ext$ groups:
\begin{corollary}\label{corollaryInductionFormulaExt}
Let $0 \longrightarrow L \overset{i}{\longrightarrow} Q \overset{p}{\longrightarrow} V \longrightarrow 0$ be an allowable short exact sequence in $\mathcal{A}$, with $Q$ a relatively projective object. We have
\begin{align*}
&\Ext_{\mathcal{A},\mathcal{B}}^n(V,W) \cong \Ext_{\mathcal{A},\mathcal{B}}^{n-1}(L,W) \:\: \text{ for n} > 1, \\
&\Ext_{\mathcal{A},\mathcal{B}}^1(V,W) \cong \Hom_{\mathcal{A}}(L,W)/\mathrm{im}(i^*)
\end{align*}
where $i^*$ is the pullback $\Hom_{\mathcal{A}}(Q, W) \to \Hom_{\mathcal{A}}(L, W)$.
\end{corollary}
\begin{proof}
Since $Q$ is relatively projective we have $\Ext^n_{\mathcal{A},\mathcal{B}}(Q,-) = 0$ for all $n \geq 1$. Thanks to Theorem \ref{thLongExactSequenceExtGroups} we get for all $n > 1$ an exact sequence
\begin{equation*}
\Ext^{n-1}_{\mathcal{A},\mathcal{B}}(Q,W) = 0 \overset{i^*}{\longrightarrow} \Ext^{n-1}_{\mathcal{A},\mathcal{B}}(L,W) \overset{c^{n-1}}{\longrightarrow} \Ext^{n}_{\mathcal{A},\mathcal{B}}(V,W) \overset{p^*}{\longrightarrow} \Ext^n_{\mathcal{A},\mathcal{B}}(Q,W) = 0
\end{equation*}
which implies that $c^{n-1}$ is an isomorphism of vector spaces. For the case $n=1$, again thanks to Theorem \ref{thLongExactSequenceExtGroups} we have the exact sequence
\[ 0 \longrightarrow \Hom_{\mathcal{A}}(V,W) \overset{p^*}{\longrightarrow} \Hom_{\mathcal{A}}(Q,W) \overset{i^*}{\longrightarrow} \Hom_{\mathcal{A}}(L,W) \overset{c^0}{\longrightarrow} \Ext^1_{\mathcal{A},\mathcal{B}}(V,W) \overset{p^*}{\longrightarrow} \Ext^1_{\mathcal{A},\mathcal{B}}(Q,W) = 0 \]
which gives the desired isomorphism.
\end{proof}

\subsection{Relatively projective cover}\label{subsectionRelProjCover}
\indent Relatively projective covers have been defined and studied in \cite{thevenaz} in the special case where $\mathcal{A}$ is the category of modules over a ring, $\mathcal{B}$ is the category of modules over a subring and the adjunction is given by the usual restriction and induction functors.

\smallskip

\indent Here we only assume that $\mathcal{A}$ is an abelian category {\em in which every object has finite length}; the abelian category $\mathcal{B}$ and the additive functors $\mathcal{U}, \mathcal{F}$ remain arbitrary, except of course that they form a resolvent pair as in \eqref{adjunction}. We prove existence, uniqueness and some properties of the relatively projective cover in this general setting. Recall that the {\em length} of an object $V \in \mathcal{A}$, denoted by $\ell(V)$, is the length of any of its Jordan--H\"older series \cite[\S 1.5]{EGNO}. A basic fact is that if $V \in \mathcal{A}$ and $S \hookrightarrow V$ is a subobject then $\ell(V/S) = \ell(V) - \ell(S)$, where $V/S = \mathrm{coker}(S \hookrightarrow V)$. As a result:
\begin{equation}\label{propLengthEpiMono}
\text{If } \ell(V) = \ell(W) \text{ then }
\begin{array}{l}
e : V \to W \text{ is an epimorphism} \: \implies\: e \text{ is an isomorphism.} \\
j : V \to W \text{ is a monomorphism} \: \implies\: j \text{ is an isomorphism.}
\end{array}
\end{equation}
For the first claim we use that $V/\ker(e) \cong W$ implies $\ell\bigl(\ker(e)\bigr) = 0$ and thus $e$ is a monomorphism; but in an abelian category a morphism which is both mono and epi is an isomorphism. For the second claim $\ell\bigl( W/\mathrm{im}(j) \bigr) = \ell(W) - \ell\bigl(\mathrm{im}(j)\bigr) = \ell(W) - \ell(V) = 0$.

\smallskip

\indent As explained in \cite[\S 1]{thevenaz}, there are two possible definitions of a relatively projective cover (as for usual projective covers) based respectively on the notions of minimal epimorphism and relatively essential epimorphism. Recall that an allowable epimorphism $\pi : M \to V$ in $\mathcal{A}$ is called
\begin{itemize}
\item {\em minimal} if for all $f \in \mathrm{End}_{\mathcal{A}}(M)$,
\begin{equation}\label{defMinimalEpi}
\pi f = \pi \quad \implies \quad f \text{ is an isomorphism.}
\end{equation}
\item {\em relatively essential} if for every subobject $\iota : S \hookrightarrow M$,
\[ \pi\iota : S \to V \text{ is an allowable epimorphism} \quad \implies \quad \iota \:\text{ is an isomorphism}. \]
We write $\pi \iota = \pi|_S$ when the monomorphism $\iota$ is implicit.
\end{itemize}
Thanks to the assumption that any object of $\mathcal{A}$ has finite length we have:
\begin{lemma}\label{lemmaMinimalEssential}
Let $P$ be a relatively projective object and $\pi : P \to V$ be an allowable epimorphism. Then $\pi$ is minimal if and only if it is relatively essential.
\end{lemma}
\begin{proof}
Assume that $\pi$ is minimal and let $\iota : S \hookrightarrow P$ be a subobject such that $\pi \iota$ is an allowable epimorphism. We can fill the diagram
\[ \xymatrix{
& P \ar@{-->}[ld]_f \ar[d]^{\pi} & \\
S \ar[r]_{\pi \iota} & V \ar[r] & 0
}\]
Since $\pi$ is minimal it follows that $\iota f$ is an isomorphism and hence that $\iota$ is an epimorphism. But $\iota$ is also a monomorphism by definition, so it is an isomorphism.
\\\indent Conversely, assume that $\pi$ is relatively essential and let $f : P \to P$ be such that $\pi f = \pi$. Write $f = je$ where $e : P \to I$ is an epimorphism and $j : I \to P$ is a monomorphism. Since $(\pi j) e = \pi f = \pi$ is an epimorphism, we have that $\pi j$ is an epimorphism. Moreover $\pi j$ is allowable: indeed, since $\pi$ is allowable and is an epimorphism there exists $s \in \Hom_{\mathcal{B}}\bigl( \mathcal{U}(V),\mathcal{U}(P) \bigr)$ such that $\mathcal{U}(\pi)s = \mathrm{id}_{\mathcal{U}(V)}$, so we have
\[ \mathcal{U}(\pi j) \, \mathcal{U}(e)s = \mathcal{U}(\pi f) s = \mathcal{U}(\pi) s = \mathrm{id}_{\mathcal{U}(V)}. \]
It follows that $j : I \to P$ is an isomorphism because $\pi$ is relatively essential, and thus $f = je$ is an epimorphism from an object of finite length to itself. Hence $f$ is an isomorphism by \eqref{propLengthEpiMono}.
\end{proof}

\begin{definition}\label{defRelProjCover}
A relatively projective cover of $V \in \mathcal{A}$ is a relatively projective object $R_V$ together with a relatively essential allowable epimorphism $p_V : R_V \to V$.
\end{definition}

\begin{proposition}\label{propPropertiesRelProjCover}
1. Any object $V \in \mathcal{A}$ has a relatively projective cover, which is unique up to isomorphism.
\\2. Let $P$ be any relatively projective object with an allowable epimorphism $\pi : P \to V$. Then $R_V$ is a direct summand of $P$ and $\pi|_{R_V} = p_V$.
\\3. Assume that the categories $\mathcal{A}$, $\mathcal{B}$ and the functor $\mathcal{U}$ are $k$-linear. If $\mathrm{End}_{\mathcal{B}}\bigl(\mathcal{U}(V)\bigr) \cong k$ then $\Hom_{\mathcal{A}}(R_V,V) \cong k$, with basis element $p_V$.
\end{proposition}
\begin{proof}
1. For existence, let $P$ be a relatively projective object together with an allowable epimorphism $\pi : P \to V$. There always exists such a pair $(P,\pi)$; for instance one can take $P = G(V)$ and $\pi = \varepsilon_V$ where $G = \mathcal{F}\mathcal{U}$ is the comonad on $\mathcal{A}$ and $\varepsilon$ is the counit of $G$. Consider
\[ \mathbb{S} = \bigl\{ j : S \hookrightarrow P \text{ subobject} \, \big| \, \pi j \text{ is an allowable epimorphism} \bigr\} \]
and let $\iota : R \hookrightarrow P$ be an element of $\mathbb{S}$ with $R$ of minimal length.
\\We first show that $\pi \iota$ is minimal in the sense of \eqref{defMinimalEpi}. Let $g : R \to R$ be such that $\pi \iota g = \pi \iota$ and decompose $g$ as $R \overset{e}{\rightarrow} I \overset{j}{\rightarrow} R$, with $e$ epimorphism and $j$ monomorphism. By the same arguments as in the proof of Lemma \ref{lemmaMinimalEssential} we get that $\pi \iota j : I \to V$ is an allowable epimorphism and thus $(\iota j : I \hookrightarrow P) \in \mathbb{S}$. Since $I$ is a subobject of $R$ we have $\ell(I) \leq \ell(R)$ but the minimality assumption on $\ell(R)$ forces $\ell(I) = \ell(R)$. It follows that $j : I \hookrightarrow R$ is an isomorphism by \eqref{propLengthEpiMono}. Hence $g = je$ is an epimorphism and thus $g$ is an isomorphism by \eqref{propLengthEpiMono}, as desired.
\\Now we show that $R$ is relatively projective. Since $P$ is relatively projective there exists $f : P \to R$ such that $(\pi\iota) f = \pi$. Then $(\pi \iota)(f \iota) = \pi\iota$, which implies that $f\iota : R \to R$ is an isomorphism by the minimality of $\pi\iota$ proven just above. If we define $\mathrm{pr} = (f\iota)^{-1}f : P \to R$ we have $\mathrm{pr} \, \iota = \mathrm{id}_R$, which shows that $R$ is a direct summand of $P$ and hence that $R$ is relatively projective by Lemma \ref{lemmaBasicPropertiesRelProj}.
\\As a result $R_V = R$ together with $p_V = \pi \iota$ is a relatively projective cover by Lemma \ref{lemmaMinimalEssential}.
\\\indent For uniqueness, let $(R_V, p_V)$, $(R'_V, p'_V)$ be two relatively projective covers of $V$. We can fill the diagrams
\[ \xymatrix{
& R_V \ar@{-->}[ld]_g \ar[d]^{p_V} & \\
R'_V \ar[r]_{p'_V} & V \ar[r] & 0
} \qquad
\xymatrix{
& R'_V \ar@{-->}[ld]_h \ar[d]^{p'_V} & \\
R_V \ar[r]_{p_V} & V \ar[r] & 0
}\]
Note that $p_V hg = p'_V g = p_V$ and $p'_V gh = p_V h =p'_V$. Since $p_V$ and $p'_V$ are minimal by Lemma \ref{lemmaMinimalEssential}, $hg$ and $gh$ are isomorphisms. It follows that $g$ and $h$ are isomorphisms.
\\2. By the proof of item 1 we know that $P$ contains a direct summand $R$ such that $(R,\pi|_R)$ is a relatively projective cover. By uniqueness of the projective cover, $(R_V, p_V)$ is isomorphic to $(R,\pi|_R)$.
\\3. By the previous item $R_ V$ is a direct summand of $G(V)$, so we have
\begin{equation}\label{HomRelProjCover}
\Hom_{\mathcal{A}}(R_V,V) \hookrightarrow  \Hom_{\mathcal{A}}(G(V),V) = \Hom_{\mathcal{A}}(\mathcal{F}\mathcal{U}(V),V)
 \underset{(\star)}{\cong} \Hom_{\mathcal{B}}(\mathcal{U}(V),\mathcal{U}(V)) \cong k.
\end{equation}
The bijection $(\star)$ is $k$-linear, given by $f \mapsto \mathcal{U}(f)\eta_{\mathcal{U}(V)}$ where $\eta$ is the unit of the adjunction. As a result $\Hom_{\mathcal{A}}(R_V,V)$ is at most one-dimensional; but it contains $p_V$, so it is one-dimensional.
\end{proof}

\begin{remark}\label{remarkHowToComputeRelProjCover}
It follows from the proof of Proposition \ref{propPropertiesRelProjCover} that the relatively projective cover $R_V$ of $V$ is isomorphic to any of the direct summands of $G(V)$ which cover $V$ by the restriction of $\varepsilon_V$ and which have minimal length; in particular all such direct summands are isomorphic. If the assumptions in the third item of Proposition \ref{propPropertiesRelProjCover} hold then we see from \eqref{HomRelProjCover} that $R_V$ is the unique indecomposable direct summand of $G(V)$ on which $\varepsilon_V$ does not vanish.
\end{remark}

\subsection{Comonad cohomology and relative $\Ext$ groups}\label{sectionRelExtGroupsAsComonadCohomology}
Let $G = \mathcal{F}\mathcal{U}$ be the comonad on $\mathcal{A}$ associated to an adjunction $\mathcal{F} \dashv \mathcal{U}$ as in \eqref{adjunction}. In this section we review the theory of comonad cohomology introduced in \cite{BB} and in the case of a resolvent pair we relate it to relative Ext groups. Our notations and conventions are those of \cite{GHS}.
\begin{itemize}[itemsep=0em]
\item An object $P$ of $\mathcal{A}$ is called {\em $G$-projective} if there exists a morphism $s : P \to G(P)$ such that $\varepsilon_P \, s = \mathrm{id}_P$, where $\varepsilon$ is the counit of $G$. By \cite[Lem.\,2.5]{GHS}, this is equivalent to the requirement that $P$ is a direct summand of some $G(V)$.
\item A sequence $X \overset{f}{\longrightarrow} Y \overset{g}{\longrightarrow} Z$ in $\mathcal{A}$ is called {\em $G$-exact} if $gf = 0$ and the sequence of abelian groups
\[ \Hom_{\mathcal{A}}(G(V),X) \overset{f_*}{\longrightarrow} \Hom_{\mathcal{A}}(G(V),Y) \overset{g_*}{\longrightarrow} \Hom_{\mathcal{A}}(G(V),Z) \]
is exact for all $V \in \mathcal{A}$. Note that thanks to the adjunction property, this last condition is equivalent to the exactness of
\begin{equation}\label{GExact}
\Hom_{\mathcal{B}}\bigl(\mathcal{U}(V),\mathcal{U}(X)\bigr) \xrightarrow{\mathcal{U}(f)_*} \Hom_{\mathcal{B}}\bigl(\mathcal{U}(V),\mathcal{U}(Y)\bigr) \xrightarrow{\mathcal{U}(g)_*} \Hom_{\mathcal{B}}\bigl(\mathcal{U}(V),\mathcal{U}(Z)\bigr).
\end{equation}
for all $V \in \mathcal{A}$.
\item A sequence $0 \longleftarrow V \overset{d_0}{\longleftarrow} P_0 \overset{d_1}{\longleftarrow} P_1 \overset{d_2}{\longleftarrow} \ldots$ in $\mathcal{A}$ is called a {\em $G$-resolution} if each $P_i$ is $G$-projective and the sequence is $G$-exact.
\end{itemize}
There always exists at least one $G$-projective resolution of any object $V$ in $\mathcal{A}$, called the bar resolution:
\begin{equation}\label{barResolutionG}
\mathrm{Bar}^{\bullet}_G(V) = \left( 0 \longleftarrow V \overset{\varepsilon_V}{\longleftarrow} G(V) \overset{d_1}{\longleftarrow} G^2(V) \overset{d_2}{\longleftarrow} \ldots \right)
\end{equation}
where $\varepsilon_V : G(V) \to V$ is the counit of the comonad $G$ and
\begin{equation}\label{barDifferentialG}
d_n = \sum_{i=0}^n (-1)^iG^{n-i}\bigl(\varepsilon_{G^i(V)}\bigr).
\end{equation}

\begin{definition}\label{defComonadCohomology}
Let $\mathcal{E}$ be an abelian category and $E : \mathcal{A} \to \mathcal{E}$ be a contravariant additive functor. Given a $G$-projective resolution as above, the  cohomology of the complex in $\mathcal{E}$
\[ 0 \longrightarrow E(P_0) \xrightarrow{E(d_1)} E(P_1) \xrightarrow{E(d_2)} E(P_2) \xrightarrow{E(d_3)} \ldots \]
is called the cohomology of $V$ associated to $G$ with coefficients in $E$, and is denoted $H^{\bullet}_G(V,E)$.
\end{definition}
There is an analogue of the fundamental lemma of homological algebra for comonad cohomology which implies that $H^{\bullet}_G(V,E)$ does not depend on the choice of a $G$-projective resolution \cite[\S 4.2]{BB}. For the particular choice of the bar resolution \eqref{barResolutionG}, we denote the complex in Definition \ref{defComonadCohomology} by
\begin{equation}\label{barComplexCoefficients}
\mathrm{Bar}^{\bullet}_G(V,E) = \left( 0 \longrightarrow E(G(V)) \overset{E(d_1)}{\longrightarrow} E(G^2(V)) \overset{E(d_2)}{\longrightarrow} \ldots \right)
\end{equation}

\indent We now relate the notions from comonad cohomology to those from relative homological algebra. Suppose we are given a resolvent pair of categories as in \eqref{adjunction} and let $G = \mathcal{F}\mathcal{U}$ be the associated comonad on $\mathcal{A}$.
\begin{lemma}
Let $X \overset{f}{\longrightarrow} Y \overset{g}{\longrightarrow} Z$ be a sequence in $\mathcal{A}$.
\begin{enumerate}[itemsep=0em, topsep=.3em]
\item Let $\ker(g)=(k : K \to Y)$. The sequence is $G$-exact if and only if $gf = 0$ and there exists $t \in \Hom_{\mathcal{B}}\bigl(\mathcal{U}(K),\mathcal{U}(X)\bigr)$ such that $\mathcal{U}(f)t = \mathcal{U}(k)$.
\item If the sequence is $G$-exact and $g$ is allowable, then $f$ is allowable.
\item If the sequence is exact and $f$ is allowable, then it is $G$-exact.
\item If the sequence is $G$-exact, then it is exact.
\end{enumerate}
\end{lemma}
\begin{proof}
Note that $\ker(\mathcal{U}(g)) = \mathcal{U}(\ker(g))$ and $\mathrm{im}(\mathcal{U}(f)) = \mathcal{U}(\mathrm{im}(f))$ since  $\mathcal{U}$ is exact.
\\1. Since $gk = 0$, the exactness of \eqref{GExact} with $V=K$ implies that there exists $t : \mathcal{U}(K) \to \mathcal{U}(X)$ such that $\mathcal{U}(f)t = \mathcal{U}(k)$. Conversely, let $V \in \mathcal{A}$ and $h : \mathcal{U}(V) \to \mathcal{U}(Y)$ such that $\mathcal{U}(g)h = 0$. Thus by the universal property of the kernel, there exists $u : \mathcal{U}(V) \to \mathcal{U}(K)$ such that $h = \mathcal{U}(k)u = \mathcal{U}(f)tu$, which shows that $h \in \mathrm{im}(\mathcal{U}(f)_*)$. Hence the sequence \eqref{GExact} is exact.
\\2. By definition, there exists $s : \mathcal{U}(Z) \to \mathcal{U}(Y)$ such that $\mathcal{U}(g)s\mathcal{U}(g) = \mathcal{U}(g)$. Let $\pi = \mathrm{id}_{\mathcal{U}(Y)} - s\,\mathcal{U}(g) : \mathcal{U}(Y) \to \mathcal{U}(Y)$. We have $\mathcal{U}(g)\pi = 0$, so by the universal property of the kernel we get $u : \mathcal{U}(Y) \to \mathcal{U}(K)$ such that $\mathcal{U}(k)u = \pi$. Now since the sequence is $G$-exact, we have $t : \mathcal{U}(K) \to \mathcal{U}(X)$ from item 1. Then $tu : \mathcal{U}(Y) \to \mathcal{U}(X)$ satisfies
\[ \mathcal{U}(f)tu\,\mathcal{U}(f) = \mathcal{U}(k)u\,\mathcal{U}(f) = \pi\,\mathcal{U}(f) = \mathcal{U}(f) - s\,\mathcal{U}(g)\,\mathcal{U}(f) = \mathcal{U}(f) \]
where for the last equality we used that $gf=0$. Hence $f$ is allowable.
\\3. Since $f$ is allowable, we have some $s : \mathcal{U}(Y) \to \mathcal{U}(X)$ such that $\mathcal{U}(f)s\,\mathcal{U}(f) = \mathcal{U}(f)$. Since the sequence is exact, $\ker(g) = (k : K \to Y) = \mathrm{im}(f)$, and by definition of the image there exists an epimorphism $e : X \to K$ such that $f = ke$. We have
\[ \mathcal{U}(f)\bigl(s\,\mathcal{U}(k)\bigr)\mathcal{U}(e) = \mathcal{U}(f)s\,\mathcal{U}(f) = \mathcal{U}(f) = \mathcal{U}(k)\,\mathcal{U}(e). \]
Since $e$ is an epimorphism, we deduce $\mathcal{U}(f)\bigl(s\,\mathcal{U}(k)\bigr) = \mathcal{U}(k)$ and by item 1 the sequence is $G$-exact.
\\4. We show that $\ker(\mathcal{U}(g)) = \mathrm{im}(\mathcal{U}(f))$, then it follows that $\ker(g) = \mathrm{im}(f)$ since $\mathcal{U}$ is faithful. First, $\mathcal{U}(k)$ is a monomorphism and by definition of the kernel, since $\mathcal{U}(f)\,\mathcal{U}(g) = 0$, there exists $u : \mathcal{U}(X) \to \mathcal{U}(K)$ such that $\mathcal{U}(f) = \mathcal{U}(k)\,\mathcal{U}(u)$. It remains to prove that $\mathcal{U}(k)$ is universal for these properties. So let $m : I \to \mathcal{U}(Y)$ be a monomorphism such that there exists $e : \mathcal{U}(X) \to I$ with $\mathcal{U}(f) = me$. We want to show that there exists $v : \mathcal{U}(K) \to I$ such that $\mathcal{U}(k) =mv$. Using item 1, we have $\mathcal{U}(k) = \mathcal{U}(f)t = met$, so that we can take $v=et$.
\end{proof}
\begin{proposition}\label{relExtAndComonadCohom}
Let $\xymatrix{\mathcal{A} \ar@<-.5ex>[r]_{\mathcal{U}} &  \ar@<-.5ex>[l]_{\mathcal{F}}  \mathcal{B}}$ be a resolvent pair of abelian categories and $G = \mathcal{F}\mathcal{U}$ be the associated comonad on $\mathcal{A}$.
\begin{enumerate}[itemsep=0em, topsep=.3em]
\item An object $P \in \mathcal{A}$ is $G$-projective if and only if it is relatively projective.
\item A sequence is a $G$-resolution if and only if it is a relatively projective resolution.
\item We have an equality of complexes $\mathrm{Bar}^{\bullet}_G\bigl(V,\Hom_{\mathcal{A}}(?,W)\bigr) = \mathrm{Bar}^{\bullet}_{\mathcal{A},\mathcal{B}}(V,W)$, for any $V,W \in \mathcal{A}$.
\item This implies $\Ext^n_{\mathcal{A},\mathcal{B}}(V,W) \cong H^n_G\bigl( V, \Hom_{\mathcal{A}}(?,W) \bigr)$ for all $n \geq 0$.
\end{enumerate}
\end{proposition}
\begin{proof}
1. If $P$ is $G$-projective then it is a direct summand of some $G(V) = \mathcal{F}(\mathcal{U}(V))$ and hence is relatively projective by Lemma \ref{lemmaBasicPropertiesRelProj}. Conversely, assume that $P$ is a direct summand of some $\mathcal{F}(X)$, with morphisms $\iota : P \to \mathcal{F}(X)$ and $\pi : \mathcal{F}(X) \to P$ such that $\pi\iota = \mathrm{id}_P$. Let $\eta : \mathrm{Id}_{\mathcal{B}} \to \mathcal{U}\mathcal{F}$ and $\varepsilon : \mathcal{F}\mathcal{U} \to \mathrm{Id}_{\mathcal{A}}$ be respectively the unit and the counit of the adjunction. Define
\begin{align*}
&\iota' : P \overset{\iota}{\longrightarrow} \mathcal{F}(X) \xrightarrow{\mathcal{F}(\eta_X)} \mathcal{F}\mathcal{U}\mathcal{F}(X) = G(\mathcal{F}(X)),\\
&\pi' : G(\mathcal{F}(X)) = \mathcal{F}\mathcal{U}\mathcal{F}(X) \xrightarrow{\varepsilon_{\mathcal{F}(X)}} \mathcal{F}(X) \overset{\pi}{\longrightarrow} P.
\end{align*}
Then $\pi'\iota' = \pi \varepsilon_{\mathcal{F}(X)} \mathcal{F}(\eta_X) \iota = \pi \iota = \mathrm{id}_P$, thanks to one of the defining properties of $\varepsilon$ and $\eta$. Hence $P$ is $G$-projective, as a direct summand of $G(\mathcal{F}(X))$.
\\2. Let $\ldots \overset{d_2}{\longrightarrow} P_1 \overset{d_1}{\longrightarrow} P_0 \overset{d_0}{\longrightarrow} V \longrightarrow 0$ be a $G$-projective resolution. Since $P_0 \overset{d_0}{\longrightarrow} V \longrightarrow 0$ is $G$-exact and $V \overset{0}{\longrightarrow} 0$ is allowable, so is $d_0$ by the previous lemma. Assume now that some $d_i$ is allowable. Since $P_{i+1} \overset{d_{i+1}}{\longrightarrow} P_i \overset{d_i}{\longrightarrow} P_{i-1}$ is $G$-exact, $d_{i+1}$ is allowable by the previous lemma. By induction, all the $d_i$'s are allowable. Since moreover we have seen that a $G$-exact sequence is exact, we have a relatively projective resolution.
\\The converse implication is obvious due to the third item in the previous lemma.
\\3. Trivial, the definitions are the same.
\\4. Follows immediately from any of the two previous items.
\end{proof}

\section{Relative $\Ext$ groups for tensor categories}\label{relativeExtTensorCategories}
\indent Let $k$ be a field. Recall that:
\begin{itemize}[itemsep=0em, topsep=.3em]
\item A {\em tensor category} is a $k$-linear abelian rigid monoidal category with $k$-bilinear tensor product, such that the monoidal unit $\boldsymbol{1}$ is a simple object and $\mathrm{End}_{\mathcal{C}}(\boldsymbol{1}) \cong k$.
\item A $k$-linear abelian category is {\em finite} if it is equivalent as a $k$-linear category to the category of finite-dimensional modules over some finite-dimensional $k$-algebra \cite[Def.\,1.8.5]{EGNO}.
\item A {\em finite tensor category} is a tensor category which is finite as a $k$-linear abelian category.
\end{itemize}
\noindent In a tensor category the monoidal product is exact, because of rigidity \cite[Prop.\,4.2.1]{EGNO}. We will work with strict tensor categories.

\smallskip

\indent Recall that a monoidal category $\mathcal{C}$ is rigid if every object $X \in \mathcal{C}$ has a right dual $X^{\vee}$ and a left dual $^{\vee}\!X$ together with left and right (co)evaluation morphisms
\[ \begin{array}{ll}
\mathrm{ev}_X : X^{\vee} \otimes X \to \boldsymbol{1}, & \mathrm{coev}_X : \boldsymbol{1} \to X \otimes X^{\vee},\\[3pt]
\widetilde{\mathrm{ev}}_X : X \otimes {^{\vee}\!X} \to \boldsymbol{1}, & \widetilde{\mathrm{coev}}_X : \boldsymbol{1} \to {^{\vee}\!X} \otimes X,
\end{array} \]
satisfying the standard axioms. For certain computations it will be convenient to represent morphisms with diagrams, following the usual rules:
\begin{center}
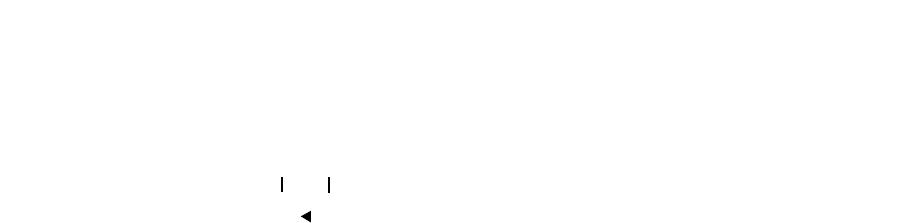
\end{center}
Note that we read diagrams from bottom to top. Let $(\mathcal{M}, \triangleright)$ be a strict left $\mathcal{C}$-module category; we extend to $\mathcal{M}$ the diagrammatic representation of morphisms as follows:
\begin{center}
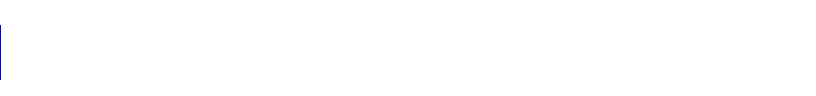
\end{center}
The objects and morphisms in $\mathcal{C}$ (resp. in $\mathcal{M}$) are in black (resp. in blue); the objects and morphisms in $\mathcal{M}$ are necessarily at the right of the diagrams.

\subsection{Monoidal resolvent pairs}\label{sectionMonoidalResolventPairs}
In this subsection, we consider a resolvent pair of categories as in \eqref{adjunction} but moreover we assume that $\mathcal{A},\mathcal{B}$ are tensor categories and that the exact, faithful and linear functor $\mathcal{U} : \mathcal{A} \to \mathcal{B}$ is monoidal.

\begin{lemma}\label{fTensgAllowable}
If $f \in \Hom_{\mathcal{A}}(X,X')$ and $g \in \Hom_{\mathcal{A}}(Y,Y')$ are allowable morphisms then $f \otimes g$ is an allowable morphism.
\end{lemma}
\begin{proof}
Let $s \in \Hom_{\mathcal{B}}(X',X), t \in \Hom_{\mathcal{B}}(Y',Y)$ be such that $\mathcal{U}(f)s\,\mathcal{U}(f) = \mathcal{U}(f)$ and $\mathcal{U}(g)\, t\,\mathcal{U}(g) = \mathcal{U}(g)$ and let $\mathcal{U}^{(2)}_{V,W} : \mathcal{U}(V) \otimes \mathcal{U}(W) \overset{\sim}{\to} \mathcal{U}(V \otimes W)$ be the monoidal structure of $\mathcal{U}$. Then
\begin{align*}
\mathcal{U}(f \otimes g)\left( \mathcal{U}^{(2)}_{X,Y} (s \otimes t) \, (\mathcal{U}^{(2)}_{X',Y'})^{-1} \right) \mathcal{U}(f \otimes g) &= \mathcal{U}^{(2)}_{X',Y'} \bigl(\mathcal{U}(f) \otimes \mathcal{U}(g)\bigr) (s \otimes t) \bigl(\mathcal{U}(f) \otimes \mathcal{U}(g)\bigr) (\mathcal{U}^{(2)}_{X,Y})^{-1}\\
&= \mathcal{U}^{(2)}_{X',Y'} \bigl(\mathcal{U}(f) \otimes \mathcal{U}(g)\bigr) (\mathcal{U}^{(2)}_{X,Y})^{-1} = \mathcal{U}(f \otimes g).\qedhere
\end{align*}
\end{proof}

\begin{proposition}\label{relProjTensorIdeal}
The full subcategory $\mathrm{Proj}_{\mathcal{A},\mathcal{B}}$ of relatively projective objects is a tensor ideal in $\mathcal{A}$, which is stable under the duality functors $P \mapsto P^{\vee}$ and $P \mapsto {^{\vee}\!P}$.
\end{proposition}
\begin{proof}
The proofs are straightforward adaptations of the corresponding facts for usual projective objects \cite[Props.\,4.2.12 and 6.1.3]{EGNO}. Indeed, let $P$ be a relatively projective object and $V$ any object in $\mathcal{A}$, we want to show that $V \otimes P$ and $P \otimes V$ are relatively projective. So take $0 \to X \overset{j}{\to} Y \overset{\sigma}{\to} Z \to 0$ an allowable short exact sequence in $\mathcal{A}$. By exactness of $\otimes$ \cite[Prop.\,4.2.1]{EGNO}, the sequence
\[ 0 \longrightarrow X \otimes V^{\vee} \xrightarrow{j \otimes \mathrm{id}} Y \otimes V^{\vee} \xrightarrow{\sigma \otimes \mathrm{id}} Z \otimes V^{\vee} \longrightarrow 0 \]
is exact and it is allowable due to Lemma \ref{fTensgAllowable}. Hence by Lemma \ref{relProjAndHom} the sequence
\[ 0 \longrightarrow \Hom_{\mathcal{A}}(P, X \otimes V^{\vee}) \xrightarrow{(j \otimes \mathrm{id})_*} \Hom_{\mathcal{A}}(P, Y \otimes V^{\vee}) \xrightarrow{(\sigma \otimes \mathrm{id})_*} \Hom_{\mathcal{A}}(P,Z \otimes V^{\vee}) \longrightarrow 0 \]
is exact. Due to the natural isomorphism $\Hom_{\mathcal{A}}(- \otimes V, -) \cong \Hom_{\mathcal{A}}(-, - \otimes V^{\vee})$ \cite[Prop.\,2.10.8]{EGNO} it follows that
\[ 0 \longrightarrow \Hom_{\mathcal{A}}(P \otimes V, X) \overset{j_*}{\longrightarrow} \Hom_{\mathcal{A}}(P \otimes V, Y) \overset{\sigma_*}{\longrightarrow} \Hom_{\mathcal{A}}(P \otimes V, Z) \longrightarrow 0 \]
is exact. Thus by Lemma \ref{relProjAndHom}, $P \otimes V$ is relatively projective. The proof for $V \otimes P$ is similar. 
\\We now prove that $P^{\vee}$ is relatively projective. Take again $0 \to X \overset{j}{\to} Y \overset{\sigma}{\to} Z \to 0$ an allowable short exact sequence in $\mathcal{A}$. By exactness of $P \otimes -$ and left-exactness of $\Hom_{\mathcal{A}}(\boldsymbol{1}, -)$ the sequence
\begin{equation}\label{leftExactHomProofDual}
0 \longrightarrow \Hom_{\mathcal{A}}(\boldsymbol{1}, P \otimes X) \xrightarrow{( \mathrm{id} \otimes j)_*} \Hom_{\mathcal{A}}(\boldsymbol{1}, P \otimes Y) \xrightarrow{(\mathrm{id} \otimes \sigma)_*} \Hom_{\mathcal{A}}(\boldsymbol{1}, P \otimes Z)
\end{equation}
is exact. Since we have already seen that $P \otimes Z$ is relatively projective and $\mathrm{id} \otimes \sigma$ is an allowable epimorphism, there exists $s \in \Hom_{\mathcal{A}}\bigl(P \otimes Z, P \otimes Y\bigr)$ such that $(\mathrm{id} \otimes \sigma)\, s = \mathrm{id}$. Hence $(\mathrm{id} \otimes \sigma)_*$ is surjective and thus \eqref{leftExactHomProofDual} can be extended to a short exact sequence. The natural isomorphism $\Hom_{\mathcal{A}}(- \otimes P^{\vee}, -) \cong \Hom_{\mathcal{A}}(-, P \otimes -)$ implies that
\[ 0 \longrightarrow \Hom_{\mathcal{A}}(P^{\vee}, X) \overset{j_*}{\longrightarrow} \Hom_{\mathcal{A}}(P^{\vee}, Y) \overset{\sigma_*}{\longrightarrow} \Hom_{\mathcal{A}}(P^{\vee}, Z) \longrightarrow 0 \]
is exact as well and the result follows from Lemma \ref{relProjAndHom}. The proof for $^{\vee}\!P$ is similar.

\end{proof}

\begin{corollary}\label{switchFormulaCoeffsRelExt}
It holds 
\[ \Ext^n_{\mathcal{A},\mathcal{B}}(X, Y) \cong \Ext^n_{\mathcal{A},\mathcal{B}}(X \otimes {^{\vee}Y}, \boldsymbol{1}), \quad \Ext^n_{\mathcal{A},\mathcal{B}}(X, Y) \cong \Ext^n_{\mathcal{A},\mathcal{B}}(\boldsymbol{1}, Y \otimes X^{\vee}) \]
where $\boldsymbol{1}$ is the tensor unit of $\mathcal{A}$.
\end{corollary}
\begin{proof}
We show the first isomorphism, the proof of the second being similar. Let $0 \longleftarrow X \overset{d_0}{\longleftarrow} P_0 \overset{d_1}{\longleftarrow} P_1 \overset{d_2}{\longleftarrow} \ldots$ be a relatively projective resolution of $X$. By exactness of $\otimes$, the sequence
\[ 0 \longleftarrow X \otimes {^{\vee}Y} \xleftarrow{d_0 \otimes \mathrm{id}} P_0 \otimes {^{\vee}Y} \xleftarrow{d_1 \otimes \mathrm{id}} P_1\otimes {^{\vee}Y} \xleftarrow{d_2 \otimes \mathrm{id}} \ldots \]
is exact. By the previous proposition, each $P_i \otimes {^{\vee}Y}$ is relatively projective and by Lemma \ref{fTensgAllowable}, $d_i \otimes \mathrm{id}$ is allowable. Hence it is a relatively projective resolution of $X \otimes {^{\vee}Y}$. The result follows from the natural isomorphism $\Hom_{\mathcal{A}}(X, -) \cong \Hom_{\mathcal{A}}(X \otimes {^{\vee}(-)},\mathbf{1})$.
\end{proof}
\noindent Note that thanks to the second isomorphism, it is enough to know a relatively projective resolution of the unit object $\boldsymbol{1}$ to compute any $\Ext$ group.

\smallskip

\indent In general it is difficult to determine a relatively projective resolution which is simple enough to make the computation of (the dimension of) the Ext groups feasible. The following proposition, which is a refinement of Corollary \ref{corollaryInductionFormulaExt}, replaces this problem by the computation of  the first step of a relatively projective resolution and of certain Hom spaces:
\begin{proposition}\label{propositionExtWithRelProjCovers}
Let $V \in \mathcal{A}$ and let 
\[ 0 \longrightarrow K \overset{j}{\longrightarrow} P \overset{\pi}{\longrightarrow} \mathbf{1} \longrightarrow 0, \qquad 0 \longrightarrow L \overset{i}{\longrightarrow} Q \overset{p}{\longrightarrow} V \longrightarrow 0 \]
be allowable short exact sequences in $\mathcal{A}$ where $P, Q$ are relatively projective objects. Then for all $n \geq 1$,
\[ \Ext^n_{\mathcal{A},\mathcal{B}}(V,W) \cong \Hom_{\mathcal{A}}\bigl( K, M_n \bigr)\big/\mathrm{im}(j^*) \]
where $M_n = W \otimes L^{\vee} \otimes (K^{\vee})^{\otimes (n-2)}$ for $n \geq 2$, $M_1 = W \otimes V^{\vee}$ and $j^*$ is the pullback $\Hom_{\mathcal{A}}\bigl( P, M_n \bigr) \to \Hom_{\mathcal{A}}\bigl( K, M_n \bigr)$.
\end{proposition}
\begin{proof}
For $n \geq 2$ we use Corollary \ref{corollaryInductionFormulaExt} and Corollary \ref{switchFormulaCoeffsRelExt} several times:
\begin{align*}
\Ext^n_{\mathcal{A},\mathcal{B}}(V,W) &\cong \Ext^{n-1}_{\mathcal{A},\mathcal{B}}(L,W) \cong \Ext^{n-1}_{\mathcal{A},\mathcal{B}}(\mathbf{1}, W \otimes L^{\vee})\\
&\cong \Ext^{n-2}_{\mathcal{A},\mathcal{B}}(K, W \otimes L^{\vee}) \cong \Ext^{n-2}_{\mathcal{A},\mathcal{B}}(\mathbf{1}, W \otimes L^{\vee} \otimes K^{\vee})\\
& \cong \ldots\\
&\cong \Ext^1_{\mathcal{A},\mathcal{B}}\bigl(K, W \otimes L^{\vee} \otimes (K^{\vee})^{\otimes (n-3)}\bigr) \cong \Ext^1_{\mathcal{A},\mathcal{B}}\bigl(\mathbf{1}, W \otimes L^{\vee} \otimes (K^{\vee})^{\otimes (n-2)}\bigr)\\
&\cong \Hom_{\mathcal{A}}\bigl( K, W \otimes L^{\vee} \otimes (K^{\vee})^{\otimes (n-2)} \bigr)\big/\mathrm{im}(j^*).
\end{align*}
For $n=1$ we simply note that $\Ext^1_{\mathcal{A},\mathcal{B}}(V,W) \cong \Ext^1_{\mathcal{A},\mathcal{B}}(\boldsymbol{1}, W \otimes V^{\vee})$ and we use Corollary \ref{corollaryInductionFormulaExt}.
\end{proof}

\begin{corollary}\label{CoroDimExtWithHom}
With the notations of the previous proposition we have for all $n \geq 1$
\[ \dim \Ext^n_{\mathcal{A},\mathcal{B}}(V,W) = \dim \Hom_{\mathcal{A}}(K,M_n) - \dim \Hom_{\mathcal{A}}(P,M_n) + \dim \Hom_{\mathcal{A}}(\mathbf{1},M_n) \]
In particular
\begin{align*}
\dim \Ext^n_{\mathcal{A},\mathcal{B}}(\mathbf{1}, \mathbf{1}) = \dim \Hom_{\mathcal{A}}\bigl(K,(K^{\vee})^{\otimes (n-1)}\bigr) &- \dim \Hom_{\mathcal{A}}\bigl(P,(K^{\vee}\bigr)^{\otimes (n-1)}\bigr) 
\\&+ \dim \Hom_{\mathcal{A}}\bigl(\mathbf{1},(K^{\vee}\bigr)^{\otimes (n-1)}\bigr)
\end{align*}
with the convention $(K^{\vee})^{\otimes \,0} = \boldsymbol{1}$.
\end{corollary}
\begin{proof}
Thanks to the exactness of
\[ 0 \longrightarrow \Hom_{\mathcal{A}}(\mathbf{1}, M_n) \overset{\pi^*}{\longrightarrow} \Hom_{\mathcal{A}}(P, M_n) \overset{j^*}{\longrightarrow} \Hom_{\mathcal{A}}(K, M_n) \]
we have $\dim \mathrm{im}(j^*) = \dim \Hom_{\mathcal{A}}(P,M_n) - \dim \Hom_{\mathcal{A}}(\mathbf{1},M_n)$ and then the isomorphism in the previous proposition gives the result for $n\geq 2$.
\end{proof}

\noindent Assume that $\mathcal{A}$ is finite as a $k$-linear abelian category, in particular every object in $\mathcal{A}$ has finite length. Then the relatively projective cover $R_{\mathbf{1}}$ of $\mathbf{1}$ exists by Proposition \ref{propPropertiesRelProjCover} and the second formula of Corollary \ref{CoroDimExtWithHom} admits a special case when $P = R_{\mathbf{1}}$:
\begin{corollary}\label{DimExt2WithHom}
Let $K_{\mathbf{1}} = \ker(p_{\mathbf{1}} : R_{\mathbf{1}} \twoheadrightarrow \mathbf{1})$. If $\mathcal{A}$ is finite and $k$ (the ground field of $\mathcal{A}$) has characteristic $0$ and is algebraically closed, we have
\[ \dim \Ext^2_{\mathcal{A},\mathcal{B}}(\mathbf{1}, \mathbf{1}) = \dim \Hom_{\mathcal{A}}(K_{\mathbf{1}},K_{\mathbf{1}}^{\vee}) - \dim \Hom_{\mathcal{A}}(R_{\mathbf{1}},K_{\mathbf{1}}^{\vee}). \]
\end{corollary}
\begin{proof}
By definition there is the short exact sequence $0 \longrightarrow K_{\mathbf{1}} \longrightarrow R_{\mathbf{1}} \longrightarrow \mathbf{1} \longrightarrow 0$. Hence Corollary \ref{CoroDimExtWithHom} gives
\[ \dim \Ext^2_{\mathcal{A},\mathcal{B}}(\mathbf{1}, \mathbf{1}) = \dim \Hom_{\mathcal{A}}(K_{\mathbf{1}},K_{\mathbf{1}}^{\vee}) - \dim \Hom_{\mathcal{A}}(R_{\mathbf{1}},K_{\mathbf{1}}^{\vee}) + \dim \Hom_{\mathcal{A}}(\mathbf{1},K_{\mathbf{1}}^{\vee}). \]
We want to show that the last term is equal to $0$. Using Corollary \ref{CoroDimExtWithHom} we find
\[ \dim \Hom_{\mathcal{A}}(\boldsymbol{1}, K_{\boldsymbol{1}}^{\vee}) = \dim \Hom_{\mathcal{A}}(K_{\boldsymbol{1}}, \boldsymbol{1}) = \dim \Ext_{\mathcal{A},\mathcal{B}}^1(\boldsymbol{1}, \boldsymbol{1}) + \dim \Hom_{\mathcal{A}}(R_{\boldsymbol{1}}, \boldsymbol{1}) - \dim \Hom_{\mathcal{A}}(\boldsymbol{1}, \boldsymbol{1}) \]
We have $\mathrm{End}_{\mathcal{A}}(\mathbf{1}) \cong k$ by our general assumptions made at the beginning of \S\ref{relativeExtTensorCategories}. Since the functor $\mathcal{U}$ is monoidal we also have $\mathrm{End}_{\mathcal{B}}\bigl(\mathcal{U}(\mathbf{1})\bigr) = \mathrm{End}_{\mathcal{B}}(\mathbf{1}) \cong k$ and thus $\Hom_{\mathcal{A}}(R_{\mathbf{1}},\mathbf{1}) \cong k$ by Proposition \ref{propPropertiesRelProjCover}. Finally $\Ext^1_{\mathcal{A},\mathcal{B}}(\mathbf{1}, \mathbf{1}) \subset \Ext^1_{\mathcal{A}}(\mathbf{1}, \mathbf{1}) = 0$, where we used \eqref{injectionRelExt1FullExt1} and the fact that in a finite tensor category over an algebraically closed field of characteristic $0$ the unit object does not have non-trivial self-extensions \cite[Thm.\,4.4.1]{EGNO}. As a result $\dim \Hom_{\mathcal{A}}(\boldsymbol{1}, K_{\boldsymbol{1}}^{\vee}) = 0$.
\end{proof}

\subsection{Tensor product of algebras in braided tensor categories}\label{sectionAlgebrasInBraidedTensorCategories}
\indent In this section we generalize Example \ref{exampleAtensBRelBInVect} to algebras in a braided finite tensor category $\mathcal{C}$. We assume that $\mathcal{C}$ is strict and we denote its braiding by $c$:
\[ c_{X,Y} : X \otimes Y \overset{\sim}{\to} Y \otimes X.  \]
Recall that a (unital) algebra $(A,m_A,\eta_A)$ in $\mathcal{C}$ is an object $A \in \mathcal{C}$ together with morphisms $m_A \in \Hom_{\mathcal{C}}(A \otimes A, A)$ and $\eta_A \in \Hom_{\mathcal{C}}(\mathbf{1}, A)$ satisfying the usual axioms. Similarly, a (left) $A$-module $\mathbb{V} = (V,\rho^A_V)$ in $\mathcal{C}$ is an object $V \in \mathcal{C}$ together with a morphism $\rho^A_V \in \Hom_{\mathcal{C}}(A \otimes V, V)$ satisfying the usual axioms. A morphism of $A$-modules is defined in the obvious way and we get the category $A\text{-}\mathrm{mod}_{\mathcal{C}}$ of left $A$-modules in $\mathcal{C}$. Note in particular that $(A,m_A) \in A\text{-}\mathrm{mod}_{\mathcal{C}}$. One can define similarly a right $A$-module and then has the category $\mathrm{mod}_{\mathcal{C}}\text{-}A$ of right $A$-modules in $\mathcal{C}$.

\smallskip

\indent If $(B,m_B,\eta_B)$ is another algebra in $\mathcal{C}$, then $A \otimes B$ is again an algebra in $\mathcal{C}$ thanks to the braiding, with
\[ m_{A \otimes B} =  (m_A \otimes m_B)(\mathrm{id}_A \otimes c_{B,A} \otimes \mathrm{id}_B), \quad \eta_{A \otimes B} = \eta_A \otimes \eta_B. \]

\smallskip We have the forgetful functor $\mathcal{U} : (A \otimes B)\text{-}\mathrm{mod}_{\mathcal{C}} \to B\text{-}\mathrm{mod}_{\mathcal{C}}$ induced by the pullback along the morphism $\eta_A \otimes \mathrm{id}_B \in \mathrm{Hom}_{\mathcal{C}}(B, A \otimes B)$, namely $\mathcal{U}(\mathbb{V}) = \mathcal{U}(V, \rho^{A \otimes B}_V) = \big(V, \rho^{A \otimes B}_V(\eta_A \otimes \mathrm{id}_{B \otimes V})\big)$ on objects and $\mathcal{U}(f) = f$ on morphisms. It is clear that $\mathcal{U}$ is exact and faithful; moreover, it has a left adjoint $\mathcal{F}$ given by
\[ \mathcal{F}(\mathbb{W}) = (A, m_A) \boxtimes_{\mathcal{C}} \mathbb{W}, \qquad \mathcal{F}(f) = \mathrm{id}_A \boxtimes_{\mathcal{C}} f \] 
for $\mathbb{W} \in B\text{-}\mathrm{mod}_{\mathcal{C}}$, where the bifunctor $\boxtimes_{\mathcal{C}} : A\text{-}\mathrm{mod}_{\mathcal{C}} \times B\text{-}\mathrm{mod}_{\mathcal{C}} \to (A \otimes B)\text{-}\mathrm{mod}_{\mathcal{C}}$ is defined on objects by
\[ (V,\rho^A_V) \boxtimes_{\mathcal{C}} (W,\rho^B_W) = (V \otimes W, \rho^{A \otimes B}_{V \otimes W}) \]
with
\begin{equation}\label{defActionTensorProduct}
\rho^{A \otimes B}_{V \otimes W} : A \otimes B \otimes V \otimes W \xrightarrow{\mathrm{id} \otimes c_{B,V} \otimes \mathrm{id}} A \otimes V \otimes B \otimes W \xrightarrow{\rho^A_V \otimes \rho^B_W} V \otimes W
\end{equation}
and on morphisms by $f \boxtimes_{\mathcal{C}} g = f \otimes g$ (it is straightforward to check that $f \otimes g$ commutes with the action). One checks easily that the adjunction property holds:
\[ \begin{array}{rcl}
\Hom_{(A \otimes B)\text{-}\mathrm{mod}_{\mathcal{C}}}(A \boxtimes_{\mathcal{C}} \mathbb{W}, \mathbb{M}) & \overset{\sim}{\to} & \Hom_{B\text{-}\mathrm{mod}_{\mathcal{C}}}(\mathbb{W}, \mathcal{U}(\mathbb{M}))\\
f & \mapsto & f(\eta_A \otimes \mathrm{id}_V)\\
\rho^{A \otimes B}_M(\mathrm{id}_A \otimes \eta_B \otimes g) & \rotatebox[origin=c]{180}{$\mapsto$} & g
\end{array} \]

\begin{remark}
Let us endow the category $A\text{-}\mathrm{mod}_{\mathcal{C}}$ with the usual structure of right $\mathcal{C}$-module category given by
\begin{equation*}
(V, \rho^A_V) \triangleleft X = (V \otimes X, \rho^A_V \otimes \mathrm{id}_X)
\end{equation*}
(where $X \in \mathcal{C}$) and the category $B\text{-}\mathrm{mod}_{\mathcal{C}}$ with a structure of left $\mathcal{C}$-module category by
\begin{equation}\label{defModuleStructureBmod}
X \triangleright (W, \rho^B_W) = \bigl(X \otimes W, (\mathrm{id}_X \otimes \rho^B_W)(c_{B,X} \otimes \mathrm{id}_W) \bigr).
\end{equation}
On morphisms, the functors $\triangleright$ and $\triangleleft$ are just the tensor product of morphisms in $\mathcal{C}$. This allows us to consider the relative Deligne product of the $\mathcal{C}$-module categories $A\text{-}\mathrm{mod}_{\mathcal{C}}$ and $B\text{-}\mathrm{mod}_{\mathcal{C}}$, denoted by $A\text{-}\mathrm{mod}_{\mathcal{C}} \boxtimes_{\mathcal{C}} B\text{-}\mathrm{mod}_{\mathcal{C}}$ (see e.g.\,\cite[Def.\,3.2]{DSPS}). In fact, the category $(A \otimes B)\text{-}\mathrm{mod}_{\mathcal{C}}$ is a realization of this relative Deligne product:
\[ (A \otimes B)\text{-}\mathrm{mod}_{\mathcal{C}} \cong A\text{-}\mathrm{mod}_{\mathcal{C}} \boxtimes_{\mathcal{C}} B\text{-}\mathrm{mod}_{\mathcal{C}}. \] 
This explains the notation $\boxtimes_{\mathcal{C}}$ for the bifunctor introduced above. To see this, let $B^{\mathrm{op}} = (B, m_B \, c_{B,B}^{-1}, \eta_B)$ be the algebra $B$ with opposite multiplication and endow the category $\mathrm{mod}_{\mathcal{C}}\text{-}B^{\mathrm{op}}$ of right $B^{\mathrm{op}}$-modules in $\mathcal{C}$ with the following structure of $\mathcal{C}$-module category:
\[ X \triangleright (V, \rho^{B^{\mathrm{op}}}_V) = (X \otimes V, \mathrm{id}_X \otimes \rho^{B^{\mathrm{op}}}_V). \]
Thanks to the braiding in $\mathcal{C}$ we can construct an isomorphism $B\text{-}\mathrm{mod}_{\mathcal{C}} \cong \mathrm{mod}_{\mathcal{C}}\text{-}B^{\mathrm{op}}$ of left $\mathcal{C}$-module-categories. Then we have
\[ (A \otimes B)\text{-}\mathrm{mod}_{\mathcal{C}} \cong A\text{-}\mathrm{mod}_{\mathcal{C}}\text{-}B^{\mathrm{op}} \cong A\text{-}\mathrm{mod}_{\mathcal{C}} \boxtimes_{\mathcal{C}} \mathrm{mod}_{\mathcal{C}}\text{-}B^{\mathrm{op}} \cong A\text{-}\mathrm{mod}_{\mathcal{C}} \boxtimes_{\mathcal{C}} B\text{-}\mathrm{mod}_{\mathcal{C}}. \]
The first equivalence simply uses the braiding to identify a left $(A \otimes B)$-module with a $(A, B^{\mathrm{op}})$-bimodule, while the second is \cite[Thm.\,3.3]{DSPS}.
\end{remark}

\indent We will see that the following resolvent pairs of categories are related:
\[ \xymatrix@R=.7em{
A\text{-}\mathrm{mod}_{\mathcal{C}}\ar@/^.7em/[dd]^{\mathfrak{U}}\\
\dashv\\
\ar@/^.7em/[uu]^{\mathfrak{F}} \mathcal{C}
}\qquad\quad
\xymatrix@R=.7em{
(A \otimes B)\text{-}\mathrm{mod}_{\mathcal{C}}\ar@/^.7em/[dd]^{\mathcal{U}}\\
\dashv\\
\ar@/^.7em/[uu]^{\mathcal{F}} B\text{-}\mathrm{mod}_{\mathcal{C}}
} \]
In the former $\mathfrak{U}$ is the forgetful functor $\mathfrak{U}(\mathbb{V}) = V$ and $\mathfrak{F}(X) = (A \otimes X, m_A \otimes \mathrm{id}_X)$.
\begin{lemma}\label{lemmaRelProjInAxB}
Let $\mathbb{W} \in B\text{-}\mathrm{mod}_{\mathcal{C}}$. If $\mathbb{P}$ is relatively projective in $A\text{-}\mathrm{mod}_{\mathcal{C}}$ with respect to the adjunction $\mathfrak{F} \dashv \mathfrak{U}$, then $\mathbb{P} \boxtimes_{\mathcal{C}} \mathbb{W}$ is relatively projective in $(A \otimes B)\text{-}\mathrm{mod}_{\mathcal{C}}$ with respect to the adjunction $\mathcal{F} \dashv \mathcal{U}$.
\end{lemma}
\begin{proof}
By assumption $\mathbb{P}$ is a direct summand of $\mathfrak{F}(X)$ for some $X \in \mathcal{C}$, which means that we have a retract $ \mathrm{id}_{\mathbb{P}} : \mathbb{P} \overset{j}{\to} \mathfrak{F}(X) \overset{\pi}{\to} \mathbb{P}$ in $A\text{-}\mathrm{mod}_{\mathcal{C}}$. It follows that we have a retract
\[ \mathrm{id}_{\mathbb{P} \, \boxtimes_{\mathcal{C}} \mathbb{W}} : \mathbb{P} \boxtimes_{\mathcal{C}} \mathbb{W} \xrightarrow{j \boxtimes_{\mathcal{C}} \mathrm{id}_{\mathbb{W}}} \mathfrak{F}(X) \boxtimes_{\mathcal{C}} \mathbb{W} \xrightarrow{\pi \boxtimes_{\mathcal{C}} \mathrm{id}_{\mathbb{W}}} \mathbb{P} \boxtimes_{\mathcal{C}} \mathbb{W} \]
in $(A \otimes B)\text{-}\mathrm{mod}_{\mathcal{C}}$, so that $\mathbb{P} \boxtimes_{\mathcal{C}} \mathbb{W}$ is a direct summand of $\mathfrak{F}(X) \boxtimes_{\mathcal{C}} \mathbb{W}$. Observe that
\[ \mathcal{F}(X \triangleright \mathbb{W}) = \mathfrak{F}(X) \boxtimes_{\mathcal{C}} \mathbb{W} \]
where $\triangleright$ was defined in \eqref{defModuleStructureBmod}. Indeed:
\begin{align*}
\mathcal{F}(X \triangleright \mathbb{W}) &= \bigl( A \otimes (X \otimes W), \bigl(m_A \otimes ((\mathrm{id}_X \otimes \rho^B_W)(c_{B,X} \otimes \mathrm{id}_W))\bigr)(\mathrm{id}_A \otimes c_{B,A} \otimes \mathrm{id}_{X \otimes W}) \bigr),\\
\mathfrak{F}(X) \boxtimes_{\mathcal{C}} \mathbb{W} &= \bigl( (A \otimes X) \otimes W, (m_A \otimes \mathrm{id}_X \otimes \rho^B_W)(\mathrm{id}_A \otimes c_{B,A\otimes X} \otimes \mathrm{id}_W) \bigr)
\end{align*}
and the two actions coincide thanks to the property of the braiding. We thus get that $\mathbb{P} \boxtimes_{\mathcal{C}} \mathbb{W}$ is a direct summand of $\mathcal{F}(X \triangleright \mathbb{W})$ and hence that it is relatively projective by Lemma \ref{lemmaBasicPropertiesRelProj}.
\end{proof}

Let $(\mathcal{M}, \triangleright)$ be a (left) $\mathcal{C}$-module category. Recall from e.g.\,\cite[\S 7.9]{EGNO} that for $M_1, M_2 \in \mathcal{M}$ the internal Hom object $\underline{\Hom}(M_1,M_2) \in \mathcal{C}$, if it exists, is defined by the natural isomorphism
\begin{equation}\label{defInternalHomObject}
I_X^{M_1, M_2} : \Hom_{\mathcal{M}}(X \triangleright M_1, M_2) \cong \Hom_{\mathcal{C}}(X, \underline{\Hom}(M_1, M_2)).
\end{equation}
Below we will use the shorter notation $I_X$. The following facts will be needed in the proof of Proposition \ref{propATensBrelB} below.
\begin{lemma}\label{lemmaInternalHoms}
Assume that the $\mathcal{C}$-module category $\mathcal{M}$ is strict and that $\underline{\Hom}(M_1, M_2)$ exists.
\\ 1. We have an isomorphism
\[ \Psi_{X,Y} : \Hom_{\mathcal{M}}(X \triangleright M_1, Y \triangleright M_2) \overset{\sim}{\to} \Hom_{\mathcal{C}}\bigl( X, Y \otimes \underline{\Hom}(M_1, M_2) \bigr) \]
which is natural in $X$ and $Y$.
\\2. For $f \in \Hom_{\mathcal{M}}(X \triangleright M_1, Y \triangleright M_2)$ and $g \in \Hom_{\mathcal{C}}(Z,Z')$,
\[ \Psi_{Z \otimes X, Z' \otimes Y}(g \triangleright f) = g \otimes \Psi_{X,Y}(f).  \]
\end{lemma}
\begin{proof}
1. We construct $\Psi_{X,Y}$ as the following composition:
\begin{align*}
\Hom_{\mathcal{M}}(X \triangleright M_1, Y \triangleright M_2) &\xrightarrow{\:\varphi^{X \triangleright M_1, M_2}_Y\:} \Hom_{\mathcal{M}}\bigl((Y^{\vee} \otimes X) \triangleright M_1, M_2\bigr)\\
&\xrightarrow{\:I_{Y^{\vee} \otimes X}\:} \Hom_{\mathcal{C}}\bigl(Y^{\vee} \otimes X, \underline{\Hom}(M_1, M_2) \bigr)\\
&\xrightarrow{\:\overline{\varphi}^{X, \underline{\Hom}(M_1, M_2)}_Y\:} \Hom_{\mathcal{C}}\bigl(X, Y \otimes \underline{\Hom}(M_1, M_2) \bigr)
\end{align*}
where $I$ is the natural isomorphism from \eqref{defInternalHomObject} while the isomorphisms $\varphi, \overline{\varphi}$ are well-known:
\begin{align*}
&\begin{array}{rcrcl}
\varphi_Y^{N_1, N_2} & : & \Hom_{\mathcal{M}}(N_1, Y \triangleright N_2) & \to & \Hom_{\mathcal{M}}(Y^{\vee} \triangleright N_1, N_2),\\
&& g & \mapsto & (\mathrm{ev}_Y \triangleright \mathrm{id}_{N_2})(\mathrm{id}_{Y^{\vee}} \triangleright f)
\end{array}\\[.7em]
&\begin{array}{rcrcl}
\overline{\varphi}_Y^{X_1, X_2} & : & \Hom_{\mathcal{C}}(Y^{\vee} \otimes X_1, X_2) & \to & \Hom_{\mathcal{C}}(X_1, Y \otimes X_2)\\
&& h & \mapsto & (\mathrm{id}_Y \otimes h)(\mathrm{coev}_Y \otimes \mathrm{id}_{X_1}).
\end{array}
\end{align*}
It is easy to check that $\varphi_Y^{N_1, N_2}$ (resp. $\overline{\varphi}_Y^{X_1, X_2}$) is natural in $Y, N_1, N_2$ (resp. in $Y, X_1, X_2$). Using these facts and the naturality of $I$, a straightforward computation shows that $\Psi$ is natural in $X,Y$.
\\2. By the very definition of $\Psi$, we have
\begin{equation}\label{defPsiHom}
\Psi_{X,Y}(f) = \biggl( \mathrm{id}_Y \otimes I_{Y^{\vee} \otimes X}\bigl( (\mathrm{ev}_Y \triangleright \mathrm{id}_{M_2})(\mathrm{id}_{Y^{\vee}} \triangleright f) \bigr) \biggr) \bigl( \mathrm{coev}_Y \otimes \mathrm{id}_X \bigr).
\end{equation}
Thus in order to compute $\Psi_{Z \otimes X, Z' \otimes Y}(g \triangleright f)$ we have to simplify this expression:
\begin{equation}\label{PsiHomfg}
\biggl( \mathrm{id}_{Z' \otimes Y} \otimes I_{Y^{\vee} \otimes (Z')^{\vee} \otimes Z \otimes X}\bigl( (\mathrm{ev}_{Z' \otimes Y} \triangleright \mathrm{id}_{M_2})(\mathrm{id}_{Y^{\vee} \otimes (Z')^{\vee}} \triangleright (g \triangleright f)) \bigr) \biggr) \bigl( \mathrm{coev}_{Z' \otimes Y} \otimes \mathrm{id}_{Z \otimes X} \bigr).
\end{equation}
Note that
\begin{center}
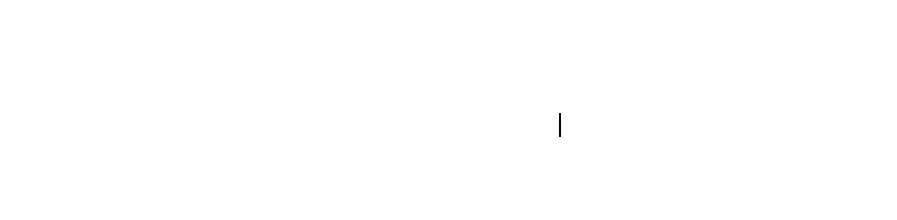
\end{center}
where the objects and morphisms in $\mathcal{C}$ (resp. in $\mathcal{M}$) are in black (resp. in blue). Then due to the naturality of $I$, we obtain:
\begin{equation}\label{simplificationI}
\begin{array}{rl}
&\!\!\!I_{Y^{\vee} \otimes (Z')^{\vee} \otimes Z \otimes X}\biggl( (\mathrm{ev}_Y \triangleright \mathrm{id}_{M_2}) (\mathrm{id}_{Y^{\vee}} \triangleright f) \bigl( \bigl(\mathrm{id}_{Y^{\vee}} \otimes \bigl( \mathrm{ev}_{Z'}(\mathrm{id}_{Z'} \otimes g) \bigr) \otimes \mathrm{id}_X \bigr) \triangleright \mathrm{id}_{M_1} \bigr) \biggr)\\[1em]
=& \!\!\!I_{Y^{\vee} \otimes X}\bigl( (\mathrm{ev}_Y \triangleright \mathrm{id}_{M_2})(\mathrm{id}_{Y^{\vee}} \triangleright f) \bigr) \bigl(\mathrm{id}_{Y^{\vee}} \otimes \bigl( \mathrm{ev}_{Z'}(\mathrm{id}_{Z'} \otimes g) \bigr) \otimes \mathrm{id}_X \bigr)
\end{array}
\end{equation}
and it follows that
\begin{center}
\begingroup%
  \makeatletter%
  \providecommand\color[2][]{%
    \errmessage{(Inkscape) Color is used for the text in Inkscape, but the package 'color.sty' is not loaded}%
    \renewcommand\color[2][]{}%
  }%
  \providecommand\transparent[1]{%
    \errmessage{(Inkscape) Transparency is used (non-zero) for the text in Inkscape, but the package 'transparent.sty' is not loaded}%
    \renewcommand\transparent[1]{}%
  }%
  \providecommand\rotatebox[2]{#2}%
  \newcommand*\fsize{\dimexpr\f@size pt\relax}%
  \newcommand*\lineheight[1]{\fontsize{\fsize}{#1\fsize}\selectfont}%
  \ifx\svgwidth\undefined%
    \setlength{\unitlength}{399.14070285bp}%
    \ifx\svgscale\undefined%
      \relax%
    \else%
      \setlength{\unitlength}{\unitlength * \real{\svgscale}}%
    \fi%
  \else%
    \setlength{\unitlength}{\svgwidth}%
  \fi%
  \global\let\svgwidth\undefined%
  \global\let\svgscale\undefined%
  \makeatother%
  \begin{picture}(1,0.24365132)%
    \lineheight{1}%
    \setlength\tabcolsep{0pt}%
    \put(0,0){\includegraphics[width=\unitlength,page=1]{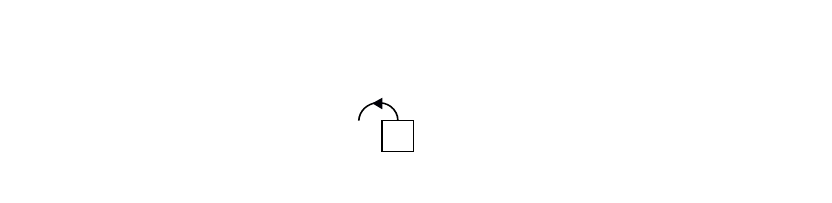}}%
    \put(0.47156195,0.07503479){\color[rgb]{0,0,0}\makebox(0,0)[lt]{\lineheight{1.25}\smash{\begin{tabular}[t]{l}$g$\end{tabular}}}}%
    \put(0,0){\includegraphics[width=\unitlength,page=2]{proofLemmaInternalHom2.pdf}}%
    \put(0.66349955,0.02331093){\color[rgb]{0,0,0}\makebox(0,0)[lt]{\lineheight{1.25}\smash{\begin{tabular}[t]{l}$_X$\end{tabular}}}}%
    \put(0.47108787,0.0189259){\color[rgb]{0,0,0}\makebox(0,0)[lt]{\lineheight{1.25}\smash{\begin{tabular}[t]{l}$_{Z}$\end{tabular}}}}%
    \put(0,0){\includegraphics[width=\unitlength,page=3]{proofLemmaInternalHom2.pdf}}%
    \put(0.3684162,0.15609865){\color[rgb]{0,0,0}\makebox(0,0)[lt]{\lineheight{1.25}\smash{\begin{tabular}[t]{l}$I_{Y^{\vee} \otimes X}\bigl( (\mathrm{ev}_Y \triangleright \mathrm{id}_{M_2})(\mathrm{id}_{Y^{\vee}} \triangleright f) \bigr)$\end{tabular}}}}%
    \put(0,0){\includegraphics[width=\unitlength,page=4]{proofLemmaInternalHom2.pdf}}%
    \put(0.31967672,0.23031617){\color[rgb]{0,0,0}\makebox(0,0)[lt]{\lineheight{1.25}\smash{\begin{tabular}[t]{l}$_{Y}$\end{tabular}}}}%
    \put(0,0){\includegraphics[width=\unitlength,page=5]{proofLemmaInternalHom2.pdf}}%
    \put(0.26330562,0.23031618){\color[rgb]{0,0,0}\makebox(0,0)[lt]{\lineheight{1.25}\smash{\begin{tabular}[t]{l}$_{Z'}$\end{tabular}}}}%
    \put(0,0){\includegraphics[width=\unitlength,page=6]{proofLemmaInternalHom2.pdf}}%
    \put(0.50008091,0.23506146){\color[rgb]{0,0,0}\makebox(0,0)[lt]{\lineheight{1.25}\smash{\begin{tabular}[t]{l}$_{\underline{\Hom}(M_1, M_2)}$\end{tabular}}}}%
    \put(-0.00043059,0.12831128){\color[rgb]{0,0,0}\makebox(0,0)[lt]{\lineheight{1.25}\smash{\begin{tabular}[t]{l}$\Psi_{Z \otimes X, Z' \otimes Y}(g \triangleright f) =$\end{tabular}}}}%
    \put(0.77776946,0.12831135){\color[rgb]{0,0,0}\makebox(0,0)[lt]{\lineheight{1.25}\smash{\begin{tabular}[t]{l}$=g \otimes \Psi_{X,Y}(f)$.\end{tabular}}}}%
  \end{picture}%
\endgroup%

\end{center}
For the first equality we used \eqref{PsiHomfg} and \eqref{simplificationI} while for the second equality we used the zig-zag axiom of ev/coev and \eqref{defPsiHom}.
\end{proof}
Internal Homs can be used to describe Hom spaces between objects of the form $\mathbb{V} \boxtimes_{\mathcal{C}} \mathbb{W}$, see \cite[Thm.\,3.3]{DSPS}. The next proposition shows that they can also be used to describe Ext spaces between such objects.
\begin{proposition}\label{propATensBrelB}
Let $\mathcal{C}$ be a braided finite tensor category and $A,B$ be algebras in $\mathcal{C}$. Let $\mathbb{V}, \mathbb{V}'$ be objects in $A\text{-}\mathrm{mod}_{\mathcal{C}}$ and $\mathbb{W}, \mathbb{W}'$ be objects in $B\text{-}\mathrm{mod}_{\mathcal{C}}$ such that $\underline{\Hom}(\mathbb{W},\mathbb{W}')$ exists. Then we have for all $n \geq 0$:
\[ \Ext^n_{(A \otimes B)\text{-}\mathrm{mod}_{\mathcal{C}}, \, B\text{-}\mathrm{mod}_{\mathcal{C}}}(\mathbb{V} \boxtimes_{\mathcal{C}} \mathbb{W}, \mathbb{V}' \boxtimes_{\mathcal{C}} \mathbb{W}') \cong \Ext_{A\text{-}\mathrm{mod}_{\mathcal{C}},\,\mathcal{C}}^n\bigl(\mathbb{V}, \mathbb{V}' \triangleleft \underline{\Hom}(\mathbb{W},\mathbb{W}')\bigr). \]
\end{proposition}
\begin{proof}
Let $0 \longleftarrow \mathbb{V} \overset{d_0}{\longleftarrow} \mathbb{P}_0 \overset{d_1}{\longleftarrow} \mathbb{P}_1 \overset{d_2}{\longleftarrow} \ldots$ be a relatively projective resolution in $A\text{-}\mathrm{mod}_{\mathcal{C}}$ and since $d_i$ is allowable let $t_i \in \Hom_{\mathcal{C}}(P_{i-1},P_i)$ be such that $d_i \, t_i \, d_i = d_i$ (we write $d_i$ instead of $\mathfrak{U}(d_i)$ because $\mathfrak{U}$ is the identity on morphisms). By the Lemma \ref{lemmaRelProjInAxB}
\begin{equation}\label{relProjResProofAxB}
0 \longleftarrow \mathbb{V} \boxtimes_{\mathcal{C}} \mathbb{W} \xleftarrow{d_0 \otimes \mathrm{id}_{\mathbb{W}}} \mathbb{P}_0 \boxtimes_{\mathcal{C}} \mathbb{W} \xleftarrow{d_1 \otimes \mathrm{id}_{\mathbb{W}}} \mathbb{P}_1 \boxtimes_{\mathcal{C}} \mathbb{W} \xleftarrow{d_2 \otimes \mathrm{id}_{\mathbb{W}}} \ldots
\end{equation}
is an exact sequence of relatively projective modules in $(A \otimes B)\text{-}\mathrm{mod}_{\mathcal{C}}$. Let us show that it is allowable (with respect to the adjunction $\mathcal{F} \dashv \mathcal{U}$). Note that $\mathcal{U}(\mathbb{P}_i \boxtimes_{\mathcal{C}} \mathbb{W}) = P_i \triangleright \mathbb{W}$ with the notation from \eqref{defModuleStructureBmod}. Hence $t_i \otimes \mathrm{id}_{W} = t_i \triangleright \mathrm{id}_{\mathbb{W}} \in \Hom_{B\text{-}\mathrm{mod}_{\mathcal{C}}}(\mathcal{U}(\mathbb{P}_{i-1} \boxtimes_{\mathcal{C}} \mathbb{W}), \mathcal{U}(\mathbb{P}_i \boxtimes_{\mathcal{C}} \mathbb{W}))$ and $(d_i \otimes \mathrm{id}_{W}) (t_i \otimes \mathrm{id}_{W}) (d_i \otimes \mathrm{id}_{W}) = (d_i \otimes \mathrm{id}_{W})$, which shows that $d_i \otimes \mathrm{id}_{W}$ is allowable (again, we do not write $\mathcal{U}(d_i \otimes \mathrm{id}_{W})$ because $\mathcal{U}$ is the identity on morphisms). Hence \eqref{relProjResProofAxB} is a relatively projective resolution.

\smallskip

\indent Let us now construct an isomorphism of complexes which will prove the result. From Lemma \ref{lemmaInternalHoms} we have an isomorphism
\[ \Psi_{P,V'} : \Hom_{B\text{-}\mathrm{mod}_{\mathcal{C}}}(P \triangleright \mathbb{W}, V' \triangleright \mathbb{W}') \overset{\sim}{\to} \Hom_{\mathcal{C}}\bigl( P, V' \otimes \underline{\Hom}(\mathbb{W}, \mathbb{W}') \bigr) \]
which is natural in $P$ and $V'$. We define a natural isomorphism
\begin{equation*}
\widetilde{\Psi}_{\mathbb{P},\mathbb{V}'} : \Hom_{(A \otimes B)\text{-}\mathrm{mod}_{\mathcal{C}}}(\mathbb{P} \boxtimes_{\mathcal{C}} \mathbb{W}, \mathbb{V}' \boxtimes_{\mathcal{C}} \mathbb{W}') \overset{\sim}{\to} \Hom_{A\text{-}\mathrm{mod}_{\mathcal{C}}}\bigl( \mathbb{P}, \mathbb{V}' \triangleleft \underline{\Hom}(\mathbb{W}, \mathbb{W}') \bigr).
\end{equation*}
by letting $\widetilde{\Psi}_{\mathbb{P},\mathbb{V}'}(f) = \Psi_{P,V'}\bigl( \mathcal{U}(f) \bigr)$. We use for instance that $\mathcal{U}(\mathbb{P} \boxtimes_{\mathcal{C}} \mathbb{W}) = P \triangleright \mathbb{W}$ for any $\mathbb{P} = (P, \rho^P) \in A\text{-}\mathrm{mod}_{\mathcal{C}}$ \textit{etc}. Recall that $\mathcal{U}(f) = f$. Let us check that $\widetilde{\Psi}_{\mathbb{P},\mathbb{V}'}(f) = \Psi_{P,V'}(f)$ is indeed $A$-linear; this is based on two basic observations. First, any $f \in \Hom_{(A \otimes B)\text{-}\mathrm{mod}_{\mathcal{C}}}(\mathbb{P} \boxtimes_{\mathcal{C}} \mathbb{W}, \mathbb{V}' \boxtimes_{\mathcal{C}} \mathbb{W}')$ satisfies
\begin{equation}\label{fIsALinear}
\begin{array}{rl}
f(\rho^A_P \otimes \mathrm{id}_W)\!\!\! &= f \, \rho^{A \otimes B}_{P \otimes W} \, (\mathrm{id}_A \otimes \eta_B \otimes \mathrm{id}_{P \otimes W}) = \rho^{A \otimes B}_{V' \otimes W'} \, (\mathrm{id}_{A \otimes B} \otimes f) (\mathrm{id}_A \otimes \eta_B \otimes \mathrm{id}_{P \otimes W})\\[.6em]
&= \rho^{A \otimes B}_{V' \otimes W'} \, (\mathrm{id}_A \otimes \eta_B \otimes \mathrm{id}_{V' \otimes W'}) (\mathrm{id}_A \otimes f) = (\rho^A_{V'} \otimes \mathrm{id}_{W'}) (\mathrm{id}_A \otimes f)
\end{array}
\end{equation}
where for the first and last equalities we used the definition \eqref{defActionTensorProduct} together with the fact that the unit $\eta_B$ acts as the identity on any $B$-module. Second, the naturality of $\Psi$ means that for $g \in \Hom_{\mathcal{C}}(P,Q)$ and $h \in \Hom_{\mathcal{C}}(V'_1, V'_2)$ we have
\[ \Psi_{P,V'} \circ (g \triangleright \mathrm{id}_W)^* = g^* \circ \Psi_{Q,V'}, \qquad  (h \otimes \mathrm{id}_{\underline{\Hom}(\mathbb{W}, \mathbb{W}')})_* \circ \Psi_{P,V'_1} = \Psi_{P,V'_2} \circ (h \triangleright \mathrm{id}_{W'})_* \]
where $(-)^*$ and $(-)_*$ respectively denote the pullback and the push-forward; more explicitly for $f \in  \Hom_{B\text{-}\mathrm{mod}_{\mathcal{C}}}(P \triangleright \mathbb{W}, V' \triangleright \mathbb{W}')$:
\[ \Psi_{P,V'}\bigl( f(g \otimes \mathrm{id}_W) \bigr) = \Psi_{Q,V'}(f)\,g, \qquad (h \otimes \mathrm{id}_{\underline{\Hom}(\mathbb{W}, \mathbb{W}')}) \Psi_{P,V'_1}(f) = \Psi_{P,V'_2}\bigl( (h \otimes \mathrm{id}_{W'})f \bigr). \]
Applying this to $g = \rho^A_P \in \Hom_{\mathcal{C}}(A \otimes P, P)$ and $h = \rho^A_{V'} \in \Hom_{\mathcal{C}}(A \otimes V', V')$, we get for $f \in \Hom_{(A \otimes B)\text{-}\mathrm{mod}_{\mathcal{C}}}(\mathbb{P} \boxtimes_{\mathcal{C}} \mathbb{W}, \mathbb{V}' \boxtimes_{\mathcal{C}} \mathbb{W}')$:
\begin{align*}
\Psi_{P,V'}(f) \, \rho^A_P &= \Psi_{A \otimes P, V'}\bigl( f \, (\rho^A_P \otimes \mathrm{id}_{W}) \bigr) = \Psi_{A \otimes P, V'}\bigl( (\rho^A_{V'} \otimes \mathrm{id}_{W'}) (\mathrm{id}_A \otimes f) \bigr)\\
&= \bigl(\rho^A_{V'} \otimes \mathrm{id}_{\underline{\Hom}(\mathbb{W},\mathbb{W}')}\bigr) \Psi_{A \otimes P, A \otimes V'}\bigl( \mathrm{id}_A \otimes f \bigr) = \bigl(\rho^A_{V'} \otimes \mathrm{id}_{\underline{\Hom}(\mathbb{W},\mathbb{W}')}\bigr) \bigl( \mathrm{id}_A \otimes \Psi_{P, V'}(f) \bigr)
\end{align*}
where the second equality is from \eqref{fIsALinear} and the last equality is an application of the second item in Lemma \ref{lemmaInternalHoms}. This means that $\widetilde{\Psi}_{\mathbb{P},\mathbb{V}'}(f)$ is $A$-linear. Thanks to naturality in $\mathbb{P}$, the family $\bigl(\widetilde{\Psi}_{\mathbb{P}_i,\mathbb{V}'}\bigr)_{i \in \mathbb{N}}$ is an isomorphism of complexes.
\end{proof}

\begin{remark}
Let $H$ be a finite-dimensional quasi-triangular Hopf algebra (in $\mathrm{Vect}_k$) and take $\mathcal{C} = H\text{-}\mathrm{mod}$. Let $A$ and $B$ be finite-dimensional $H$-module algebras (or in other words algebras in $\mathcal{C}$). Then $A\text{-}\mathrm{mod}_{\mathcal{C}} \cong (A \# H)\text{-}\mathrm{mod}$ where $A\# H$ is the smash product of $A$ and $H$, and similarly for $B$. Let $A \, \widetilde{\otimes} B$ be the braided tensor product of $A$ and $B$ (in other words, their tensor product in $\mathcal{C}$). Proposition \ref{propATensBrelB} is rewritten as
\[ \Ext^n_{(A \widetilde{\otimes} B) \# H, B\#H}(\mathbb{V} \boxtimes_{\mathcal{C}} \mathbb{W}, \mathbb{V}' \boxtimes_{\mathcal{C}} \mathbb{W}') \cong \Ext_{A\#H, H}^n\bigl(\mathbb{V}, \mathbb{V}' \triangleleft \Hom_B(\mathbb{W},\mathbb{W}')\bigr). \]
Indeed it is not difficult to see that in this situation, the internal Hom object $\underline{\Hom}(\mathbb{W},\mathbb{W}')$ is $\Hom_B(\mathbb{W},\mathbb{W}')$ endowed with the $H$-action given by $(h \cdot f)(w) = h'' \cdot f(S^{-1}(h') \cdot w)$, where $\Delta(h) = h' \otimes h''$ is the coproduct and $S$ is the antipode.
\end{remark}

\smallskip

Recall that a morphism $\beta \in \Hom_{\mathcal{C}}(B,\mathbf{1})$ is called an augmentation for the algebra $B$ if it satisfies $\beta \otimes \beta = \beta \, m_B$ and $\beta \, \eta_B = \mathrm{id}_{\mathbf{1}}$. If $B$ has an augmentation $\beta$ then the tensor unit $\mathbf{1} \in \mathcal{C}$ can be endowed with a structure of $B$-module internal to $\mathcal{C}$: $\mathbf{1}_{\beta} = (\mathbf{1}, \beta) \in B\text{-}\mathrm{mod}_{\mathcal{C}}$.

\begin{corollary}\label{relATensBsemisimple}
1. With the notations of Proposition \ref{propATensBrelB} and assuming moreover that the category $\mathcal{C}$ is semisimple, we have
\[ \Ext^n_{(A \otimes B)\text{-}\mathrm{mod}_{\mathcal{C}}, B\text{-}\mathrm{mod}_{\mathcal{C}}}(\mathbb{V} \boxtimes_{\mathcal{C}} \mathbb{W}, \mathbb{V}' \boxtimes_{\mathcal{C}} \mathbb{W}') \cong \Ext^n_{A\text{-}\mathrm{mod}_{\mathcal{C}}}\bigl(\mathbb{V}, \mathbb{V}' \triangleleft \underline{\Hom}(\mathbb{W},\mathbb{W}')\bigr). \]
2. If the algebra $B$ has an augmentation $\beta$ then
\[ \Ext^n_{(A \otimes B)\text{-}\mathrm{mod}_{\mathcal{C}}, \, B\text{-}\mathrm{mod}_{\mathcal{C}}}(\mathbb{V} \boxtimes_{\mathcal{C}} \mathbf{1}_{\beta}, \mathbb{V}' \boxtimes_{\mathcal{C}} \mathbf{1}_{\beta}) \cong \Ext_{A\text{-}\mathrm{mod}_{\mathcal{C}}, \, \mathcal{C}}^n(\mathbb{V}, \mathbb{V}') \]
where here we do not assume that $\mathcal{C}$ is semisimple.
\end{corollary}
\begin{proof}
1. Since $\mathcal{C}$ is semisimple, we know from Example \ref{exampleSemisimpleBottom} that $\Ext_{A\text{-}\mathrm{mod}_{\mathcal{C}},\,\mathcal{C}}^n(\mathbb{V}, \mathbb{M}) = \Ext_{A\text{-}\mathrm{mod}_{\mathcal{C}}}^n(\mathbb{V}, \mathbb{M})$ and the claim follows directly from Proposition \ref{propATensBrelB}.
\\2. We have $\Hom_{B\text{-}\mathrm{mod}_{\mathcal{C}}}(X \triangleright \mathbf{1}_{\beta}, \mathbf{1}_{\beta}) = \Hom_{\mathcal{C}}(X,\mathbf{1})$. Indeed as an object in $\mathcal{C}$, $X \triangleright \mathbf{1}_{\beta}$ is equal to $X$ and one sees that a morphism $f : X \to \mathbf{1}$ in $\mathcal{C}$ is the same thing as a $B$-linear morphism $f = f \otimes \mathrm{id}_{\mathbf{1}} : X \triangleright \mathbf{1}_{\beta} = X \to \mathbf{1}_{\beta}$. Hence by definition $\underline{\Hom}(\mathbf{1}_{\beta}, \mathbf{1}_{\beta}) = \mathbf{1}$ and the claim follows from Proposition \ref{propATensBrelB}.
\end{proof}

\section{Davydov--Yetter cohomology and relative $\Ext$ groups}\label{sectionDYcohomologyAndRelExt}
\indent Davydov--Yetter (DY) cohomology classifies infinitesimal deformations of tensor structures, as reviewed in the Introduction. Here we will work with the version of this cohomology theory with coefficients introduced in \cite{GHS}.

\smallskip

\indent In \S \ref{sectionRelationDYCohomology} we relate DY cohomology to the relative Ext groups of a certain adjunction, as a consequence of the work of \cite{GHS} and of Proposition \ref{relExtAndComonadCohom}. In \S \ref{sectionFirstConsequences} we derive the easiest consequences of this relation; then in \S \ref{subsectionProductOnDY} we obtain a product on DY cohomology thanks to the Yoneda product on the relative Ext groups and we determine an explicit formula for it; finally in \S \ref{subsectionLongExactSequenceDYCohom} we obtain a long exact sequence for DY cohomology thanks to the long exact sequence for relative Ext groups. \S \ref{sectionAdjunctionCentralizerFunctor} contains properties of the underlying adjunction which are used to prove some of the results.

\medskip

\indent Recall from \cite[Chap.\,IX, \S 6]{macLaneCategories} the notion of a coend. Let $\mathcal{X}, \mathcal{Y}$ be categories and let $T : \mathcal{X}^{\mathrm{op}} \times \mathcal{X} \to \mathcal{Y}$ be a bifunctor. The coend of $T$, if it exists, is a universal pair $(C,i)$, where $C$ is an object in $\mathcal{Y}$ and $i = ( i_X : T(X,X) \to C )_{X \in \mathcal{X}}$ is a dinatural transformation. This means that for any dinatural transformation $\psi_X : T(X,X) \to V$, there exists a unique morphism $\Psi \in \Hom_{\mathcal{Y}}(C,V)$ such that $\psi_X = \Psi \, i_X$ for all $X \in \mathcal{X}$. The object $C$ is denoted by $\int^{X \in \mathcal{X}} T(X,X)$. Note in particular that for any object $W \in \mathcal{Y}$, two morphisms $f, g \in \Hom_{\mathcal{Y}}(C, W)$ are equal if and only if $f \, i_X = g \, i_X$ for all $X \in \mathcal{X}$.

\smallskip

\indent In this section, $\mathcal{C}$ and $\mathcal{D}$ denote $k$-linear finite strict tensor categories (see definitions at the beginning of section \S \ref{relativeExtTensorCategories}) and $F : \mathcal{C} \to \mathcal{D}$ is a $k$-linear strict monoidal functor, with $k$ a field.

\begin{remark}\label{remarkStrictification}
The strictness assumptions are not a loss of generality because any monoidal functor admits a strictification, which is a strict monoidal functor between strict monoidal categories \cite{JS}; in a future paper we will show that the DY cohomology with coefficients of a functor is isomorphic to the DY cohomology with coefficients of its strictification. 
\end{remark}

\subsection{The centralizer of a functor and the associated adjunction}\label{sectionAdjunctionCentralizerFunctor}
\indent Recall that there exists a monoidal natural isomorphism $d_X : F(X^{\vee}) \overset{\sim}{\to} F(X)^{\vee}$ (see \textit{e.g.}\,\cite[Lem.\,1.1]{NS}); in the sequel for simplicity we identify $F(X)^{\vee}$ with $F(X^{\vee})$.

\begin{itemize}
\item A half-braiding relative to $F$ is a natural isomorphism $\rho^V : V \otimes F(-) \to F(-) \otimes V$ such that 
\begin{equation}\label{halfBraidingOnTensorProducts}
\rho^V_{X \otimes Y} = (\mathrm{id}_{F(X)} \otimes \rho^V_Y)(\rho^V_X \otimes \mathrm{id}_{F(Y)})
\end{equation}
for all $X,Y \in \mathcal{C}$.
\item The centralizer $\mathcal{Z}(F)$ of the functor $F$ is the category whose objects are pairs $\mathsf{V} = (V,\rho^V)$ where $V \in \mathcal{D}$ and $\rho^V$ is a half-braiding relative to $F$ and whose morphisms $f : (V,\rho^V) \to (W,\rho^W)$ are morphisms $f \in \Hom_{\mathcal{D}}(V,W)$ such that the diagram
\[\xymatrix{
V \otimes F(X) \ar[d]_{f \otimes \mathrm{id}_{F(X)}}\ar[r]^{\rho^V_X} & F(X) \otimes V \ar[d]^{\mathrm{id}_{F(X)}\otimes f}\\
W \otimes F(X) \ar[r]_{\rho^W_X} & F(X) \otimes W
}\]
commutes for all $X \in \mathcal{C}$.
\end{itemize}
\noindent $\mathcal{Z}(F)$ is a (strict) tensor category \cite{majid}, \cite[\S 3.2]{shimizu} with tensor product
\[ (V, \rho^V) \otimes (W,\rho^W) = \bigl(V \otimes W, \, \rho^{V \otimes W} = (\rho^V \otimes \mathrm{id}_W)(\mathrm{id}_V \otimes \rho^W) \bigr) \]
and with left and right duals given by $(V, \rho^V)^{\vee} = (V^{\vee}, \rho^{V^{\vee}})$, $^{\vee}(V, \rho^V) = ({^{\vee}V}, \rho^{^{\vee}V})$, where
\begin{equation}\label{defDualityZF}
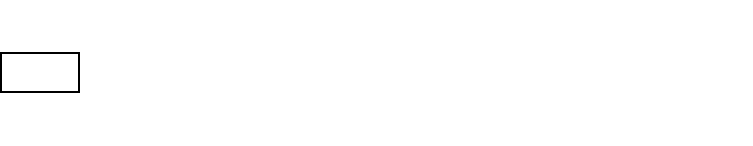
\end{equation}
\noindent In particular, if $F = \mathrm{Id}_{\mathcal{C}}$ is the identity functor, then $\mathcal{Z}(\mathrm{Id}_{\mathcal{C}})$ is $\mathcal{Z}(\mathcal{C})$, the usual Drinfeld center of $\mathcal{C}$.

\smallskip

\noindent \textbf{Assumption:} We assume from now on that $F$ is an exact functor.

\smallskip

\noindent Under this assumption, $\mathcal{Z}(F)$ is a {\em finite} tensor category, see \cite[\S 3.3]{shimizu}\footnote{\cite{shimizu} assumes that $k$ is algebraically closed and that $\boldsymbol{1}$ is a simple object. Here we also assume that $\boldsymbol{1}$ is a simple object but not that $k$ is algebraically closed; instead we add the assumption that $\mathrm{End}_{\mathcal{C}}(\boldsymbol{1}) \cong k$. The fact that $\mathcal{Z}(F)$ is finite remains true under these assumptions.}.

\smallskip

\indent Let $\mathcal{U} : \mathcal{Z}(F) \to \mathcal{D}$ be the forgetful functor defined by $\mathcal{U}(V,\rho^V) = V$, $\mathcal{U}(f) = f$. Let us recall the explicit construction of its left adjoint $\mathcal{F} : \mathcal{D} \to \mathcal{Z}(F)$ (\cite{DS, BV, shimizu}, here we use the conventions of \cite[\S 3.3]{GHS}). First, consider
\begin{equation}\label{ZFcoend}
Z_F(V) = \int^{X \in \mathcal{C}} F(X)^{\vee} \otimes V \otimes F(X)
\end{equation}
and denote by 
\[ i_X(V) : F(X)^{\vee} \otimes V \otimes F(X) \to Z_F(V) \]
the associated universal dinatural transformation. The coend $Z_F(V)$ exists for all $V$ by exactness of $F$ and of $\otimes$, thanks to \cite[Cor.\,5.1.8.]{KL}. Moreover, $Z_F$ gives a functor $\mathcal{D} \to \mathcal{D}$; for a morphism $f : V \to W$, $Z_F(f)$ is defined to be the unique morphism such that
\begin{equation}\label{defZFf}
Z_F(f) \, i_X(V) = i_X(W) \bigl(\mathrm{id}_{F(X)^{\vee}} \otimes f \otimes \mathrm{id}_{F(X)}\bigr)
\end{equation}
for all $X \in \mathcal{C}$. Here we use the universality of the coend, the right hand side of \eqref{defZFf} being dinatural in $X$. In the sequel we often implicitly use universality of the coend to define morphisms. For any $Y \in \mathcal{C}$, we define $\rho^{Z_F(V)}_Y : Z_F(V) \otimes F(Y) \to F(Y) \otimes Z_F(V)$ as the unique morphism in $\mathcal{D}$ such that
\begin{equation}\label{defRhoZF}
\rho_Y^{Z_F(V)} \bigl( i_X(V) \otimes \mathrm{id}_{F(Y)} \bigr) = \bigl( \mathrm{id}_{F(Y)} \otimes i_{X \otimes Y}(V) \bigr) \bigl( \mathrm{coev}_{F(Y)} \otimes \mathrm{id}_{F(X)^{\vee} \otimes V \otimes F(X) \otimes F(Y)} \bigr)
\end{equation}
for all $X \in \mathcal{C}$ (again we use the universality of the coend). Then one can show that $\rho^{Z_F(V)}$ is a half-braiding relative to $F$ and that the left adjoint to $\mathcal{U}$ is given by
\[ \mathcal{F}(V) = \bigl( Z_F(V), \rho^{Z_F(V)} \bigr). \]
The forgetful functor $\mathcal{U}$ is clearly additive, exact and faithful, and it has the left adjoint $\mathcal{F}$, so we have a resolvent pair of categories:
\begin{equation}\label{relSystemZF}
\xymatrix@R=.7em{
\mathcal{Z}(F)\ar@/^.7em/[dd]^{\mathcal{U}}\\
\dashv\\
\ar@/^.7em/[uu]^{\mathcal{F}}\mathcal{D}
}
\end{equation}
Since $\mathcal{U} : \mathcal{Z}(F) \to \mathcal{D}$ is obviously monoidal, Proposition \ref{relProjTensorIdeal} applies:
\begin{corollary}
The relatively projective objects in $\mathcal{Z}(F)$ form a tensor ideal.
\end{corollary}

\smallskip

\indent In the sequel we will use intensively the bar resolution \eqref{barResolutionG} (or equivalently \eqref{relativeBarResolution}) for the adjunction \eqref{relSystemZF}, so in the rest of this section we describe it in detail. First, let $G$ be the comonad on $\mathcal{Z}(F)$ for this adjunction:
\begin{equation}\label{defComonadOnZF}
G = \mathcal{F}\mathcal{U} : \mathcal{Z}(F) \to \mathcal{Z}(F), \quad \:\:  G(\mathsf{V}) = \bigl(Z_F(V), \rho^{Z_F(V)} \bigr), \quad \:\: G(f) = Z_F(f)
\end{equation}
where $\mathsf{V} = (V, \rho^V) \in \mathcal{Z}(F)$ and $f$ is a morphism (recall that a morphism $f \in \Hom_{\mathcal{Z}(F)}(\mathsf{V}, \mathsf{W})$ is just a morphism $f \in \Hom_{\mathcal{D}}(V,W)$ which commutes with the half-braidings $\rho^V, \rho^W$). For any object $\mathsf{V}$, let $\varepsilon_{\mathsf{V}} \in \Hom_{\mathcal{D}}\bigl( Z_F(V), V \bigr)$ be the unique morphism such that
\begin{equation}\label{defCounitG}
\varepsilon_{\mathsf{V}} \, i_X(V) = \bigl( \mathrm{ev}_{F(X)} \otimes \mathrm{id}_V \bigr) \bigl( \mathrm{id}_{F(X)^{\vee}} \otimes \rho^V_X \bigr). 
\end{equation}
One can show that $\varepsilon_{\mathsf{V}} \in \Hom_{\mathcal{Z}(F)}\bigl( G(\mathsf{V}), \mathsf{V} \bigr)$ and that $\varepsilon : G \to \mathrm{Id}$ is the counit of $G$.

\smallskip

\indent The chain objects of the bar resolution of $\mathsf{V}$ are
\begin{equation*}
\mathrm{Bar}^n_{\mathcal{Z}(F), \mathcal{D}}(\mathsf{V}) = \mathrm{Bar}^n_G(\mathsf{V}) = G^{n+1}(\mathsf{V}) = \biggl( Z_F^{n+1}(V), \rho^{Z_F^{n+1}(V)} \biggr)
\end{equation*}
for $\mathsf{V} = (V,\rho^V) \in \mathcal{Z}(F)$, and where the first equality simply reminds the two possible notations. Using the Fubini theorem for coends \cite[Prop.\,IX.8]{macLaneCategories}, we have
\[ Z_F^n(V) = \int^{(X_1, \ldots, X_n) \in \mathcal{C}^n} F(X_n)^{\vee} \otimes \ldots \otimes F(X_1)^{\vee} \otimes V \otimes F(X_1) \otimes \ldots \otimes F(X_n) \]
with the universal dinatural transformation $i^{(n)}_{X_1, \ldots, X_n}(V)$ defined inductively by
\begin{align*}
i^{(n)}_{X_1, \ldots, X_n}(V) &= i_{X_n}\bigl( Z_F^{n-1}(V) \bigr) \bigl( \mathrm{id}_{F(X_n)^{\vee}} \otimes i^{(n-1)}_{X_1, \ldots, X_{n-1}}(V) \otimes \mathrm{id}_{F(X_n)} \bigr)\\
&= i^{(n-1)}_{X_2, \ldots, X_n}\bigl( Z_F(V) \bigr) \bigl( \mathrm{id}_{F(X_n)^{\vee} \otimes \ldots \otimes F(X_2)^{\vee}} \otimes i_{X_1}(V) \otimes \mathrm{id}_{F(X_2) \otimes \ldots \otimes F(X_n)} \bigr).
\end{align*}
We can characterize the bar differential $d^{\mathsf{V}}_n : \mathrm{Bar}_G^n(\mathsf{V}) \to \mathrm{Bar}_G^{n-1}(\mathsf{V})$ thanks to the universal dinatural transformations:
\begin{lemma}\label{lemmaBarDifferentialZF}
For the comonad $G$ on $\mathcal{Z}(F)$, the differential of the bar resolution for an object $\mathsf{V} = (V,\rho^V)$ satisfies
\begin{align*}
d^{\mathsf{V}}_n \, i^{(n+1)}_{X_1, \ldots, X_{n+1}}(V) =\:\, & i^{(n)}_{X_2, \ldots, X_{n+1}}(V) \bigl( \mathrm{id}_{F(X_{n+1})^{\vee} \otimes \ldots \otimes F(X_2)^{\vee}} \otimes \varepsilon_{\mathsf{V}}i_{X_1}(V) \otimes \mathrm{id}_{F(X_2) \otimes \ldots \otimes F(X_{n+1})}\bigr)\\
&+ \sum_{j=1}^n (-1)^j \, i^{(n)}_{X_1, \ldots, X_j \otimes X_{j+1}, \ldots, X_{n+1}}(V)
\end{align*}
where $\varepsilon_{\mathsf{V}}$ is the counit of $G$, defined in \eqref{defCounitG}.
\end{lemma}
\begin{proof}
Recall the general definition of the bar differential in \eqref{barDifferentialG}. The result is obtained by a computation based on the definitions of the counit and of the value of $G$ on a morphism  (see \eqref{defComonadOnZF} and \eqref{defZFf}); we skip the details.
\end{proof}

\indent Let $\mathsf{V} = (V,\rho^V), \mathsf{W} = (W, \rho^W) \in \mathcal{Z}(F)$ and for $n \in \mathbb{N}$ let 
\[ \Psi^{\mathsf{V}, \mathsf{W}}_n \in \Hom_{\mathcal{D}}\big( Z_F^n(V \otimes W), Z_F^n(V) \otimes W \big) \]
be the unique morphism such that $\Psi^{\mathsf{V}, \mathsf{W}}_n \, i^{(n)}_{X_1, \ldots, X_n}(V \otimes W)$ is equal to the following dinatural transformation
\begin{align*}
&F(X_n)^{\vee} \otimes \ldots \otimes F(X_1)^{\vee} \otimes V \otimes W \otimes F(X_1) \otimes \ldots \otimes F(X_n)\\ &\xrightarrow{\mathrm{id}_{F(X_n)^{\vee} \otimes \ldots \otimes F(X_1)^{\vee} \otimes V} \otimes \rho^W_{X_1 \otimes \ldots \otimes X_n}} F(X_n)^{\vee} \otimes \ldots \otimes F(X_1)^{\vee} \otimes V \otimes F(X_1) \otimes \ldots \otimes F(X_n) \otimes W \\
&\xrightarrow{i^{(n)}_{X_1, \ldots, X_n}(V) \otimes \mathrm{id}_W} Z_F(V) \otimes W.
\end{align*}
\begin{proposition}\label{propGVtensW}
For each $n$, $\Psi^{\mathsf{V}, \mathsf{W}}_n$ is an isomorphism in $\mathcal{Z}(F)$:
\begin{equation*}
\Psi^{\mathsf{V}, \mathsf{W}}_n : G^n(\mathsf{V} \otimes \mathsf{W}) \overset{\sim}{\to} G^n(\mathsf{V}) \otimes \mathsf{W}.
\end{equation*}
Moreover, the family $\big(\Psi^{\mathsf{V}, \mathsf{W}}_n\big)_{n \in \mathbb{N}}$ is an isomorphism of complexes from $\mathrm{Bar}^{\bullet}_G(\mathsf{V} \otimes \mathsf{W})$ to $\mathrm{Bar}^{\bullet}_G(\mathsf{V}) \otimes \mathsf{W}$.
\end{proposition}
\begin{proof}
Let us first show that $\Psi_1^{\mathsf{V}, \mathsf{W}}$ is an isomorphism in $\mathcal{Z}(F)$. The inverse of $\Psi_1^{\mathsf{V}, \mathsf{W}}$ is simply the unique morphism in $\Hom_{\mathcal{D}}\bigl( Z_F(V) \otimes W, Z_F(V \otimes W) \bigr)$ determined by the following dinatural transformation:
\[ F(X)^{\vee} \otimes V \otimes F(X) \otimes W \xrightarrow{\mathrm{id}_{F(X)^{\vee}} \otimes (\rho^W_X)^{-1}} F(X)^{\vee} \otimes V \otimes W \otimes F(X) \xrightarrow{i_X(V \otimes W)} Z_F(V \otimes W). \]
A morphism in $\mathcal{Z}(F)$ is a morphism in $\mathcal{D}$ which commutes with the half-braidings, so we have to check that the following diagram is commuative:
\begin{equation}\label{diagramPsiAndHalfBraidings}
\xymatrix@C=6em{
Z_F(V \otimes W) \otimes F(Y) \ar[r]^{\rho^{Z_F(V \otimes W)}_Y} \ar[d]_{\Psi_1^{\mathsf{V}, \mathsf{W}} \otimes \mathrm{id}_{F(Y)}}& F(Y) \otimes Z_F(V \otimes W) \ar[d]^{\mathrm{id}_{F(Y)} \otimes \Psi_1^{\mathsf{V}, \mathsf{W}}}\\
Z_F(V) \otimes W \otimes F(Y) \ar[r]_{\rho^{Z_F(V) \otimes W}_Y}& F(Y) \otimes Z_F(V) \otimes W
}\end{equation}
for all $Y \in \mathcal{C}$. The computation is displayed in Figure \ref{proofPsiCommutesHalfBraidings}. The first equality uses the definition of $\Psi_1^{\mathsf{V}, \mathsf{W}}$ and $\rho^{Z_F(V) \otimes W}_Y = \big(\rho^{Z_F(V)}_Y \otimes \mathrm{id}_W\big) \big(\mathrm{id}_{Z_F(V)} \otimes \rho^W_Y\big)$, the second equality uses the definition of $\rho^{Z_F(V)}$ (see \eqref{defRhoZF}) and $(\mathrm{id}_{F(X)} \otimes \rho^W_Y)(\rho^W_X \otimes \mathrm{id}_{F(Y)}) = \rho^W_{X \otimes Y}$, and the third and fourth equalities are the definitions of $\Psi_1^{\mathsf{V}, \mathsf{W}}$ and $\rho^{Z_F(V \otimes W)}$ respectively. Since this holds for any $X \in \mathcal{C}$, the diagram is commutative. Now for general $n$, it is not difficult to show that $\Psi_{n+1}^{\mathsf{V}, \mathsf{W}}$ can be constructed as follows:
\[ Z_F^{n+1}(V \otimes W) \xrightarrow{Z_F(\Psi_n^{\mathsf{V},\mathsf{W}})} Z_F\big( Z_F^n(V) \otimes W \big) \xrightarrow{\Psi_1^{Z_F^n(V),W}} Z_F^{n+1}(V) \otimes W. \]
Hence by induction we obtain that $\Psi_n^{\mathsf{V}, \mathsf{W}}$ is an isomorphism in $\mathcal{Z}(F)$ for all $n$. To prove the last claim, we must check that
\[ \xymatrix@C=4em{
G^{n+1}(\mathsf{V} \otimes \mathsf{W}) \ar[r]^{d^{\mathsf{V \otimes W}}_n} \ar[d]_{\Psi_{n+1}^{\mathsf{V}, \mathsf{W}}}& G^n(\mathsf{V} \otimes \mathsf{W}) \ar[d]^{\Psi_n^{\mathsf{V},\mathsf{W}}}\\
G^{n+1}(\mathsf{V}) \otimes \mathsf{W} \ar[r]^{d^{\mathsf{V}}_n \otimes \mathrm{id}_{\mathsf{W}}} & G^n(\mathsf{V}) \otimes \mathsf{W}
} \]
is commutative, where $d^{\mathsf{V} \otimes \mathsf{W}}, d^{\mathsf{V}}$ are the bar differentials for $\mathsf{V} \otimes \mathsf{W}$ and $\mathsf{V}$ respectively. This follows from a straightforward computation using Lemma \ref{lemmaBarDifferentialZF} and is left to the reader.
\begin{figure}
\centering
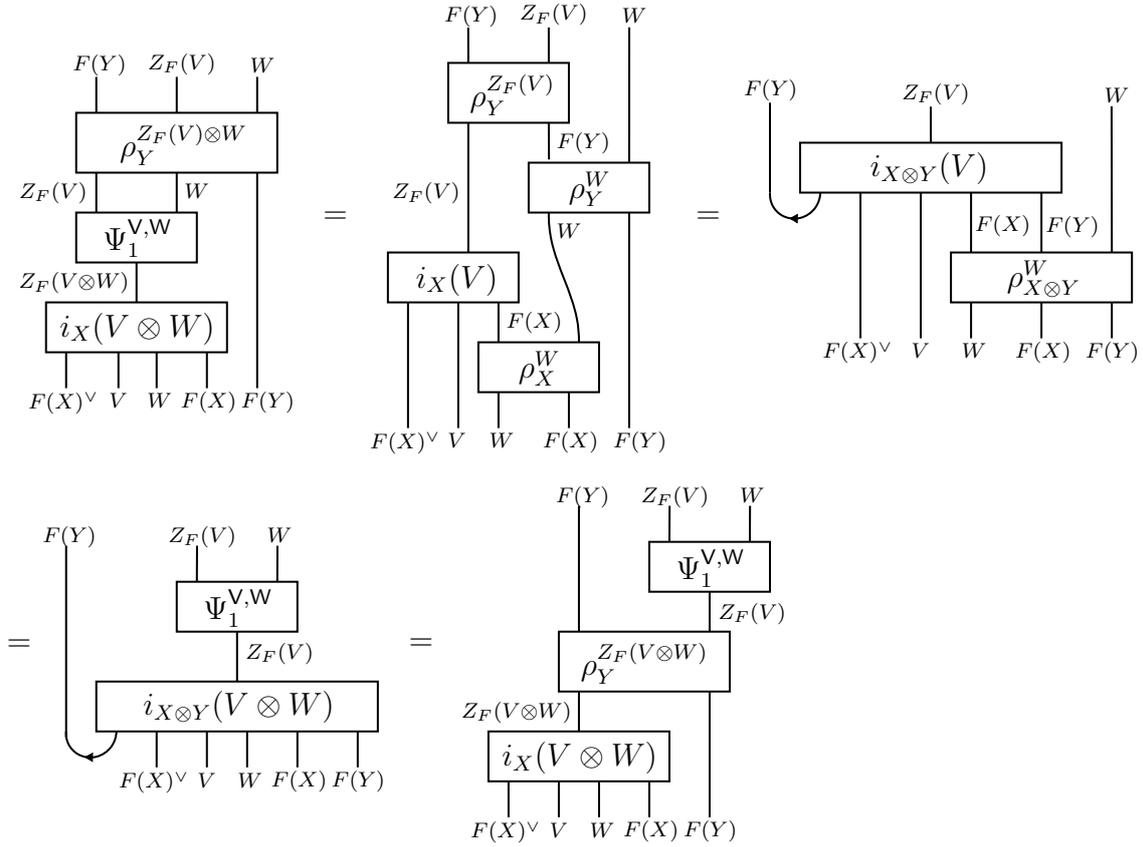
\caption{Proof of the commutation of the diagram \eqref{diagramPsiAndHalfBraidings}.}
\label{proofPsiCommutesHalfBraidings}
\end{figure}
\end{proof}

\begin{remark}
One can similarly construct isomorphisms $G^n(\mathsf{V} \otimes \mathsf{W}) \overset{\sim}{\to} \mathsf{V} \otimes G^n(\mathsf{W})$ which yield an isomorphism of complexes from $\mathrm{Bar}^{\bullet}_G(\mathsf{V} \otimes \mathsf{W})$ to $ \mathsf{V} \otimes \mathrm{Bar}^{\bullet}_G(\mathsf{W})$.
\end{remark}

\indent Recall that for coefficients $\mathsf{V}, \mathsf{W}$, the complex associated to the bar resolution of $\mathsf{V}$ is
\[ \mathrm{Bar}^n_{\mathcal{Z}(F), \mathcal{D}}(\mathsf{V}, \mathsf{W}) = \Hom_{\mathcal{Z}(F)}\!\big( \mathrm{Bar}^n_G(\mathsf{V}), \mathsf{W} \big) = \Hom_{\mathcal{Z}(F)}\!\big( G^{n+1}(\mathsf{V}), \mathsf{W} \big). \]
The isomorphism in Proposition \ref{propGVtensW} allows us to write down explicitly the first isomorphism from Corollary \ref{switchFormulaCoeffsRelExt} (and a similar description applies for the second). Indeed,
\[ \Hom_{\mathcal{Z}(F)}\!\big( G^{n+1}(\mathsf{V}), \mathsf{W} \big) \overset{\sim}{\longrightarrow} \Hom_{\mathcal{Z}(F)}\!\big( G^{n+1}(\mathsf{V}) \otimes {^{\vee}\mathsf{W}}, \boldsymbol{1} \big)  \xrightarrow{(\Psi^{\mathsf{V}, {^{\vee}\!\mathsf{W}}}_{n+1})^{*}} \Hom_{\mathcal{Z}(F)}\!\big( G^{n+1}(\mathsf{V} \otimes {^{\vee}\mathsf{W}}), \boldsymbol{1} \big) \]
is an isomorphism and due to the naturality of the first arrow and to Proposition \ref{propGVtensW} it yields an isomorphism of complexes
\begin{equation}\label{isoXiDualBar}
\flecheIso{\mathrm{Bar}^n_{\mathcal{Z}(F), \mathcal{D}}(\mathsf{V},\mathsf{W})}{\mathrm{Bar}^n_{\mathcal{Z}(F), \mathcal{D}}(\mathsf{V} \otimes {^{\vee}\mathsf{W}}, \boldsymbol{1})}{\alpha}{\widetilde{\mathrm{ev}}_W(\alpha \otimes \mathrm{id}_{^{\vee}\!W})\Psi_{n+1}^{\mathsf{V}, {^{\vee}\!\mathsf{W}}}}
\end{equation}
which in turn gives an isomorphism $\Ext^{\bullet}_{\mathcal{Z}(F), \mathcal{D}}(\mathsf{V},\mathsf{W}) \overset{\sim}{\to} \Ext^{\bullet}_{\mathcal{Z}(F), \mathcal{D}}(\mathsf{V} \otimes {^{\vee}\mathsf{W}}, \boldsymbol{1})$.

\subsection{Relation to Davydov-Yetter cohomology}\label{sectionRelationDYCohomology}
We recall the Davydov--Yetter (DY) cochain complex with coefficients, as it was introduced in \cite{GHS}:
\begin{itemize}
\item Let $\mathsf{U} = (U,\rho^U), \mathsf{V} = (V,\rho^V) \in \mathcal{Z}(F)$. The $n$-th Davydov--Yetter cochain space of $F$ with coefficients $\mathsf{U}, \mathsf{V}$, denoted by $C^n_{\mathrm{DY}}(F,\mathsf{U},\mathsf{V})$, is the space of all natural transformations $f$ of the form
\[ f_{X_1,\ldots,X_n} : U \otimes F(X_1) \otimes \ldots \otimes F(X_n) \to F(X_1) \otimes \ldots \otimes F(X_n) \otimes V. \]
The case $n=0$ is $C^0_{\mathrm{DY}}(F,\mathsf{U},\mathsf{V}) = \Hom_{\mathcal{D}}(U,V)$.
\item The Davydov--Yetter differential $\delta^n : C^n_{\mathrm{DY}}(F,\mathsf{U},\mathsf{V}) \to C^{n+1}_{\mathrm{DY}}(F,\mathsf{U},\mathsf{V})$ is defined by
\begin{equation}\label{defDYdifferential}
\begin{array}{rl}
\displaystyle \delta^n(f)_{X_0, \ldots, X_n} = &\displaystyle \!\!\!(\mathrm{id}_{F(X_0)} \otimes f_{X_1, \ldots, X_n})(\rho^U_{X_0} \otimes \mathrm{id}_{F(X_1) \, \otimes \ldots \otimes F(X_n)})\\[.3em]
&\displaystyle \!\!\!+ \: \sum_{i=1}^n (-1)^i f_{X_0, \ldots, X_{i-1} \otimes X_i, \ldots, X_n}\\[1.2em]
&\displaystyle \!\!\!+ \:(-1)^{n+1} (\mathrm{id}_{F(X_0) \otimes \ldots \otimes F(X_{n-1})} \otimes \rho^V_{X_n})(f_{X_0, \ldots, X_{n-1}} \otimes \mathrm{id}_{F(X_n)}).
\end{array}
\end{equation}
The case $n=0$ is $\delta^0(f)_{X_0} = (\mathrm{id}_{F(X_0)} \otimes f)\rho^U_{X_0} - \rho^V_{X_0}(f \otimes \mathrm{id}_{F(X_0)})$.
\item The $n$-th Davydov--Yetter cohomology\footnote{We will also use the shortand ``DY cohomology''.} group of $F$ with coefficients $\mathsf{U}, \mathsf{V}$, denoted by $H^n_{\mathrm{DY}}(F,\mathsf{U},\mathsf{V})$, is $\ker(\delta^n)/\mathrm{im}(\delta^{n-1})$. The case $n=0$ is $H^0_{\mathrm{DY}}(F,\mathsf{U},\mathsf{V}) = \ker(\delta^0)$.
\item If coefficients are trivial (\textit{i.e.}\,$\mathsf{U} = \mathsf{V} = \mathbf{1}$), then we write simply $C^n_{\mathrm{DY}}(F)$ and $H^n_{\mathrm{DY}}(F)$. If the functor $F$ is the identity functor $\mathrm{Id}_{\mathcal{C}}$, then we write $C^n_{\mathrm{DY}}(\mathcal{C},\mathsf{U},\mathsf{V})$ and $H^n_{\mathrm{DY}}(\mathcal{C},\mathsf{U},\mathsf{V})$. Finally, $C^n_{\mathrm{DY}}(\mathcal{C})$ and $H^n_{\mathrm{DY}}(\mathcal{C})$ mean respectively the DY cochains and DY cohomology of the identity functor with trivial coefficients.
\end{itemize}

\indent Let $G = \mathcal{F}\mathcal{U}$ be the comonad on $\mathcal{Z}(F)$ associated to the adjunction \eqref{relSystemZF} discussed in the previous subsection and recall that the bar complex of a comonad is defined in \eqref{barComplexCoefficients}. We have the key result:
\begin{theorem}\label{thGHS}{\em \cite[Thm.\,3.11]{GHS}}
Let $\mathsf{U} = (U,\rho^U), \mathsf{V} = (V,\rho^V) \in \mathcal{Z}(F)$. Under the previous assumptions, there is an isomorphism of cochain complexes
\[ C^{\bullet}_{\mathrm{DY}}(F,\mathsf{U},\mathsf{V}) \cong \mathrm{Bar}^{\bullet}_G\bigl(\mathsf{U},\mathrm{Hom}_{\mathcal{Z}(F)}(?,\mathsf{V})\bigr). \]
It follows that \[H^n_{\mathrm{DY}}(F,\mathsf{U},\mathsf{V}) \cong H^n_G\bigl(\mathsf{U},\mathrm{Hom}_{\mathcal{Z}(F)}(?,\mathsf{V})\bigr).\]
\end{theorem}
\noindent Now, since the adjunction \eqref{relSystemZF} is a resolvent pair of categories, this theorem can be restated as follows, thanks to Proposition \ref{relExtAndComonadCohom}:
\begin{corollary}\label{DYrelExt}
With the same notations, there is an isomorphism of cochain complexes
\[ C^{\bullet}_{\mathrm{DY}}(F,\mathsf{U},\mathsf{V}) \cong \mathrm{Bar}^{\bullet}_{\mathcal{Z}(F),\mathcal{D}}(\mathsf{U},\mathsf{V}). \]
It follows that \[H^n_{\mathrm{DY}}(F,\mathsf{U},\mathsf{V}) \cong \Ext^n_{\mathcal{Z}(F),\mathcal{D}}(\mathsf{U},\mathsf{V}).\]
\end{corollary}
\noindent Note that this corollary is an extended version of the Theorem \ref{thIntroIsoDYRelExt} stated in the Introduction.

\smallskip

\indent Since we will need it in the sequel, we recall from \cite{GHS} the explicit form of the isomorphism of cochain complexes in Corollary \ref{DYrelExt}, which we denote by $\Gamma$:
\begin{equation}\label{isoDYG}
\forall \, n, \:\: \Gamma : C^n_{\mathrm{DY}}(F,\mathsf{U},\mathsf{V}) \overset{\sim}{\to} \mathrm{Bar}^n_{\mathcal{Z}(F),\mathcal{D}}(\mathsf{U},\mathsf{V}) = \Hom_{\mathcal{Z}(F)}\bigl( G^{n+1}(\mathsf{U}), \mathsf{V} \bigr).
\end{equation}
For $f \in C^n_{\mathrm{DY}}(F,\mathsf{U},\mathsf{V})$, $\Gamma(f) : G^{n+1}(\mathsf{U}) \to \mathsf{V}$ is the unique morphism such that
\begin{equation}\label{defGamma}
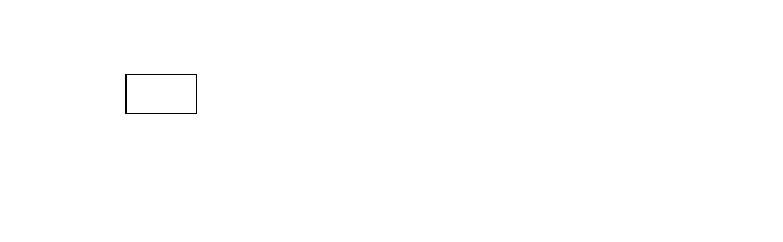
\end{equation}
where as before $\mathsf{U} = (U,\rho^U)$ and $\mathsf{V} = (V,\rho^V)$. Conversely, for $g \in \Hom_{\mathcal{Z}(F)}\bigl( G^{n+1}(\mathsf{U}), \mathsf{V} \bigr)$, the components of the natural transformation $\Gamma^{-1}(g)$ are given by
\begin{equation}\label{defGammaInverse}
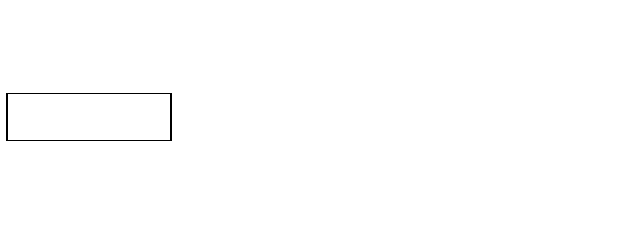
\end{equation}
where $\mathbf{1}$ is the tensor unit of $\mathcal{C}$. To show that $\Gamma$ and $\Gamma^{-1}$ are inverse to each other, one uses the following equality, which is an easy consequence of the definitions and which will also be used in the next section:
\begin{equation}\label{propertyHalfBraidingWihUnit}
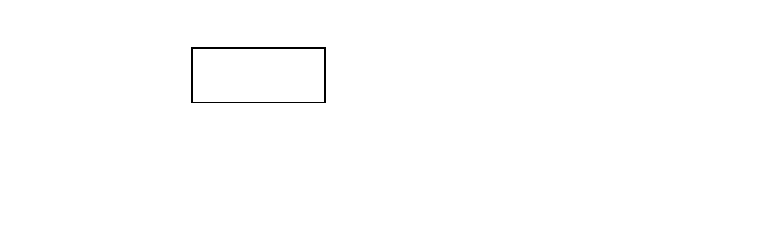
\end{equation}

\subsection{First consequences}\label{sectionFirstConsequences}
An easy consequence of Corollary \ref{DYrelExt} is:
\begin{proposition}\label{propH1Equals0}
Let $F : \mathcal{C} \to \mathcal{D}$ be an exact tensor functor. If $\mathrm{Ext}^1_{\mathcal{Z}(F)}(\boldsymbol{1},\boldsymbol{1}) = 0$ then $H^1_{\mathrm{DY}}(F) = 0$. In particular if the ground field $k$ of $\mathcal{C}$ and $\mathcal{D}$ has characteristic $0$ and is algebraically closed then $H^1_{\mathrm{DY}}(F) = 0$.
\end{proposition}
\begin{proof}
The first claim is due to \eqref{injectionRelExt1FullExt1}:
\[ H^1_{\mathrm{DY}}(F) \cong \Ext^1_{\mathcal{Z}(F),\mathcal{D}}(\boldsymbol{1},\boldsymbol{1}) \subset  \Ext^1_{\mathcal{Z}(F)}(\boldsymbol{1},\boldsymbol{1}). \]
The second claim follows from the fact that under these assumptions on $k$ the unit object $\boldsymbol{1}$ of a finite tensor category does not have non-trivial self-extensions \cite[Thm.\,4.4.1]{EGNO}.
\end{proof}
\begin{remark}\label{remarkFDerivations}
A cocycle in $C^1_{\mathrm{DY}}(F)$ is a $F$-derivation, that is an element $f \in \mathrm{Nat}(F,F)$ such that $f_{X \otimes Y} = f_X \otimes \mathrm{id}_{F(Y)} + \mathrm{id}_{F(X)} \otimes f_Y$. For trivial coefficients we have $\delta_0 = 0$, and thus $H^1_{\mathrm{DY}}(F) = \ker(\delta_1)$. Hence the proposition means that non-zero $F$-derivations do not exist under the above assumptions on the ground field $k$. Note that if $\mathcal{C} = H\text{-}\mathrm{mod}$ and if $F = U :\mathcal{C} \to \mathrm{Vect}_k$ is the forgetful functor it means that a finite-dimensional Hopf $k$-algebra $H$ does not have any non-zero primitive elements, thanks to the description of the DY complex for $U$ given in \eqref{DYForgetfulHModTrivialCoeffs} in section \ref{sectionFinDimHopfAlgebras} below. This fact is actually true for finite-dimensional bialgebras, see for example \cite[Cor.\,5.9.1]{EGNO}.
\end{remark}

\smallskip

\indent Next, combining Corollary \ref{DYrelExt} and Corollary \ref{switchFormulaCoeffsRelExt}, we obtain
\begin{equation}\label{AbstractIsosDualDY}
H^n_{\mathrm{DY}}(F; \mathsf{V}, \mathsf{W}) \cong H^n_{\mathrm{DY}}(F; \mathsf{V} \otimes {^{\vee}\mathsf{W}}, \boldsymbol{1}), \quad H^n_{\mathrm{DY}}(F; \mathsf{V}, \mathsf{W}) \cong H^n_{\mathrm{DY}}(F; \boldsymbol{1}, \mathsf{W} \otimes \mathsf{V}^{\vee})
\end{equation}
but even better we can transport the explicit isomorphism from \eqref{isoXiDualBar} through $\Gamma$, which defines
\begin{equation}\label{isoDualDY}
\forall \, n, \:\: \xi : C^n_{\mathrm{DY}}(F; \mathsf{V}, \mathsf{W}) \overset{\Gamma}{\longrightarrow} \mathrm{Bar}^n_{\mathcal{Z}(F), \mathcal{D}}(\mathsf{V},\mathsf{W}) \overset{\text{\eqref{isoXiDualBar}}}{\longrightarrow} \mathrm{Bar}^n_{\mathcal{Z}(F), \mathcal{D}}(\mathsf{V} \otimes {^{\vee}\mathsf{W}}, \boldsymbol{1}) \overset{\Gamma^{-1}}{\longrightarrow} C^n_{\mathrm{DY}}(F; \mathsf{V} \otimes {^{\vee}\mathsf{W}}, \boldsymbol{1}).
\end{equation}
Then for all $n$, $\xi$ descends to an isomorphism $\overline{\xi} : H^n_{\mathrm{DY}}(F; \mathsf{V}, \mathsf{W}) \cong H^n_{\mathrm{DY}}(F; \mathsf{V} \otimes {^{\vee}\mathsf{W}}, \boldsymbol{1})$. A similar description applies for the second isomorphism in \eqref{AbstractIsosDualDY}. A simple diagrammatic computation reveals that for $f \in C^n_{\mathrm{DY}}(F; \mathsf{V}, \mathsf{W})$, the components of the natural transformation $\xi(f)$ are
\begin{equation}\label{isoXiDualDY}
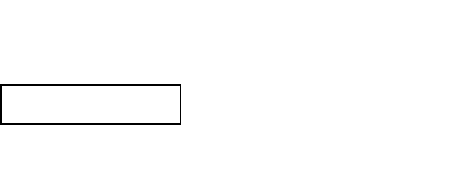
\end{equation}

\smallskip

\indent Finally, the dimension formulas from Corollaries \ref{CoroDimExtWithHom} and \ref{DimExt2WithHom} translate into dimension formulas for the DY cohomology groups:
\begin{corollary}\label{corollaryDimensionDYGroups}
Let \[ 0 \longrightarrow \mathsf{K} \overset{j}{\longrightarrow} \mathsf{P} \overset{\pi}{\longrightarrow} \mathbf{1} \longrightarrow 0, \qquad 0 \longrightarrow \mathsf{L} \overset{i}{\longrightarrow} \mathsf{Q} \overset{p}{\longrightarrow} \mathsf{V} \longrightarrow 0 \]
be allowable short exact sequences in $\mathcal{Z}(F)$ where $\mathsf{P}, \mathsf{Q}$ are relatively projective objects. Then for $n \geq 2$,
\[ \dim H_{\mathrm{DY}}^n(F;\mathsf{V},\mathsf{W}) = \dim \Hom_{\mathcal{Z}(F)}(\mathsf{K},\mathsf{M}) - \dim \Hom_{\mathcal{Z}(F)}(\mathsf{P},\mathsf{M}) + \dim \Hom_{\mathcal{Z}(F)}(\mathbf{1},\mathsf{M}) \]
where $\mathsf{M} = \mathsf{W} \otimes \mathsf{L}^{\vee} \otimes (\mathsf{K}^{\vee})^{\otimes (n-2)}$. In particular,
\begin{align*}
\dim H_{\mathrm{DY}}^n(F) = \dim \Hom_{\mathcal{Z}(F)}\bigl(\mathsf{K},(\mathsf{K}^{\vee})^{\otimes (n-1)}\bigr) &- \dim \Hom_{\mathcal{Z}(F)}\bigl(\mathsf{P},(\mathsf{K}^{\vee}\bigr)^{\otimes (n-1)}\bigr)\\
&+ \dim \Hom_{\mathcal{Z}(F)}\bigl(\mathbf{1},(\mathsf{K}^{\vee}\bigr)^{\otimes (n-1)}\bigr).
\end{align*}
If moreover $\mathsf{P}$ is the relatively projective cover (see Definition \ref{defRelProjCover}) of $\mathbf{1}$ and the ground field $k$ has characteristic $0$ and is algebraically closed, then
\[ \dim H_{\mathrm{DY}}^2(F) = \dim \Hom_{\mathcal{Z}(F)}\bigl(\mathsf{K},\mathsf{K}^{\vee}\bigr) - \dim \Hom_{\mathcal{Z}(F)}\bigl(\mathsf{P},\mathsf{K}^{\vee}\bigr). \]
\end{corollary}

\subsection{The Yoneda product on DY cohomology}\label{subsectionProductOnDY}
\indent As recalled in section \ref{relExtGroups}, there is the Yoneda product $\circ$ on relative Ext cohomology. Thanks to the isomorphism of cochain complexes $\Gamma$ in \eqref{isoDYG} we get a product on the DY side, which we still denote by $\circ$:
\begin{equation}\label{defYonedaProductDY}
f \circ g = \Gamma^{-1}\bigl( \Gamma(f) \circ \Gamma(g) \bigr).
\end{equation}
This is an associative product on DY cocycles which descends to the cohomology groups.

\smallskip

\indent From \eqref{IsoYExtExtOnCocycles} and \eqref{defYonedaProduct}, we know that any $n$-cocycle can be written as a Yoneda product of $n$ $1$-cocycles. Hence the same is true for DY cocycles, through the isomorphism $\Gamma$:
\begin{lemma}\label{factorization1CocyclesDY}
Let $f \in C^n_{\mathrm{DY}}(F,\mathsf{U},\mathsf{W})$ be a cocycle. Then there exists cocycles 
\[ g_1 \in C^1_{\mathrm{DY}}(F,\mathsf{V}_1,\mathsf{W}), \: g_2 \in C^1_{\mathrm{DY}}(F,\mathsf{V}_2,\mathsf{V}_1), \ldots, \: g_n \in C^1_{\mathrm{DY}}(F,\mathsf{U},\mathsf{V}_{n-1}) \]
such that $f = g_1 \circ \ldots \circ g_n$.
\end{lemma}

\smallskip

\indent The product \eqref{defYonedaProductDY} has the following simple expression:

\begin{theorem}\label{propYonedaProductDY}
Let $f \in C^n_{\mathrm{DY}}(F,\mathsf{V},\mathsf{W})$ and $g \in C^m_{\mathrm{DY}}(F,\mathsf{U},\mathsf{V})$ be cocycles. Then the components of the cocycle $f \circ g \in C^{m+n}_{\mathrm{DY}}(F,\mathsf{U},\mathsf{W})$ are
\[ (f \circ g)_{X_1, \ldots, X_m, Y_1, \ldots, Y_n} = (-1)^{nm}\bigl(\mathrm{id}_{F(X_1) \otimes \ldots \otimes F(X_m)} \otimes f_{Y_1,\ldots,Y_n}\bigr)\bigl(g_{X_1,\ldots,X_m} \otimes \mathrm{id}_{F(Y_1) \otimes \ldots \otimes F(Y_n)}\bigr). \]
\end{theorem}
\noindent The proof of this formula is divided in three steps:
\begin{itemize}
\item We first determine the Yoneda product of a $n$-cocycle $\alpha \in \mathrm{Bar}^n_{\mathcal{Z}(F),\mathcal{D}}(\mathsf{V},\mathsf{W})$, which is an element in $\Hom_{\mathcal{Z}(F)}\bigl( G^{n+1}(\mathsf{V}), \mathsf{W} \bigr)$, with a $1$-cocycle $\beta \in \mathrm{Bar}^1_{\mathcal{Z}(F),\mathcal{D}}(\mathsf{U},\mathsf{V}) = \Hom_{\mathcal{Z}(F)}\bigl( G^2(\mathsf{U}), \mathsf{V} \bigr)$. See Lemma \ref{lemmaYonedaProductBarGResolution}.
\item From this we compute the product \eqref{defYonedaProductDY} for a DY $n$-cocycle $f \in C^n_{\mathrm{DY}}(F,\mathsf{V},\mathsf{W})$ and a DY $1$-cocycle $g \in C^1_{\mathrm{DY}}(F,\mathsf{U},\mathsf{V})$. See Lemma \ref{lemmaDYYonedaProductFor1Cocycle}.
\item Then we deduce the product \eqref{defYonedaProductDY} for general DY cocycles $f,g$ by induction on $m$ thanks to Lemma \ref{factorization1CocyclesDY}.
\end{itemize}
Restricting $g$ to be a $1$-cocycle in the two first steps makes the proofs of the corresponding lemmas less cumbersome. We write as usual $\mathsf{U} = (U, \rho^U), \mathsf{V} = (V, \rho^V), \mathsf{W} = (W, \rho^W)$.

\begin{lemma}\label{lemmaYonedaProductBarGResolution}
Let $\alpha \in \mathrm{Bar}^n_{\mathcal{Z}(F),\mathcal{D}}(\mathsf{V},\mathsf{W}), \beta \in \mathrm{Bar}^1_{\mathcal{Z}(F),\mathcal{D}}(\mathsf{U},\mathsf{V})$ be cocycles. The Yoneda product $\alpha \circ \beta \in \mathrm{Bar}^{n+1}_{\mathcal{Z}(F),\mathcal{D}}(\mathsf{U},\mathsf{W})$ is the unique morphism such that
\[ (\alpha \circ \beta) \, i^{(n+2)}_{X,Y_1,\ldots, Y_{n+1}}(U) = (-1)^n \, \alpha \, i_{Y_1, \ldots, Y_{n+1}}^{(n+1)}(V)\bigl( \mathrm{id}_{F(Y_{n+1})^{\vee} \otimes \ldots \otimes F(Y_1)^{\vee}} \otimes \beta \, i^{(2)}_{X,\mathbf{1}}(U) \otimes \mathrm{id}_{F(Y_1) \otimes \ldots \otimes F(Y_{n+1})}\bigr) \]
where $\mathbf{1}$ is the tensor unit of $\mathcal{C}$.
\end{lemma}

\begin{proof}
Since we use the bar resolution, the diagram in \eqref{diagramLiftYoneda} becomes
\begin{equation}\label{diagramBarGResolutionForYoneda}
\xymatrix@C=2em@R=2em{
G^{n+2}(\mathsf{U}) \ar@{-->}[d]^{\widetilde{\beta}_n} \ar[r]^{\:\:\:\: d^{\mathsf{U}}_{n+1}} & \: \ldots \: \ar[r]^{\!\!\!\!\!\!\!\!\!d^{\mathsf{U}}_{l+3}} & G^{l+3}(\mathsf{U}) \ar@{-->}[d]^{\widetilde{\beta}_{l+1}} \ar[r]^{\!\!\!d^{\mathsf{U}}_{l+2}} & G^{l+2}(\mathsf{U}) \ar@{-->}[d]^{\widetilde{\beta}_l} \ar[r]^{\:\:\:\:d^{\mathsf{U}}_{l+1}} & \:\ldots \: \ar[r]^{\!\!\!\!\!\!d^{\mathsf{U}}_2} & G^{2}(\mathsf{U}) \ar@{-->}[d]^{\widetilde{\beta}_0} \ar[rd]^{\beta} & & \\
G^{n+1}(\mathsf{V}) \ar[r]^{\:\:\:\:\:\:d^{\mathsf{V}}_n} & \: \ldots \: \ar[r]^{\!\!\!\!\!\!\!\!\!d^{\mathsf{U}}_{l+2}} & G^{l+2}(\mathsf{V}) \ar[r]^{d^{\mathsf{V}}_{l+1}} & G^{l+1}(\mathsf{V}) \ar[r]^{\quad d^{\mathsf{V}}_l} & \: \ldots \: \ar[r]^{\!\!\!\!d^{\mathsf{V}}_1} & G(\mathsf{V}) \ar[r]^{\varepsilon_{\mathsf{V}}} & \mathsf{V} \ar[r] & 0
}
\end{equation}
where $G$ is the comonad on $\mathcal{Z}(F)$ defined in \eqref{defComonadOnZF}, $d^{\mathsf{U}}, d^{\mathsf{V}}$ are the differential of the bar resolutions of $\mathsf{U}, \mathsf{V}$ respectively and $\varepsilon_{\mathsf{V}}$ is the counit of $G$. We have to find the dashed arrows and the Yoneda product $\alpha \circ \beta$ is then $\alpha \widetilde{\beta}_n$. Define $\widetilde{\beta}_l$ as the unique morphism such that
\[ \widetilde{\beta}_l \, i^{(l+2)}_{X,Y_1,\ldots, Y_{l+1}}(U) = (-1)^l \, i_{Y_1, \ldots, Y_{l+1}}^{(l+1)}(V)\bigl( \mathrm{id}_{F(Y_{l+1})^{\vee} \otimes \ldots \otimes F(Y_1)^{\vee}} \otimes \beta \, i^{(2)}_{X,\mathbf{1}}(U) \otimes \mathrm{id}_{F(Y_1) \otimes \ldots \otimes F(Y_{l+1})}\bigr). \]
We left to the reader to check that $\widetilde{\beta}_l$ is a morphism in $\mathcal{Z}(F)$, \textit{i.e.}\,that it commutes with the half-braidings relative to $F$. We first show that the triangle in \eqref{diagramBarGResolutionForYoneda} commutes:
\begin{center}
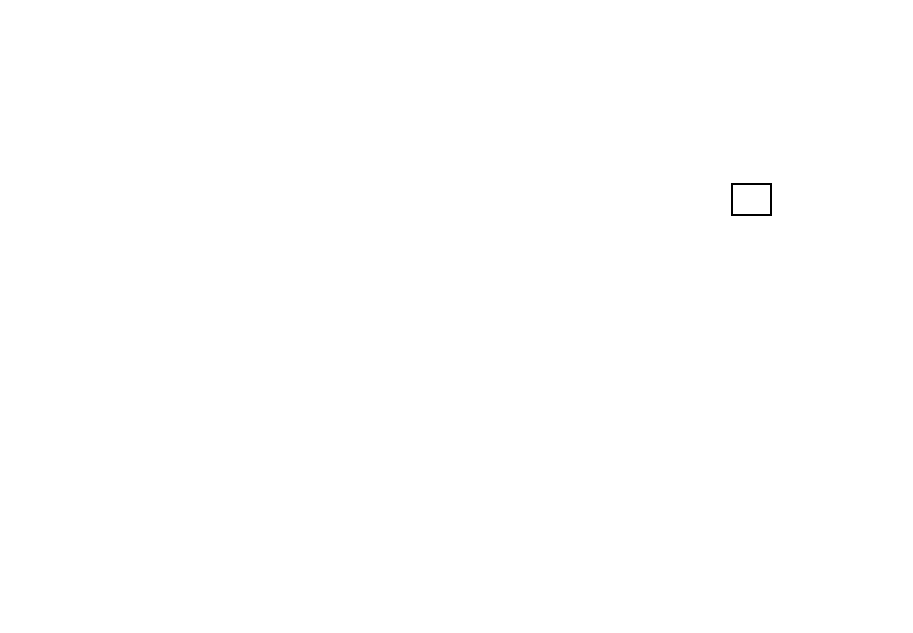
\end{center}
and the claim follows since this holds for any $X,Y \in \mathcal{C}$. The first and second equalities are just the definitions of $\widetilde{\beta}_0$ and $\varepsilon_{\mathsf{V}}$ (recall \eqref{defCounitG}), for the third equality we used that $\beta$ is a morphism in $\mathcal{Z}(F)$ and for the fourth equality we used \eqref{propertyHalfBraidingWihUnit}. 

Now we show that the squares in \eqref{diagramBarGResolutionForYoneda} commute. Recall the expression of the bar differential in Lemma \ref{lemmaBarDifferentialZF}. On the one hand we have

\begin{align*}
&\widetilde{\beta}_l \, d_{l+2}^{\mathsf{U}} \, i^{(l+3)}_{X,Y_1, \ldots, Y_{l+2}} = \widetilde{\beta}_l \, i^{(l+2)}_{Y_1, \ldots, Y_{l+2}}(U) \bigl( \mathrm{id}_{F(Y_{l+2})^{\vee} \otimes \ldots \otimes F(Y_1)^{\vee}} \otimes \varepsilon_{\mathsf{U}} \, i_X(U) \otimes \mathrm{id}_{F(Y_1) \otimes \ldots \otimes F(Y_{l+2})}\bigr)\\
& \qquad \qquad \qquad \qquad \:\: - \widetilde{\beta}_l \, i^{(l+2)}_{X \otimes Y_1, Y_2, \ldots, Y_{l+2}}(V) + \sum_{j=2}^{l+2} (-1)^j \, \widetilde{\beta}_l \, i^{(l+2)}_{X, Y_1, \ldots, Y_{j-1} \otimes Y_j, \ldots, Y_{l+2}}(V)\\
=\: & (-1)^l \, i_{Y_2, \ldots, Y_{l+2}}^{(l+1)}(V)\bigg( \mathrm{id}_{F(Y_{l+2})^{\vee} \otimes \ldots \otimes F(Y_2)^{\vee}} \otimes \beta \, i^{(2)}_{Y_1,\mathbf{1}}(U) \bigl( \mathrm{id}_{F(Y_1)^{\vee}} \otimes \varepsilon_{\mathsf{U}} \, i_X(U) \otimes \mathrm{id}_{F(Y_1)} \bigr)\\
&\hphantom{(-1)^l \, i_{Y_2, \ldots, Y_{l+2}}^{(l+1)}(V)\bigg( \mathrm{id}_{F(Y_{l+2})^{\vee} \otimes \ldots \otimes F(Y_2)^{\vee}}}\qquad\qquad\qquad\qquad\qquad\qquad\qquad\qquad \otimes \mathrm{id}_{F(Y_2) \otimes \ldots \otimes F(Y_{l+2})}\bigg)\\
&- (-1)^l \, i_{Y_2, \ldots, Y_{l+2}}^{(l+1)}(V)\bigl( \mathrm{id}_{F(Y_{l+2})^{\vee} \otimes \ldots \otimes F(Y_2)^{\vee}} \otimes \beta \, i^{(2)}_{X \otimes Y_1,\mathbf{1}}(U) \otimes \mathrm{id}_{F(Y_2) \otimes \ldots \otimes F(Y_{l+2})}\bigr)\\
& + (-1)^l \, \sum_{j=2}^{l+2} (-1)^j \, i^{(l+1)}_{Y_1, \ldots, Y_{j-1} \otimes Y_j, \ldots, Y_{l+2}}(V) \bigl( \mathrm{id}_{F(Y_{l+2})^{\vee} \otimes \ldots \otimes F(Y_1)^{\vee}} \otimes \beta \, i^{(2)}_{X,\mathbf{1}}(U) \otimes \mathrm{id}_{F(Y_1) \otimes \ldots \otimes F(Y_{l+2})}\bigr)\\
= \:& (-1)^{l+1} \, i_{Y_2, \ldots, Y_{l+2}}^{(l+1)}(V)\bigl( \mathrm{id}_{F(Y_{l+2})^{\vee} \otimes \ldots \otimes F(Y_2)^{\vee}} \otimes \beta \, i^{(2)}_{X, Y_1}(U) \otimes \mathrm{id}_{F(Y_2) \otimes \ldots \otimes F(Y_{l+2})}\bigr)\\
& + (-1)^l \sum_{j=2}^{l+2} (-1)^j \, i^{(l+1)}_{Y_1, \ldots, Y_{j-1} \otimes Y_j, \ldots, Y_{l+2}}(V) \bigl( \mathrm{id}_{F(Y_{l+2})^{\vee} \otimes \ldots \otimes F(Y_1)^{\vee}} \otimes \beta \, i^{(2)}_{X,\mathbf{1}}(U) \otimes \mathrm{id}_{F(Y_1) \otimes \ldots \otimes F(Y_{l+2})}\bigr).
\end{align*}
For the last equality we used $\beta i^{(2)}_{Y_1,\mathbf{1}} \bigl( \mathrm{id}_{F(Y_1)^{\vee}} \otimes \varepsilon_{\mathsf{U}} i_X(U) \otimes \mathrm{id}_{F(Y_1)} \bigr) - \beta i^{(2)}_{X \otimes Y_1,\mathbf{1}}(U) = - \beta i_{X,Y_1}^{(2)}(U)$, due to the assumption that $\beta$ is a cocycle ($\beta d_2^{\mathsf{U}} = 0$). On the other hand and still using Lemma \ref{lemmaBarDifferentialZF} we get
\begin{align*}
&d_{l+1}^{\mathsf{V}} \, \widetilde{\beta}_{l+1} \, i^{(l+3)}_{X,Y_1, \ldots, Y_{l+2}} = (-1)^{l+1} \, d_{l+1}^{\mathsf{V}} \, i_{Y_1, \ldots, Y_{l+2}}^{(l+2)}(V)\bigl( \mathrm{id}_{F(Y_{l+2})^{\vee} \otimes \ldots \otimes F(Y_1)^{\vee}} \otimes \beta \, i^{(2)}_{X,\mathbf{1}}(U) \otimes \mathrm{id}_{F(Y_1) \otimes \ldots \otimes F(Y_{l+2})}\bigr)\\
=\: &  (-1)^{l+1} \, i^{(n)}_{Y_2, \ldots, Y_{l+2}}(V) \bigg( \mathrm{id}_{F(Y_{l+2})^{\vee} \otimes \ldots \otimes F(Y_2)^{\vee}} \otimes \varepsilon_{\mathsf{V}}i_{Y_1}(V) \bigl( \mathrm{id}_{F(Y_1)^{\vee}} \otimes \beta \, i^{(2)}_{X,\mathbf{1}}(U) \otimes \mathrm{id}_{F(Y_1)}\bigr)\\
& \hphantom{(-1)^{l+1} \, i^{(n)}_{Y_2, \ldots, Y_{l+2}}(V) \bigg( \mathrm{id}_{F(Y_{l+2})^{\vee} \otimes \ldots \otimes F(Y_2)^{\vee}}} \qquad\qquad\qquad\qquad\qquad\qquad\qquad\qquad\otimes \mathrm{id}_{F(Y_2) \otimes \ldots \otimes F(Y_{l+2})}\bigg)\\
& + (-1)^{l+1} \, \sum_{j=1}^{l+1} (-1)^j \, i^{(l+1)}_{Y_1, \ldots, Y_j \otimes Y_{j+1}, \ldots, Y_{n+1}}(V) \bigl( \mathrm{id}_{F(Y_{l+2})^{\vee} \otimes \ldots \otimes F(Y_1)^{\vee}} \otimes \beta \, i^{(2)}_{X,\mathbf{1}}(U) \otimes \mathrm{id}_{F(Y_1) \otimes \ldots \otimes F(Y_{l+2})}\bigr)\\
=\: & (-1)^{l+1} \, i_{Y_2, \ldots, Y_{l+2}}^{(l+1)}(V)\bigl( \mathrm{id}_{F(Y_{l+2})^{\vee} \otimes \ldots \otimes F(Y_2)^{\vee}} \otimes \beta \, i^{(2)}_{X, Y_1}(U) \otimes \mathrm{id}_{F(Y_2) \otimes \ldots \otimes F(Y_{l+2})}\bigr)\\
& +(-1)^l \sum_{j=2}^{l+2} (-1)^j \, i^{(l+1)}_{Y_1, \ldots, Y_{j-1} \otimes Y_j, \ldots, Y_{n+1}}(V) \bigl( \mathrm{id}_{F(Y_{l+2})^{\vee} \otimes \ldots \otimes F(Y_1)^{\vee}} \otimes \beta \, i^{(2)}_{X,\mathbf{1}}(U) \otimes \mathrm{id}_{F(Y_1) \otimes \ldots \otimes F(Y_{l+2})}\bigr)
\end{align*}
which agrees with the result of the previous computation; since this holds for any $X, Y_1, \ldots, Y_{l+2} \in \mathcal{C}$, the desired equality follows. Note that for the last equality we used that
\[ \varepsilon_{\mathsf{V}} \, i_{Y_1}(V) \, \bigl( \mathrm{id}_{F(Y_1)^{\vee}} \otimes \beta \, i^{(2)}_{X,\mathbf{1}}(U) \otimes \mathrm{id}_{F(Y_1)}\bigr) = \varepsilon_{\mathsf{V}} \, \widetilde{\beta}_0 \, i^{(2)}_{X,Y_1}(U) = \beta \, i^{(2)}_{X,Y_1}(U) \]
due to the commutation of the triangle in \eqref{diagramBarGResolutionForYoneda}.
\end{proof}

\begin{figure}[h]
\centering
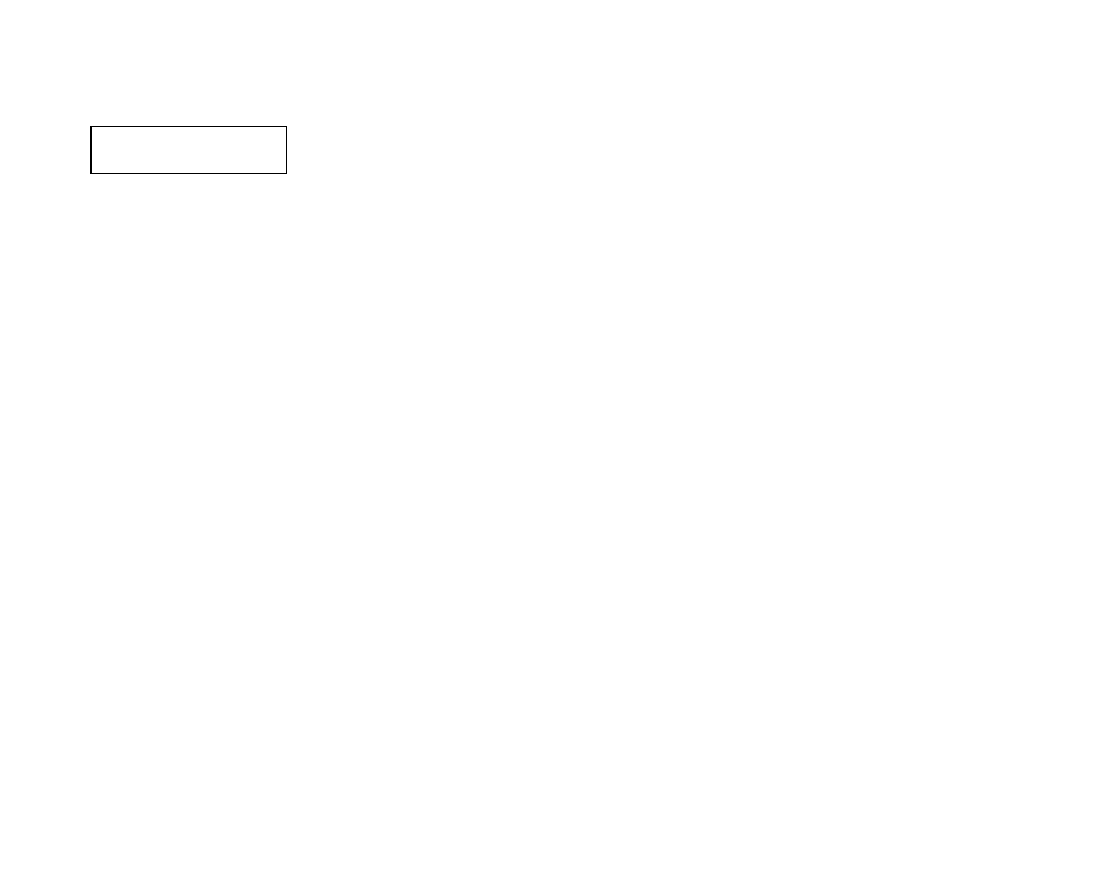
\caption{Proof of Lemma \ref{lemmaDYYonedaProductFor1Cocycle}}
\label{proofYonedaProduct1Cocycle}
\end{figure}
\begin{lemma}\label{lemmaDYYonedaProductFor1Cocycle}
Let $f \in C^n_{\mathrm{DY}}(F,\mathsf{V},\mathsf{W})$ and $g \in C^1_{\mathrm{DY}}(F,\mathsf{U},\mathsf{V})$ be cocycles. Then it holds
\[ (f \circ g)_{X, Y_1, \ldots, Y_n} = (-1)^n \, \bigl(\mathrm{id}_{F(X)} \otimes f_{Y_1,\ldots,Y_n}\bigr)\bigl(g_X \otimes \mathrm{id}_{F(Y_1) \otimes \ldots \otimes F(Y_n)}\bigr). \]
\end{lemma}
\begin{proof}
The proof is the diagrammatic computation displayed in Figure \ref{proofYonedaProduct1Cocycle}. We use the previous lemma, the formulas for $\Gamma$ in \eqref{defGamma} and its inverse in \eqref{defGammaInverse}, and the fact that $\rho_{\mathbf{1}}^V = \mathrm{id}_V$.
\end{proof}

\begin{proof}[Proof of Theorem \ref{propYonedaProductDY}]
We will conclude by induction on $m$ thanks to the previous lemma. Due to Lemma \ref{factorization1CocyclesDY}, we can write $g$ as $g' \circ h$ with $g' \in C^{m-1}_{\mathrm{DY}}(F,\mathsf{U'},\mathsf{V})$ and $h \in C^1_{\mathrm{DY}}(F,\mathsf{U},\mathsf{U}')$. Then by associativity of $\circ$ we get
\begin{align*}
&(f \circ (g' \circ h))_{X_1, \ldots, X_m, Y_1, \ldots, Y_n} = \bigl( (f \circ g') \circ h \bigr)_{X_1, X_2, \ldots, X_m, Y_1, \ldots, Y_n}\\
= \: & (-1)^{n+m-1} \bigl( \mathrm{id}_{F(X_1)} \otimes (f \circ g')_{X_2, \ldots, X_m, Y_1, \ldots, Y_n} \bigr) \bigl( h_{X_1} \otimes \mathrm{id}_{F(X_2) \otimes \ldots \otimes F(X_m) \otimes F(Y_1) \otimes \ldots \otimes F(Y_n)} \bigr)\\
= \: & (-1)^{nm + m - 1}\bigl( \mathrm{id}_{F(X_1) \otimes \ldots \otimes F(X_m)} \otimes f_{Y_1, \ldots, Y_n} \bigr) \bigl( \mathrm{id}_{F(X_1)} \otimes g'_{X_2, \ldots, X_m} \otimes \mathrm{id}_{F(Y_1) \otimes \ldots \otimes F(Y_n)} \bigr)\\
&\qquad\qquad\qquad\qquad\qquad\qquad\qquad\qquad\quad\:\:\: \bigl( h_{X_1} \otimes \mathrm{id}_{F(X_2) \otimes \ldots \otimes F(X_m) \otimes F(Y_1) \otimes \ldots \otimes F(Y_n)} \bigr)\\
= \: & (-1)^{nm}\bigl( \mathrm{id}_{F(X_1) \otimes \ldots \otimes F(X_m)} \otimes f_{Y_1, \ldots, Y_n} \bigr) \bigl( (g' \circ h)_{X_1, \ldots, X_m} \otimes \mathrm{id}_{F(Y_1) \otimes \ldots \otimes F(Y_n)} \bigr). \qedhere
\end{align*}
\end{proof}

\smallskip

\indent For $\mathsf{U} = (U, \rho^U), \mathsf{V} = (V, \rho^V) \in \mathcal{Z}(F)$, we write as usual
\[ C^{\bullet}_{\mathrm{DY}}(F; \mathsf{U}, \mathsf{V}) = \bigoplus_{n=0}^{\infty} C^n_{\mathrm{DY}}(F; \mathsf{U}, \mathsf{V}), \quad H^{\bullet}_{\mathrm{DY}}(F; \mathsf{U}, \mathsf{V}) = \bigoplus_{n=0}^{\infty} H^n_{\mathrm{DY}}(F; \mathsf{U}, \mathsf{V}). \]
Then $C^{\bullet}_{\mathrm{DY}}(F) = C^{\bullet}_{\mathrm{DY}}(F; \boldsymbol{1}, \boldsymbol{1})$ and $H^{\bullet}_{\mathrm{DY}}(F) = H^{\bullet}_{\mathrm{DY}}(F; \boldsymbol{1}, \boldsymbol{1})$ are associative graded $k$-algebras with respect to the product $\circ$. In \cite[eq.\,(3)]{DE} and \cite[\S 3.1]{BD}, a product $\cup$ is defined by
\[ (f \cup g)_{X_1, \ldots, X_n, Y_1, \ldots, Y_m} = (-1)^{(n-1)(m-1)} f_{X_1, \ldots, X_n} \otimes g_{Y_1, \ldots, Y_m} \]
for $f \in C^n_{\mathrm{DY}}(F)$, $g \in C^m_{\mathrm{DY}}(F)$ and is shown to be commutative (up to sign) on $H^{\bullet}_{\mathrm{DY}}(F)$: $f \cup g = (-1)^{mn}g \cup f$. Thanks to Theorem \ref{propYonedaProductDY}, we see that $f \cup g = (-1)^{m+n+1}g \circ f$, and thus $\circ$ is commutative (up to sign) on $H^{\bullet}_{\mathrm{DY}}(F)$ as well. 

\smallskip

\indent Recall the isomorphism $\xi$ in \eqref{isoDualDY}; for any $\mathsf{U}, \mathsf{V} \in \mathcal{Z}(F)$ we define the map
\begin{align}
\triangleright : C^n_{\mathrm{DY}}(F) \otimes C^m_{\mathrm{DY}}(F; \mathsf{U}, \mathsf{V}) &\xrightarrow{\mathrm{id} \otimes \xi} C^n_{\mathrm{DY}}(F; \boldsymbol{1}, \boldsymbol{1}) \otimes C^m_{\mathrm{DY}}(F; \mathsf{U} \otimes {^{\vee}\mathsf{V}}, \boldsymbol{1})\nonumber\\
&\overset{\circ}{\longrightarrow} C^{n+m}_{\mathrm{DY}}(F; \mathsf{U} \otimes {^{\vee}\mathsf{V}}, \boldsymbol{1}) \overset{\!\xi^{-1}}{\longrightarrow} C^{n+m}_{\mathrm{DY}}(F; \mathsf{U}, \mathsf{V}) \label{defActionTriangleYoneda}
\end{align}

\begin{proposition}\label{propositionDYGroupsAreModules}
For $f \in C^n_{\mathrm{DY}}(F)$ and $g \in C^m_{\mathrm{DY}}(F; \mathsf{U}, \mathsf{V})$, the components of $f \triangleright g \in C^{n+m}_{\mathrm{DY}}(F; \mathsf{U}, \mathsf{V})$ are
\[ (f \triangleright g)_{X_1, \ldots, X_m, Y_1, \ldots, Y_n} = (-1)^{mn} \big(\mathrm{id}_{F(X_1) \otimes \ldots \otimes F(X_m)} \otimes \rho^V_{Y_1 \otimes \ldots \otimes Y_n}\big) \big( g_{X_1, \ldots, X_m} \otimes f_{Y_1, \ldots, Y_n} \big) \]
This endows $C^{\bullet}_{\mathrm{DY}}(F; \mathsf{U}, \mathsf{V})$ with a structure of graded $C^{\bullet}_{\mathrm{DY}}(F)$-module which descends to a structure of graded $H^{\bullet}_{\mathrm{DY}}(F)$-module on $H^{\bullet}_{\mathrm{DY}}(F; \mathsf{U}, \mathsf{V})$.
\end{proposition}
\begin{proof}
For simplicity of notation, take for instance $f$ to be a $1$-cochain (for the general case simply replace $Y$ by a tensor product $Y_1 \otimes \ldots \otimes Y_n$ in the pictures below). We have:
\begin{center}
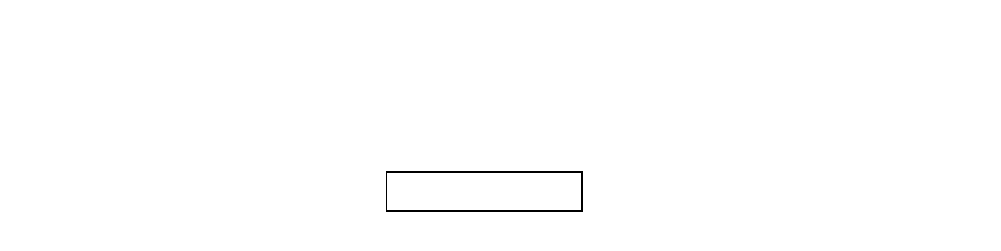
\end{center}
where the first equality is the output of a straightforward diagrammatic computation using \eqref{isoXiDualDY} and the formula in Theorem \ref{propYonedaProductDY}, while for the second we used \eqref{halfBraidingOnTensorProducts} and the naturality of the half-braiding. To obtain the desired expression, observe that
\begin{center}
\begingroup%
  \makeatletter%
  \providecommand\color[2][]{%
    \errmessage{(Inkscape) Color is used for the text in Inkscape, but the package 'color.sty' is not loaded}%
    \renewcommand\color[2][]{}%
  }%
  \providecommand\transparent[1]{%
    \errmessage{(Inkscape) Transparency is used (non-zero) for the text in Inkscape, but the package 'transparent.sty' is not loaded}%
    \renewcommand\transparent[1]{}%
  }%
  \providecommand\rotatebox[2]{#2}%
  \newcommand*\fsize{\dimexpr\f@size pt\relax}%
  \newcommand*\lineheight[1]{\fontsize{\fsize}{#1\fsize}\selectfont}%
  \ifx\svgwidth\undefined%
    \setlength{\unitlength}{139.02786254bp}%
    \ifx\svgscale\undefined%
      \relax%
    \else%
      \setlength{\unitlength}{\unitlength * \real{\svgscale}}%
    \fi%
  \else%
    \setlength{\unitlength}{\svgwidth}%
  \fi%
  \global\let\svgwidth\undefined%
  \global\let\svgscale\undefined%
  \makeatother%
  \begin{picture}(1,0.40314401)%
    \lineheight{1}%
    \setlength\tabcolsep{0pt}%
    \put(0,0){\includegraphics[width=\unitlength,page=1]{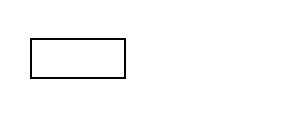}}%
    \put(0.13002703,0.17058201){\color[rgb]{0,0,0}\makebox(0,0)[lt]{\lineheight{1.25}\smash{\begin{tabular}[t]{l}$(\rho_Y^{^{\vee}\!V})^{-1}$\end{tabular}}}}%
    \put(0,0){\includegraphics[width=\unitlength,page=2]{halfBraidingDualInverse.pdf}}%
    \put(-0.00123619,0.00786776){\color[rgb]{0,0,0}\makebox(0,0)[lt]{\lineheight{1.25}\smash{\begin{tabular}[t]{l}$_{V}$\end{tabular}}}}%
    \put(0,0){\includegraphics[width=\unitlength,page=3]{halfBraidingDualInverse.pdf}}%
    \put(0.31592416,0.3784828){\color[rgb]{0,0,0}\makebox(0,0)[lt]{\lineheight{1.25}\smash{\begin{tabular}[t]{l}$_{F(Y)}$\end{tabular}}}}%
    \put(0,0){\includegraphics[width=\unitlength,page=4]{halfBraidingDualInverse.pdf}}%
    \put(0.09164922,0.00765505){\color[rgb]{0,0,0}\makebox(0,0)[lt]{\lineheight{1.25}\smash{\begin{tabular}[t]{l}$_{F(Y)}$\end{tabular}}}}%
    \put(0,0){\includegraphics[width=\unitlength,page=5]{halfBraidingDualInverse.pdf}}%
    \put(0.49913678,0.37760903){\color[rgb]{0,0,0}\makebox(0,0)[lt]{\lineheight{1.25}\smash{\begin{tabular}[t]{l}$_{V}$\end{tabular}}}}%
    \put(0,0){\includegraphics[width=\unitlength,page=6]{halfBraidingDualInverse.pdf}}%
    \put(0.81687552,0.1710639){\color[rgb]{0,0,0}\makebox(0,0)[lt]{\lineheight{1.25}\smash{\begin{tabular}[t]{l}$\rho_Y^V$\end{tabular}}}}%
    \put(0,0){\includegraphics[width=\unitlength,page=7]{halfBraidingDualInverse.pdf}}%
    \put(0.74749231,0.37848303){\color[rgb]{0,0,0}\makebox(0,0)[lt]{\lineheight{1.25}\smash{\begin{tabular}[t]{l}$_{F(Y)}$\end{tabular}}}}%
    \put(0,0){\includegraphics[width=\unitlength,page=8]{halfBraidingDualInverse.pdf}}%
    \put(0.93070498,0.37760926){\color[rgb]{0,0,0}\makebox(0,0)[lt]{\lineheight{1.25}\smash{\begin{tabular}[t]{l}$_{V}$\end{tabular}}}}%
    \put(0,0){\includegraphics[width=\unitlength,page=9]{halfBraidingDualInverse.pdf}}%
    \put(0.78098115,0.00786776){\color[rgb]{0,0,0}\makebox(0,0)[lt]{\lineheight{1.25}\smash{\begin{tabular}[t]{l}$_{V}$\end{tabular}}}}%
    \put(0,0){\includegraphics[width=\unitlength,page=10]{halfBraidingDualInverse.pdf}}%
    \put(0.87386654,0.00765505){\color[rgb]{0,0,0}\makebox(0,0)[lt]{\lineheight{1.25}\smash{\begin{tabular}[t]{l}$_{F(Y)}$\end{tabular}}}}%
    \put(0.57337763,0.18308139){\color[rgb]{0,0,0}\makebox(0,0)[lt]{\lineheight{1.25}\smash{\begin{tabular}[t]{l}$=$\end{tabular}}}}%
  \end{picture}%
\endgroup%

\end{center}
Indeed:
\begin{center}
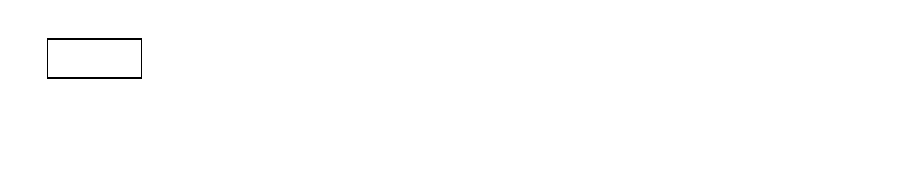
\end{center}
where the first equality is by definition of $\rho^{^{\vee}V}$ \eqref{defDualityZF}, for the second equality we used the half-braiding property \eqref{halfBraidingOnTensorProducts} and the third equality is by naturality of the half-braiding. 
With the explicit expression of $\triangleright$ just obtained, it is easy to check that $(f \circ f') \triangleright g = f \triangleright (f' \triangleright g)$. Moreover, $\xi$ and $\circ$ descend to cohomology, so that the representation $\triangleright$ descends to cohomology as well.
\end{proof}

\subsection{The long exact sequence for DY cohomology}\label{subsectionLongExactSequenceDYCohom}
We can transport the long exact sequence of relative $\Ext$ groups from Theorem \ref{thLongExactSequenceExtGroups} through the isomorphism $\Gamma$ from \eqref{defGamma} in order to obtain a long exact sequence for the Davydov--Yetter cohomology groups, which we now describe.

\smallskip

\indent Let $f : \mathsf{U} \to \mathsf{V}$ be a morphism in $\mathcal{Z}(F)$. It induces linear maps
\[ \forall \, n, \:\: f^* : \mathrm{Bar}^n_{\mathcal{Z}(F),\mathcal{D}}(\mathsf{V},\mathsf{W}) \to \mathrm{Bar}^n_{\mathcal{Z}(F),\mathcal{D}}(\mathsf{U},\mathsf{W}) \]
as defined at the beginning of \S \ref{sectionLongExactSequenceExtGroups}. Here we choose the bar resolution $\ldots \to G^2(\mathsf{X}) \to G(\mathsf{X}) \to \mathsf{X} \to 0$, which yields the bar complex $\mathrm{Bar}^n_{\mathcal{Z}(F),\mathcal{D}}(\mathsf{X},\mathsf{W}) = \Hom_{\mathcal{Z}(F)}\bigl( G^{n+1}(\mathsf{X}),\mathsf{W} \bigr)$, where $\mathsf{X}$ is $\mathsf{U}$ or $\mathsf{V}$ and $G$ is the comonad on $\mathcal{Z}(F)$ associated to the adjunction \eqref{relSystemZF}. Thanks to the isomorphism $\Gamma : C^n_{\mathrm{DY}}(F;\mathsf{X},\mathsf{W}) \to \mathrm{Bar}^n_{\mathcal{Z}(F),\mathcal{D}}(\mathsf{X},\mathsf{W})$, we have the corresponding linear maps $\Gamma^{-1}f^*\Gamma$ on DY cochains, which we also denote by $f^*$ for brevity:
\[ \forall \, n, \:\: f^* : C^n_{\mathrm{DY}}(F;\mathsf{V},\mathsf{W}) \to C^n_{\mathrm{DY}}(F;\mathsf{U},\mathsf{W}). \]

\begin{lemma}\label{pullbackDY}
For $g \in C^n_{\mathrm{DY}}(F;\mathsf{V},\mathsf{W})$, the components of $f^*(g) \in C^n_{\mathrm{DY}}(F;\mathsf{U},\mathsf{W})$ are
\[ f^*(g)_{X_1, \ldots, X_n} = g_{X_1, \ldots, X_n} \, \bigl( f \otimes \mathrm{id}_{F(X_1) \otimes \ldots \otimes F(X_n)} \bigr). \]
\end{lemma}
\begin{proof}
First note a general fact. Take a resolvent pair of abelian categories as in \eqref{adjunction} and let $G$ be the associated comonad on $\mathcal{A}$. Then for any objects $U, V \in \mathcal{A}$, the diagram
\[ \xymatrix{
G^{n+1}(U) \ar[r]^{d_n^U} \ar@{-->}[d]^{G^{n+1}(f)} & G^n(U) \ar[r]^{d_{n-1}^U} \ar@{-->}[d]^{G^n(f)} & \ldots \ar[r]^{\!\!\!\!\!\!d_1^U} & G(U) \ar[r]^{\quad d_0^U} \ar@{-->}[d]^{G(f)}& U \ar[r] \ar[d]^f & 0\\
G^{n+1}(V) \ar[r]_{d_n^V} & G^n(V) \ar[r]_{d_{n-1}^V} & \ldots \ar[r]_{\!\!\!\!\!\!d_1^V} & G(V) \ar[r]_{\quad d_0^V} & V \ar[r] & 0
} \]
commutes. The first (resp. second) row is the bar resolution of $U$ (resp. $V$), whose differential is defined in \eqref{barDifferentialG}. This is based on an easy computation using the naturality of the counit $\varepsilon : G \to \mathrm{Id}$. Then by definition (see \S \ref{sectionLongExactSequenceExtGroups}), $f^* : \Hom_{\mathcal{A}}\bigl( G^{n+1}(V),W \bigl) \to \Hom_{\mathcal{A}}\bigl( G^{n+1}(U),W \bigl)$ is simply given by $f^*(\alpha) = \alpha G^{n+1}(f)$.
\\In our case, $G^{n+1}(f) = Z_F^{n+1}(f)$ and the result is the outcome of an easy diagrammatic computation using the definition of $\Gamma$ and $\Gamma^{-1}$ from \eqref{defGamma}--\eqref{defGammaInverse} and the definition of $Z_F(f)$ from \eqref{defZFf}.
\end{proof}

\begin{corollary}\label{coroLongExactSequenceDYCohomology}
Let $S = \big(0 \longrightarrow \mathsf{U} \overset{j}{\longrightarrow} \mathsf{V} \overset{\pi}{\longrightarrow} \mathsf{W} \longrightarrow 0 \big)$
be an allowable short exact sequence in $\mathcal{Z}(F)$ and let $\mathsf{N}$ be any object in $\mathcal{Z}(F)$. Then we have a long exact sequence of $k$-vector spaces
\[\xymatrix{
0 \ar[r] & H^0_{\mathrm{DY}}(F;\mathsf{W},\mathsf{N}) \ar[r]^{\pi^*} & H^0_{\mathrm{DY}}(F;\mathsf{V},\mathsf{N}) \ar[r]^{j^*} & H^0_{\mathrm{DY}}(F;\mathsf{U},\mathsf{N}) \ar[lld]_{c^0} &\\
& H^1_{\mathrm{DY}}(F;\mathsf{W},\mathsf{N}) \ar[r]^{\pi^*} & H^1_{\mathrm{DY}}(F;\mathsf{V},\mathsf{N}) \ar[r]^{j^*} & H^1_{\mathrm{DY}}(F;\mathsf{U},\mathsf{N}) \ar[lld]_{c^1} & \\
& H^2_{\mathrm{DY}}(F;\mathsf{W},\mathsf{N}) \ar[r]^{\pi^*} & H^2_{\mathrm{DY}}(F;\mathsf{V},\mathsf{N}) \ar[r]^{j^*} & H^2_{\mathrm{DY}}(F;\mathsf{U},\mathsf{N}) & \!\!\!\!\!\!\!\!\!\!\!\!\!\!\!\!\!\!\!\ldots
} \]
\end{corollary}

\noindent A nice feature of this long exact sequence is that its arrows are easy to describe. Indeed, the pullback morphisms $\pi^*, j^*$ have a very simple expression (Lemma \ref{pullbackDY}). The connecting morphisms $c^n$ are obtained by transporting the connecting morphisms from Theorem \ref{thLongExactSequenceExtGroups} through the isomorphism $\Gamma$ from \eqref{defGamma}. Let
\[ S_{\mathrm{DY}}= \Gamma^{-1}\bigl(\eta(S)\bigl) \in H^1_{\mathrm{DY}}(F;\mathsf{W},\mathsf{U}) \]
be the DY cocycle associated to $S \in \YExt_{\mathcal{Z}(F),\mathcal{D}}^1(\mathsf{W}, \mathsf{U})$ (recall \eqref{defBarEta}). Then $c^n(g) = (-1)^n g \circ S_{\mathrm{DY}}$, where the product $\circ$ is given by the simple formula in Theorem \ref{propYonedaProductDY}. The only non-trivial point is to determine the DY cocycle $S_{\mathrm{DY}}$ associated to $S$.

\smallskip

\indent To conclude, we reformulate Corollary \ref{corollaryInductionFormulaExt} for DY cohomology. Let $S = \bigl( 0 \rightarrow \mathsf{K} \rightarrow \mathsf{P} \rightarrow \mathsf{W} \rightarrow 0 \bigr)$ be an allowable short exact sequence in $\mathcal{Z}(F)$, with $\mathsf{P}$ a relatively projective object. Then for $n \geq 1$ the connecting morphism is an isomorphism:
\begin{equation}\label{connectingMorphismIsoDY}
\foncIso{c^{n}}{H^{n}_{\mathrm{DY}}(F;\mathsf{K},\mathsf{N})}{H^{n+1}_{\mathrm{DY}}(F;\mathsf{W},\mathsf{N})}{g}{(-1)^n g \circ S_{\mathrm{DY}}}.
\end{equation}

\section{Finite-dimensional Hopf algebras}\label{sectionFinDimHopfAlgebras}
In this section, $H$ is a finite-dimensional Hopf algebra over a field $k$, with coproduct $\Delta : H \to H \otimes H$, counit $\varepsilon : H \to k$ and antipode $S : H \to H$. We will specialize the general results from the previous sections to the case $\mathcal{C} = H\text{-}\mathrm{mod}$ (which is a finite $k$-linear tensor category). Moreover, in \S \ref{methodDYCocyclesThanksToSequences} we explain a method to compute explicit DY cocycles for the identity functor of $H$-mod and in \S \ref{sectionFactoHopfAlg} we give further results in the case of factorizable Hopf algebras.

\smallskip

\indent In the sequel we denote the iterated coproduct by $\Delta^{(n)}$:
\begin{equation}\label{iteratedDelta}
\Delta^{(-1)} = \varepsilon, \quad \Delta^{(0)} = \mathrm{id}, \quad \Delta^{(1)} = \Delta, \quad \Delta^{(2)} = (\Delta \otimes \mathrm{id})\Delta, \: \ldots
\end{equation}
and we use Sweedler's notation with implicit summation:
\[ \Delta(h) = h' \otimes h'', \quad \Delta^{(2)}(h) = h' \otimes h'' \otimes h''', \quad \ldots, \quad \Delta^{(n)}(h) = h^{(1)} \otimes \ldots \otimes h^{(n+1)}. \]

\indent We are mainly interested in DY cohomology for the identity functor on $H\text{-}\mathrm{mod}$. However, another interesting functor is the forgetful functor $U : H\text{-}\mathrm{mod} \to \mathrm{Vect}_k$ which we study first.

\subsection{DY cohomology of the forgetful functor}

\indent Let $(H^*)^{\mathrm{op}}$ be $H^*$ with the product $\varphi\psi = (\psi \otimes \varphi)\Delta$ and the coproduct defined by $\Delta(\varphi)(x \otimes y) = \varphi(xy)$. It has been shown in \cite[Prop.\,7]{davydov} that $C^{\bullet}_{\mathrm{DY}}(U)$ is isomorphic to the Hochschild cochain complex of $(H^*)^{\mathrm{op}}$ with the trivial coefficient. As noted in \cite[Rem.\,5.11]{GHS} this implies that $H^{\bullet}_{\mathrm{DY}}(U) \cong \Ext_{(H^*)^{\mathrm{op}}}(\mathbb{C}, \mathbb{C})$. Here we generalize this isomorphism to non-trivial coefficients, which is relevant e.g. because of the factorization property in Lemma \ref{factorization1CocyclesDY})
\begin{lemma}\label{DYForgetfulExt}
1. We have an equivalence of tensor categories $\mathcal{Z}(U) \cong (H^*)^{\mathrm{op}}\text{-}\mathrm{mod}$.
\\2. For all $V, W \in (H^*)^{\mathrm{op}}\text{-}\mathrm{mod}$, $H^n_{\mathrm{DY}}(U; V,W) \cong \Ext^n_{(H^*)^{\mathrm{op}}}(V,W)$.
\end{lemma}
\begin{proof}
1. An object in $\mathcal{Z}(U)$ is a pair $(V,\rho^V)$, where $V$ is simply a vector space. For such a pair, the formula
\[ \varphi \cdot v = (\varphi \otimes \mathrm{id}_V)\big(\rho_H^V(v \otimes 1)\big) \]
with $\varphi \in H^*$, $v \in V$ endows $V$ with a structure of $(H^*)^{\mathrm{op}}$-module. Conversely if $(V, \cdot)$ is a $(H^*)^{\mathrm{op}}$-module, the formula
\[ \rho_H^V(v \otimes x) = h_ix \otimes h^i\cdot v \]
(where $(h_i)$ is a basis of $H$ and $(h^i)$ is its dual basis) can be extended by naturality to a half-braiding relative to $U$. It is straightforward to check that this correspondence defines a strict tensor equivalence.
\\2. We have
\begin{equation*}
H^n_{\mathrm{DY}}(U; V,W) \cong \Ext^n_{\mathcal{Z}(U), \mathrm{Vect}_k}(V,W) \cong \Ext^n_{(H^*)^{\mathrm{op}}\text{-}\mathrm{mod}, \mathrm{Vect}_k}(V,W)  = \Ext^n_{(H^*)^{\mathrm{op}}}(V,W).
\end{equation*}
The first isomorphism is due to Corollary \ref{DYrelExt}, the second isomorphism is due to the previous item (actually we only use the equivalence of abelian categories) and the equality is due to the semisimplicity of $\mathrm{Vect}_k$ (see Example \ref{exampleSemisimpleBottom}).
\end{proof}

Therefore, in order to compute $H^n_{\mathrm{DY}}(U; V,W)$, or equivalently $\Ext^n_{(H^*)^{\mathrm{op}}}(V,W)$, one just needs to find a usual projective resolution of the $(H^*)^{\mathrm{op}}$-module $V$, apply $\Hom_{(H^*)^{\mathrm{op}}}(-,W)$ and calculate cohomologies of the resulting complex.

\subsection{The resolvent pair $\mathcal{Z}(\mathrm{Id}_{H\text{-}\mathrm{mod}}) \leftrightarrows H\text{-}\mathrm{mod}$}
\indent Let $(H^*)^{\mathrm{op}}$ be $H^*$ with the product $\varphi \psi = (\psi \otimes \varphi)\Delta$. The Drinfeld double of $H$, denoted $D(H)$, is $(H^*)^{\mathrm{op}} \otimes H$ as a coalgebra and has the product
\[ (\varphi \otimes h)(\psi \otimes g) = \varphi \psi\bigl( S(h')?h''' \bigr) \otimes h''g \]
where $\psi\bigl( S(h')?h''' \bigr)$ is the linear form $x \mapsto \psi\bigl( S(h')xh''' \bigr)$ and $\varphi \psi\bigl( S(h')?h''' \bigr)$ is the linear form $x \mapsto \varphi(x'') \,\psi\bigl( S(h')x'h''' \bigr)$. As usual, we identify $\varphi \in (H^*)^{\mathrm{op}}$ with $\varphi \otimes 1 \in D(H)$ and $h \in H$ with $1 \otimes h \in D(H)$, so that $(H^*)^{\mathrm{op}}$ and $H$ become subalgebras of $D(H)$. Then $\varphi \otimes h$ can be written as $\varphi h$ and $h\psi = \psi\bigl( S(h') ? h''' \bigr)h''$.

\smallskip

\indent Recall that $\mathcal{Z}(\mathrm{Id}_{H\text{-}\mathrm{mod}}) = \mathcal{Z}(H\text{-}\mathrm{mod}) \cong {^H_H}\mathcal{YD}$. 
Indeed, if $(U,\rho^U) \in \mathcal{Z}(H\text{-}\mathrm{mod})$ then 
\[ \lambda_U = \rho^U_H(? \otimes 1) : U \to H \otimes U \]
is a left comodule structure and $(U,\lambda_U)$ is a left-left Yetter--Drinfeld module. Moreover, it is well-known that ${^H_H}\mathcal{YD} \cong D(H)\text{-}\mathrm{mod}$, since one can define a $(H^*)^{\mathrm{op}}$-action on $(V, \lambda_V) \in {^H_H}\mathcal{YD}$ by
\[ \varphi \cdot v = (\varphi \otimes \mathrm{id}_V)(\lambda_V(v)) \]
and a left coaction on $W \in D(H)\text{-}\mathrm{mod}$ by 
\begin{equation}\label{comoduleStructureOnDHModule}
\lambda_W(w) = h_i \otimes h^i \cdot w
\end{equation}
where $(h_i)$ is a basis of $H$ with dual basis $(h^ i)$. In the sequel we identify all these categories with $D(H)\text{-}\mathrm{mod}$.

\smallskip

\indent Let us rephrase the resolvent pair $\mathcal{Z}(\mathrm{Id}_{H\text{-}\mathrm{mod}}) \leftrightarrows H\text{-}\mathrm{mod}$ (recall \eqref{relSystemZF}) under the identification $\mathcal{Z}(\mathrm{Id}_{H\text{-}\mathrm{mod}}) = D(H)\text{-}\mathrm{mod}$, following \cite[\S 4]{GHS}. The forgetful functor $\mathcal{U} : D(H)\text{-}\mathrm{mod} \to H\text{-}\mathrm{mod}$ is induced by the obvious injective morphism $H \to D(H)$ and forgets the $(H^*)^{\mathrm{op}}$-action. The associated induction functor $\mathcal{F} : H\text{-}\mathrm{mod} \to D(H)\text{-}\mathrm{mod}$ is given by
\begin{equation}\label{inductionFunctorForDH}
\mathcal{F}(X) = D(H) \otimes_H X \cong (H^* \otimes_k X)_{\mathrm{coad}}, \qquad \mathcal{F}(f) = \mathrm{id}_{(H^*)^{\mathrm{op}}} \otimes f
\end{equation}
where the $D(H)$-module structure on $(H^* \otimes X)_{\mathrm{coad}}$ is
\begin{equation}\label{inductionFunctorHDH}
h \cdot (\psi \otimes x) = \psi\bigl( S(h') ? h''' \bigr) \otimes h''x, \qquad \varphi \cdot (\psi \otimes x) = (\varphi\psi) \otimes x
\end{equation}
with $h \in H, \varphi \in H^*$ and we recall that $\varphi \psi$ denotes the product in $(H^*)^{\mathrm{op}}$. Then we get the resolvent pair $D(H)\text{-}\mathrm{mod} \leftrightarrows H\text{-}\mathrm{mod}$ and Corollary \ref{DYrelExt} gives
\begin{equation}\label{DYandRelExtDHH}
H^n_{\mathrm{DY}}(H\text{-}\mathrm{mod}; V, W) \cong \Ext^n_{D(H),H}(V, W)
\end{equation}
where $V, W$ are $D(H)$-modules.

\subsection{DY cohomology of the identity functor for $H\text{-}\mathrm{mod}$}\label{sectionDYCohomologyIdentityFunctorHMod}
\indent Here we explain that the DY complex of $H\text{-}\mathrm{mod}$ with coefficients $V,W \in D(H)\text{-}\mathrm{mod}$ can be encoded with more tractable data than in the general definition with natural transformations. The point is that a natural transformation is entirely determined by its value on the regular representation.

\smallskip

\indent Let $(H^{\otimes n} \otimes W)_{\mathrm{ad}}$ be the vector space $H^{\otimes n} \otimes W$ equipped with the left $H$-module structure defined by
\[ h\cdot (x_1 \otimes \ldots \otimes x_n \otimes w) = h^{(1)}x_1S(h^{(2n+1)}) \otimes \ldots \otimes h^{(n)}x_nS(h^{(n+2)}) \otimes h^{(n+1)}\cdot w.  \]
We have a linear map
\begin{equation}\label{isoDYHomH}
\fonc{\forall \, n, \:\: \Psi}{C^n_{\mathrm{DY}}(H\text{-}\mathrm{mod};V,W)}{\Hom_H\bigl(V, (H^{\otimes n} \otimes W)_{\mathrm{ad}}\bigr)}{f}{\bigl(v \mapsto f_{H,\ldots,H}(v \otimes 1 \otimes \ldots \otimes 1)\bigr)}
\end{equation}
where on the right-hand side the natural transformation $f$ is evaluated on the regular representation $H$ on each slot. Let us show that $\Psi(f)$ is indeed $H$-linear. For $a \in H$, let $r_a \in \mathrm{End}_H(H)$ be the right multiplication by $a$: $r_a(x) = xa$. Write $\Psi(f) = \Psi(f)^1_H \otimes \ldots \otimes \Psi(f)^n_H \otimes \Psi(f)_W$ (with implicit summation); then by $H$-linearity and naturality of $f$ we get
\begin{align*}
\Psi(f)(h \cdot v) &= f_{H,\ldots,H}(h \cdot v \otimes 1 \otimes \ldots \otimes 1) = f_{H,\ldots,H}\!\left(h^{(1)}\cdot v \otimes h^{(2)}S(h^{(2n+1}) \otimes \ldots \otimes h^{(n+1)}S(h^{(n+2)})\right)\\
&= \bigl(h^{(1)} \otimes \ldots \otimes h^{(n+1)}\bigr) \cdot f_{H,\ldots,H}\bigl(\mathrm{id}_V \otimes r_{S(h^{(2n+1)})} \otimes \ldots \otimes r_{S(h^{(n+2)})}\bigr)(v \otimes 1 \otimes \ldots \otimes 1)\\
&= \bigl(h^{(1)} \otimes \ldots \otimes h^{(n+1)}\bigr) \cdot \bigl(r_{S(h^{(2n+1)})} \otimes \ldots \otimes r_{S(h^{(n+2)})} \otimes \mathrm{id}_W\bigr)f_{H,\ldots,H}(v \otimes 1 \otimes \ldots \otimes 1)\\
&= h^{(1)} \Psi(f)^1_H(v) S(h^{(2n+1)}) \otimes \ldots \otimes h^{(n)} \Psi(f)^n_H(v) S(h^{(n+2)}) \otimes h^{(n+1)}\cdot \Psi(f)_W(v)
\end{align*}
which shows that $\Psi(f) \in \Hom_H\bigl(V, (H^{\otimes n} \otimes W)_{\mathrm{ad}}\bigr)$.

\smallskip

\indent For each $n$, $\Psi$ is an isomorphism of vector spaces, whose inverse is given by
\begin{equation*}
\forall \, n, \:\: \Psi^{-1}(\varphi)_{X_1, \ldots, X_n}(v \otimes x_1 \otimes \ldots \otimes x_n) = \varphi^1_H(v) \cdot x_1 \otimes \ldots \otimes \varphi^n_H(v) \cdot x_n \otimes \varphi_W(v)
\end{equation*}
where again we write $\varphi = \varphi^1_H \otimes \ldots \otimes \varphi^n_H \otimes \varphi_W$ with implicit summation. We can thus transport the Davydov--Yetter differential through the isomorphism $\Psi$ and an easy computation gives:
\begin{lemma}\label{reformulationDYHmod}
The Davydov--Yetter complex of $H\text{-}\mathrm{mod}$ with coefficients $V,W \in D(H)$ is isomorphic to the cochain complex of vector spaces
\[ 0 \longrightarrow \Hom_H\bigl(V, W\bigr) \overset{\delta^0}{\longrightarrow} \Hom_H\bigl(V, (H \otimes W)_{\mathrm{ad}}\bigr) \overset{\delta^1}{\longrightarrow} \Hom_H\bigl(V, (H^{\otimes 2} \otimes W)_{\mathrm{ad}}\bigr) \overset{\delta^2}{\longrightarrow} \ldots \]
with differential
\[ \delta^n(\varphi) = \bigl(\mathrm{id}_H \otimes \varphi\bigr)\lambda_V + \sum_{i=1}^n (-1)^i \bigl(\mathrm{id}^{\otimes (i-1)} \otimes \Delta \otimes \mathrm{id}^{\otimes (n-i)} \otimes \mathrm{id}_W\bigr)\varphi + (-1)^{n+1}\bigl(\mathrm{id}_H^{\otimes n} \otimes \lambda_W\bigr)\varphi \]
where $\lambda_V, \lambda_W$ are the left $H$-coactions defined in \eqref{comoduleStructureOnDHModule}.
\end{lemma}

For trivial coefficients ($V = W = k$), this can be made even more explicit. Since in this case the $H$-action on $(H^{\otimes n})_{\mathrm{ad}}$ is $h \cdot (x_1 \otimes \ldots \otimes x_n) = \Delta^{(n-1)}\bigl(h'\bigr) (x_1 \otimes \ldots \otimes x_n) \Delta^{(n-1)}\bigl(S(h'')\bigr)$, we see that
\begin{equation}\label{DYHmodForTrivialCoeffs}
C^n_{\mathrm{DY}}(H\text{-}\mathrm{mod}) \underset{\Psi}{\overset{\sim}{\to}} \Hom_H\bigl(k, (H^{\otimes n})_{\mathrm{ad}}\bigr)  =  \mathcal{Z}\bigl(\Delta^{(n-1)}(H)\bigr)
\end{equation}
where $\mathcal{Z}\bigl(\Delta^{(n-1)}(H)\bigr) \subset H^{\otimes n}$ is the centralizer of the image of the iterated coproduct \eqref{iteratedDelta}. The isomorphism of Lemma \ref{reformulationDYHmod} is then
\begin{equation}\label{DYComplexExplicitFormWithCentralizers}
C_{\mathrm{DY}}^{\bullet}(H\text{-}\mathrm{mod}) \cong \Big( 0 \longrightarrow k \overset{\delta^0}{\longrightarrow} \mathcal{Z}(H) \overset{\delta^1}{\longrightarrow} \mathcal{Z}(\Delta(H)) \overset{\delta^2}{\longrightarrow} \ldots \Big)
\end{equation}
where the complex on the right has the differential
\begin{equation}\label{differentialDYHmodForTrivialCoeffs}
\delta^n(\mathbf{x}) = 1 \otimes \mathbf{x} + \sum_{i=1}^n (-1)^i (\mathrm{id}^{\otimes (i-1)} \otimes \Delta \otimes \mathrm{id}^{\otimes (n-i)})(\mathbf{x}) + (-1)^{n+1}\mathbf{x} \otimes 1.
\end{equation}
Hence $C^{\bullet}_{\mathrm{DY}}(H\text{-}\mathrm{mod})$ is a subcomplex of the Cartier complex for the coalgebra $H$. This was already noted in \cite[Prop.\,8]{davydov} and taken as a definition in \cite[\S6]{ENO}. Note that $\delta_0 = 0$.
\begin{remark}\label{remarkConcreteDY}
As noted in \cite[Prop.\,7]{davydov}, for the forgetful functor $U : H\text{-}\mathrm{mod} \to \mathrm{Vect}_k$ the centralizer condition in \eqref{DYComplexExplicitFormWithCentralizers} disappears and $C_{\mathrm{DY}}^{\bullet}(U)$ is isomorphic to the complex
\begin{equation}\label{DYForgetfulHModTrivialCoeffs}
0 \longrightarrow k \overset{\delta^0}{\longrightarrow} H \overset{\delta^1}{\longrightarrow} H^{\otimes 2} \overset{\delta^2}{\longrightarrow} \ldots
\end{equation}
with the differential \eqref{differentialDYHmodForTrivialCoeffs}. This gives an injective morphism of complexes $\iota_n : C_{\mathrm{DY}}^n(H\text{-}\mathrm{mod}) \hookrightarrow C_{\mathrm{DY}}^n(U)$; the induced morphisms $\overline{\iota}_n : H_{\mathrm{DY}}^n(H\text{-}\mathrm{mod}) \hookrightarrow H_{\mathrm{DY}}^n(U)$ are not injective in general (see e.g.\,Remark \ref{remarkRelExtNot0UsualExt0BarUi} below), except for $n=1$  because $\delta_0 = 0$.
\end{remark}

\indent We use the isomorphism \eqref{isoDYHomH} to transport the product of Theorem \ref{propYonedaProductDY}: if $\varphi \in \Hom_H\big( V, (H^{\otimes n} \otimes W)_{\mathrm{ad}} \big)$ and $\psi \in \Hom_H\big( U, (H^{\otimes m} \otimes V)_{\mathrm{ad}} \big)$ are cocycles then $\varphi \circ \psi = \Psi\big( \Psi^{-1}(\varphi) \circ \Psi^{-1}(\psi) \big)$ is given by
\begin{equation}\label{YonedaProductHomH}
\varphi \circ \psi = (-1)^{nm}(\mathrm{id}_{H^{\otimes m}} \otimes \varphi)\psi \in \Hom_H\big(U, (H^{\otimes (n + m)} \otimes W)_{\mathrm{ad}} \big).
\end{equation}

\smallskip

\indent In view of the following section, let us restate the isomorphism $\Gamma : C^n_{\mathrm{DY}}(H\text{-}\mathrm{mod};V,W) \overset{\sim}{\to} \mathrm{Bar}^n_{D(H),H}(V,W)$ defined in \eqref{defGamma}. We use the desciption of the bar resolution $\mathrm{Bar}^{\bullet}_{D(H),H}(V)$ for the adjunction $D(H)\text{-}\mathrm{mod} \leftrightarrows H\text{-}\mathrm{mod}$ given in \cite[Cor.\,4.6]{GHS}. Let $\bigl( (H^*)^{\otimes n} \otimes V \bigr)_{\mathrm{coad}}$ be the $D(H)$-module defined by
\begin{align*}
h \cdot (\varphi_1 \otimes \ldots \otimes \varphi_n \otimes v) &= \varphi_1\bigl( S(h^{(1)}) ? h^{(2n+1)} \bigr) \otimes \ldots \otimes \varphi_n\bigl( S(h^{(n)}) ? h^{(n+2)} \bigr) \otimes h^{(n+1)}v\\
\psi \cdot (\varphi_1 \otimes \ldots \otimes \varphi_n \otimes v) &= (\psi\varphi_1) \otimes \ldots \otimes \varphi_n \otimes v
\end{align*}
where we recall that $\psi\varphi_1$ denotes the product in $(H^*)^{\mathrm{op}}$. Then 
\begin{equation}\label{barResolutionDHH}
\mathrm{Bar}^n_{D(H),H}(V) = \bigl( (H^*)^{\otimes (n+1)} \otimes V \bigr)_{\mathrm{coad}}
\end{equation}
with the differential
\[ d_n(\varphi_1 \otimes \ldots \otimes \varphi_{n+1} \otimes v) = \varphi_1 \otimes \ldots \otimes \varphi_n \otimes \varphi_{n+1} v + \sum_{i=1}^n (-1)^{n-i+1} \varphi_1 \otimes \ldots \otimes (\varphi_i\varphi_{i+1}) \otimes \ldots \otimes \varphi_{n+1} \otimes v. \]
We have the isomorphism of complexes
\begin{equation}\label{defGammaTilde}
\forall \, n, \quad \widetilde{\Gamma} : \Hom_H\bigl( V, (H^{\otimes n} \otimes W)_{\mathrm{ad}} \bigr) \underset{\Psi^{-1}}{\overset{\sim}{\longrightarrow }} C^n_{\mathrm{DY}}(H\text{-}\mathrm{mod};V,W) \underset{\Gamma}{\overset{\sim}{\longrightarrow}} \mathrm{Bar}^n_{D(H),H}(V,W)
\end{equation}
between the complex of Lemma \ref{reformulationDYHmod} and $\mathrm{Bar}^{\bullet}_{D(H),H}(V,W) = \Hom_{D(H)}\!\Big( \bigl((H^*)^{\otimes (\bullet+1)} \otimes V \bigr)_{\mathrm{coad}}, W\Big)$. In the sequel we will need its inverse, which is simply given by
\begin{equation}\label{isoBarGammaInverse}
\forall \, n, \quad \widetilde{\Gamma}^{-1}(\alpha) : v \mapsto h_{i_1} \otimes \ldots \otimes h_{i_n} \otimes \alpha\bigl( \varepsilon \otimes h^{i_n} \otimes \ldots \otimes h^{i_1} \otimes v \bigr)
\end{equation}
where $(h_i)$ is a basis of $H$ with dual basis $(h^ i)$ and $\varepsilon$ is the counit of $H$. Finally, note that by definition $\widetilde{\Gamma}$ is compatible with the products $\circ$: $\widetilde{\Gamma}(f \circ g) = \widetilde{\Gamma}(f) \circ \widetilde{\Gamma}(g)$.

\subsection{DY cocycles from allowable exact sequences}\label{methodDYCocyclesThanksToSequences}
\indent The isomorphism between the groups $\Ext^{\bullet}_{D(H),H}$ and the DY cohomology groups \eqref{DYandRelExtDHH} allows us to determine the dimension of the latter thanks to a relatively projective resolution. If some dimension is not $0$, one would like to find explicit DY cocycles. A possible way to achieve this without too much computations is to find allowable $n$-fold exact sequences of $D(H)$-modules. Indeed, there is a DY cocycle associated to such a sequence, thanks to the maps
\begin{align*}
&\big\{ \text{allowable } n \text{-fold exact sequences from } W \text{ to } V \big\}\\
& \overset{\eta}{\longrightarrow} \: \big\{ \text{cocycles in } \mathrm{Bar}^n_{D(H),H}(V,W) \big\} \: \overset{\widetilde{\Gamma}^{-1}}{\longrightarrow} \: \big\{ \text{cocycles in } \Hom_H\bigl( V, (H^{\otimes n} \otimes W)_{\mathrm{ad}} \bigr) \big\}
\end{align*}
defined in \eqref{IsoYExtExtOnCocycles} and \eqref{isoBarGammaInverse} respectively. We denote by
\begin{equation}\label{defSDY}
S_{\mathrm{DY}} = \widetilde{\Gamma}^{-1}\eta(S).
\end{equation}
the DY cocycle associated to a sequence $S$.

\smallskip

\indent So assume that we have found an allowable $n$-fold exact sequence $S = \bigl(0 \to W \to X_n \to \ldots \to X_1 \to V \to 0\bigr)$ and let us explain how to compute $S_{\mathrm{DY}}$. Since we are using the bar resolution descibed explicitly in \eqref{barResolutionDHH}, computing $\eta(S)$ amounts to fill the following diagram:
\[ \xymatrix{
\big((H^*)^{\otimes (n+1)} \otimes V\big)_{\mathrm{coad}} \ar@{-->}[d]^{\eta(S)} \ar[r]^{d_{n+1}} & \big((H^*)^{\otimes n} \otimes V\big)_{\mathrm{coad}} \ar@{-->}[d] \ar[r]^{\qquad \quad \:\:\:\:\: d_n} & \ldots \ar[r]^{\!\!\!\!\!\!\!\!\!\!\!\!\!\!\!\!\!\!\!\!d_1} & \big(H^* \otimes V\big)_{\mathrm{coad}} \ar@{-->}[d] \ar[r]^{\qquad \:\: d_0} & V \ar@{=}[d]\ar[r] & 0\\
W \ar[r] & X_n \ar[r] & \ldots \ar[r] & X_1 \ar[r] & V \ar[r] & 0
} \]
The longer $S$ is, the harder it becomes to fill the diagram. So we decompose $S$ as the Yoneda product of $n$ short exact sequences $S = S_1 \circ \ldots \circ S_n$ and we compute $\eta(S_i)$ and $(S_i)_{\mathrm{DY}} = \widetilde{\Gamma}\eta(S_i)$ for each $i$. Then we note that due to the compatibility of $\eta$ and $\widetilde{\Gamma}$ with the products $\circ$, we have
\begin{equation}\label{SDYYoneda}
S_{\mathrm{DY}} = (S_1)_{\mathrm{DY}} \circ \ldots \circ (S_n)_{\mathrm{DY}}
\end{equation}
where on the right hand side $\circ$ is the product \eqref{YonedaProductHomH}. Since this product has a very simple expression, the computation becomes feasible and we get an explicit DY $n$-cocycle. Note moreover that due to the formula \eqref{isoBarGammaInverse} for $\widetilde{\Gamma}^{-1}$, it is sufficient to determine the values of the form $\eta(S_i)(\varepsilon \otimes ? \otimes ?)$. This method will be demonstrated on the examples below.

\smallskip

\indent After this, one would like to show that the so obtained cocycle is not cohomologous to $0$. It is equivalent to show that the exact sequence $S$ is not congruent to $0$. For $2$-fold exact sequences one can use Lemma \ref{lemma2FoldSequences} or Corollary \ref{remarkCriterion2FoldSequences}.

\subsection{Factorizable Hopf algebras}\label{sectionFactoHopfAlg}
\indent Let $H \otimes H$ be endowed with its standard Hopf algebra structure. We have the forgetful functor $\mathcal{U}_{\Delta} : (H \otimes H)\text{-}\mathrm{mod} \to H\text{-}\mathrm{mod}$ induced by the coproduct $\Delta : H \to H \otimes H$. It is easy to check that its left adjoint $\mathcal{F}_{\Delta} : H\text{-}\mathrm{mod} \to (H \otimes H)\text{-}\mathrm{mod}$ is given by
\[ \mathcal{F}_{\Delta}(X) = H \otimes X \quad \text{with action } (a \otimes b)\cdot (h \otimes x) = ahS(b') \otimes b''\cdot x \]
on objects and $\mathcal{F}_{\Delta}(f) = \mathrm{id}_H \otimes f$ on morphisms. Since $\mathcal{U}_{\Delta}$ is obviously additive, exact and faithful, the adjunction $\mathcal{F}_{\Delta} \dashv \mathcal{U}_{\Delta}$ is a resolvent pair. We denote by $\Ext^n_{H \otimes H, \Delta(H)}$ the corresponding relative Ext groups.

\smallskip

\indent Recall the isomorphism of categories $(H \otimes H)\text{-}\mathrm{mod} \cong H\text{-}\mathrm{mod}\text{-}H$ where the action of $H \otimes 1$ gives a left action and the action of $1 \otimes H$ composed with $S^{-1}$ gives a right action because $S$ is an isomorphism of algebras $H \to H^{\mathrm{op}}$.
The indecomposable direct summands of the regular bimodule $_HH_H$ are called the blocks of $H$. Let $k$ be the trivial $H$-module. Since $\mathcal{F}_{\Delta}(k) = H$ with action $(a \otimes b) \cdot h = ahS(b)$, the indecomposable direct summands of $\mathcal{F}_{\Delta}(k)$ are exactly the blocks. By item 2 in Lemma \ref{lemmaBasicPropertiesRelProj}, it follows that the blocks are relatively projective objects for $\mathcal{F}_{\Delta} \dashv \mathcal{U}_{\Delta}$.

\smallskip

\indent From now on we assume that the ground field $k$ is algebraically closed. Let $P_1$ be the projective cover of the trivial $H$-module $k$. Let $Q_1$ be the block which has $P_1$ as a direct summand when it is considered as a $H$-module by left multiplication. It is also known as the principal block of $H$. Note that the direct summand $Q_1$ has multiplicity $1$ in $\mathcal{F}_{\Delta}(k)$ because $P_1$ has multiplicity $1$ in the regular module $_HH$ by \cite[Thm. 9.2.1(ii)]{repTheory}.
\begin{lemma}\label{lemmaHomBlockToTrivial}
Let $\varepsilon|_{Q_1}$ be the restriction to $Q_1$ of the counit $\varepsilon : H \to k$ and let $B$ be some block. Then
\[ \Hom_{H \otimes H}(B, k \boxtimes k) = \begin{cases}
k\,\varepsilon|_{Q_1} & \text{if } B = Q_1\\
0 & \text{otherwise.}
\end{cases} \]
\end{lemma}
\begin{proof}
Recall from e.g. \cite[Thm. 9.2.1(i)]{repTheory} that if $\{S_i\}_i$ is the list of simple $H$-modules and $\{P_i\}_i$ is the list of their projective covers, we have $\dim \mathrm{Hom}_H(P_i,S_j) = \delta_{i,j}$. Let $_HB$ be $B$ seen as a left $H$-module by left multiplication. If $B \neq Q_1$, we have $\Hom_{H \otimes H}(B, k \boxtimes k) \subset \Hom_H({_HB}, k) = 0$ because $_HB$ does not contain $P_1$ as direct summand. The counit $\varepsilon : H \to k \boxtimes k$ is a morphism of $(H \otimes H)$-modules, so it vanishes on any $B\neq Q_1$. Since $\varepsilon \neq 0$ it follows that $\varepsilon|_{Q_1} \neq 0$. As a result
\[ 1 \leq \dim \Hom_{H \otimes H}\bigl( Q_1, k \boxtimes k \bigr) \leq \dim \Hom_H\bigl( {_HQ_1}, k \bigr) = 1 \]
where the last equality is due to the fact that $P_1$ has multiplicity $1$ in $_HQ_1$.
\end{proof}
\begin{proposition}\label{lemmaRelProjCoverFactorizable}
$Q_1$ is the relatively projective cover of the trivial $(H \otimes H)$-module $k \boxtimes k$ for $\mathcal{F}_{\Delta} \dashv \mathcal{U}_{\Delta}$.
\end{proposition}
\begin{proof}
The counit of the adjunction $\varepsilon_V : \mathcal{F}_{\Delta}\mathcal{U}_{\Delta}(V) \to V$ is given by $\varepsilon_V(h \otimes v) = (h \otimes 1)\cdot v$ with $h \in H$ and $v \in V$. In particular for $V = k \boxtimes k$, $\varepsilon_V$ is the counit $\varepsilon : \mathcal{F}_{\Delta}(k) = H \to k$ of the Hopf algebra $H$.  By Lemma \ref{lemmaHomBlockToTrivial}, $Q_1$ is the smallest direct summand of $\mathcal{F}_{\Delta}(k)$ on which $\varepsilon$ does not vanish and the result then follows from Remark \ref{remarkHowToComputeRelProjCover}.
\end{proof}

\smallskip

\indent Assume now that $H$ is quasitriangular with $R$-matrix $R = \sum_i R^1_i \otimes R^2_i \in H^{\otimes 2}$. Recall that $H$ is called {\em factorizable} if the linear map $H^* \to H$ defined by $\varphi \mapsto (\varphi \otimes \mathrm{id})(R'R)$ is an isomorphism of vector spaces, where $R' = \sum_i R^2_i \otimes R^1_i$. Adapting the arguments of \cite[\S 4]{schneider}, where different conventions for $D(H)$ are used, we have for factorizable $H$ an isomorphism of algebras
\[ \foncIso{\Xi}{D(H)}{H \otimes H}{\varphi \, h}{\sum_{i,j} \varphi\bigl(S(R^1_i)R^2_j  \bigr) R^2_ih' \otimes R^1_jh''} \]
which by pullback induces an isomorphism of linear categories
\[ \Xi^* : (H \otimes H)\text{-}\mathrm{mod} \to D(H)\text{-}\mathrm{mod}. \]
\begin{proposition}\label{propFactorizable}
If $H$ is factorizable then
\[ \Ext^n_{D(H),H}\bigl( \Xi^*(V), \Xi^*(W) \bigr) \cong \Ext^n_{H \otimes H, \Delta(H)}(V,W) \]
for all $V, W \in (H \otimes H)\text{-}\mathrm{mod}$ and $n \geq 0$.
\end{proposition}
\begin{proof}
The pullback functor $\Xi^*$ relates the two resolvent pairs:
\[ \xymatrix@R=.7em{
(H \otimes H)\text{-}\mathrm{mod} \ar@/^.7em/[ddr]^-{\mathcal{U}_{\Delta}} \ar[rr]^{\Xi^*}& & D(H)\text{-}\mathrm{mod} \ar@/^.7em/[ddl]^-{\mathcal{U}}\\
 & & \\
& H\text{-}\mathrm{mod} \ar@/^.7em/[uul]-^{\mathcal{F}_{\Delta}}_{\text{\normalsize $\dashv$}} \ar@/^.7em/[uur]^-{\mathcal{F}}_{\text{\normalsize $\dashv$}}&
} \]
We have $\mathcal{U}_{\Delta} = \mathcal{U} \circ \Xi^*$, from which it follows that if $f \in \Hom_{H \otimes H}(X,Y)$ is allowable for $\mathcal{F}_{\Delta} \dashv \mathcal{U}_{\Delta}$ then $\Xi^*(f) = f \in \Hom_{D(H)}\bigl(\Xi^*(X),\Xi^*(Y)\bigr)$ is allowable for $\mathcal{F} \dashv \mathcal{U}$. Moreover one checks that $\mathcal{F} \cong \Xi^* \circ \mathcal{F}_{\Delta}$, from which it follows that if $P$ is relatively projective for $\mathcal{F}_{\Delta} \dashv \mathcal{U}_{\Delta}$ then $\Xi^*(P)$ is relatively projective for $\mathcal{F} \dashv \mathcal{U}$. Thus $\Xi^*$ transforms a relatively projective resolution of $V$ into a relatively projective resolution of $\Xi^*(V)$. To conclude, note that $\Hom_{H \otimes H}(P,W) = \Hom_{D(H)}\bigl(\Xi^*(P),\Xi^*(W)\bigr)$ for all $(H \otimes H)$-module $P$ because $\Xi^*$ is the identity on morphisms, and therefore we have isomorphic complexes for both relative Ext groups.
\end{proof}
\noindent We note that Proposition \ref{propFactorizable} does not need the assumption that $k$ is algebraically closed.

\smallskip

\indent The functor $\mathcal{U}_{\Delta} : (H \otimes H)\text{-}\mathrm{mod} \to H\text{-}\mathrm{mod}$ is monoidal, with
\[ \foncIso{\bigl(\mathcal{U}_{\Delta}^{(2)}\bigr)_{V,W}}{\mathcal{U}_{\Delta}(V \otimes W)}{\mathcal{U}_{\Delta}(V) \otimes \mathcal{U}_{\Delta}(W)}{v \otimes w}{\sum_i\, (1 \otimes S(R^1_i))\cdot v \otimes (R^2_i \otimes 1)\cdot w} \]
where we use that $(S \otimes \mathrm{id})(R) = R^{-1}$ and $R\Delta = \Delta^{\mathrm{op}}R$. Thus the results of \S \ref{sectionMonoidalResolventPairs} apply to the resolvent pair $\mathcal{F}_{\Delta} \dashv \mathcal{U}_{\Delta}$. In particular, the dimension formula of Corollary \ref{CoroDimExtWithHom} applies.

\smallskip

We finish with a reformulation of Corollary \ref{corollaryDimensionDYGroups} in the case of factorizable Hopf algebras. Recall the block $Q_1$ introduced before Lemma \ref{lemmaHomBlockToTrivial}.

\begin{proposition}\label{propPropertiesQ1Facto}
If $H$ is factorizable then $Q_1$ has a unique $(H \otimes H)$-submodule isomorphic to $k \boxtimes k$, and $Q_1$  is self-dual.
Moreover $Q_1/(k \boxtimes k) \cong K_1^{\vee}$ as $(H \otimes H)$-modules, where $K_1 = \ker(\varepsilon|_{Q_1} : Q_1 \to k)$.
\end{proposition}
\begin{proof}
By \cite[Rem. 4.4]{schneider} or \cite[Prop. 3]{radford2} we know that a factorizable Hopf algebra is unimodular, which means that $H$ contains an element $\Lambda \neq 0$ such that $\Lambda h = h \Lambda = \varepsilon(h)\Lambda$ for all $h \in H$, \textit{i.e.} $\Lambda$ is an invariant element in $_HH_H$, and is called a cointegral. In this situation we have Radford's isomorphism of $H$-bimodules \cite[Thm.\,1]{radford}
\[ \Phi : H^* \overset{\sim}{\to} H, \quad \varphi \mapsto (\varphi \otimes \mathrm{id})\bigl( \Delta(\Lambda) \bigr) = \varphi(\Lambda')\Lambda'' \]
where $H$ has the regular actions and $H^*$ has the actions defined by $\langle a \cdot \varphi \cdot b, h\rangle = \langle \varphi, S(a)hS^{-1}(b) \rangle$.
Let $u = \sum_i S(R^2_i)R^1_i$ be the Drinfeld element of $H$, which is such that $S^2(h) = uhu^{-1}$ for all $h \in H$. Then we have an isomorphism of $(H \otimes H)$-modules
\[ \Phi' : \mathcal{F}_{\Delta}(k)^{\vee} \overset{\sim}{\to} \mathcal{F}_{\Delta}(k), \quad \varphi \mapsto \Phi(\varphi)u. \]
We now show that $\Phi'$ yields an isomorphism $Q_1^{\vee} \overset{\sim}{\to} Q_1$. 
First note that $H$ contains $P_1$ (projective cover of the tensor unit) with multiplicity 1, as a left $H$-module, and it is in $Q_1$. We then note that $P_1$ is self-dual, because of unimodularity, so it is also contained in $Q_1^{\vee}$. Now assume that $Q_1^{\vee}$ is not isomorphic to $Q_1$, as an $H$-bimodule, but then it should be isomorphic to some block $Q_s$ with $s\ne 1$ because $H \cong H^*$ as bimodules. However no blocks $Q_s$ contain $P_1$ unless $s=1$, so we get a contradiction and thus $Q_1^{\vee}$ should be isomorphic to $Q_1$. 

Note that $\Phi'(\varepsilon|_{Q_1}) = \Lambda u = \Lambda$ because $\varepsilon(u)=1$. Hence $\Lambda$ belongs to $Q_1$ and generates a submodule isomorphic to $k \boxtimes k$. Moreover, $\mathcal{F}_{\Delta}(k)$ contains a unique submodule isomorphic to $k \boxtimes k$ because by definition such a submodule is spanned by a cointegral, which is known to be unique up to scalar. Let $j : K_1 \to Q_1$ be the inclusion and consider
\[ \pi : Q_1 \xrightarrow{\: \Phi'^{-1}\:} Q_1^{\vee} \xrightarrow{\:j^{\vee}\:} K_1^{\vee}. \]
The morphism $\pi$ is surjective and $\pi(\Lambda) = j^{\vee}(\varepsilon|_{Q_1}) = 0$. Since $\dim(K_1^{\vee}) = \dim(Q_1) - 1$, we conclude that $\ker(\pi) = k\Lambda \cong k \boxtimes k$. Hence $\pi$ provides an isomophism $Q_1/(k \boxtimes k) \overset{\sim}{\to} K_1^{\vee}$.
\end{proof}

\begin{corollary}\label{dimFormulaDYFactoHopfAlg}
Let $H$ be a factorizable Hopf algebra. Write $T = Q_1/(k \boxtimes k)$ (see Prop. \ref{propPropertiesQ1Facto}) and $K_1 = \ker(\varepsilon|_{Q_1} : Q_1 \to k)$. We have for $n \geq 2$
\begin{align*}
\dim H^n_{\mathrm{DY}}(H\text{-}\mathrm{mod}) = \dim \Hom_{H \otimes H}\bigl( K_1, T^{\otimes (n-1)} \bigr) &- \dim \Hom_{H \otimes H}\bigl( Q_1, T^{\otimes (n-1)} \bigr)\\
&+ \dim \Hom_{H \otimes H}\bigl( k \boxtimes k, T^{\otimes (n-1)} \bigr).
\end{align*}
\end{corollary}
\begin{proof}
By \eqref{DYandRelExtDHH} and Proposition \ref{propFactorizable} we have
\[ H^n_{\mathrm{DY}}(H\text{-}\mathrm{mod}) \cong \Ext^n_{H \otimes H, \Delta(H)}(k\boxtimes k, k \boxtimes k). \]
By Proposition \ref{lemmaRelProjCoverFactorizable} we have an allowable short exact sequence
\[ 0 \longrightarrow K_1 \longrightarrow Q_1 \overset{\varepsilon|_{Q_1}}{\longrightarrow} k \boxtimes k \longrightarrow 0 \]
in $(H \otimes H)\text{-}\mathrm{mod}$. The claim then follows from Corollary \ref{CoroDimExtWithHom} since by Proposition \ref{propPropertiesQ1Facto} we know that $K_1^{\vee} \cong T$.
\end{proof}

\indent We will see an application of the above results in \S \ref{sectionSmallQuantumGroupH2} for $H = u_q(\mathfrak{sl}_2)$.

\section{Examples}\label{sectionExamples}

\indent In this section we demonstrate the use of our results on examples in the case $\mathcal{C} = H\text{-}\mathrm{mod}$ and $F = \mathrm{Id}_{\mathcal{C}}$, where $H$ is a finite-dimensional Hopf algebra. We consider the bosonization of the exterior algebra $\Lambda\mathbb{C}^k \rtimes \mathbb{C}[\mathbb{Z}_2]$, the Taft algebra, the restricted quantum group $\bar U_{\mathbf{i}}(\mathfrak{sl}_2)$ at a fourth root of unity~$\mathbf{i}$ and the small quantum group $u_q(\mathfrak{sl}_2)$ at a root of unity $q$ of odd order.

\subsection{Bosonization of exterior algebras}\label{DYBk}
Let $B_k = \Lambda\mathbb{C}^k \rtimes \mathbb{C}[\mathbb{Z}_2]$ be the $\mathbb{C}$-algebra with the presentation
\begin{equation}\label{defBk}
B_k = \mathbb{C}\langle x_1, \ldots, x_k, g \, | \, \forall \, i,j, \:\: x_ix_j = -x_jx_i, \:\: gx_i = - x_ig, \:\: g^2 = 1 \rangle.
\end{equation}
It has dimension $2^{k+1}$, with basis elements $x_1^{\alpha_1}\ldots x_k^{\alpha_k} g^{\alpha_{k+1}}$ ($\alpha_i \in \{0,1\}$). Its Hopf structure is
\[ \begin{array}{lll}
\Delta(x_i) = 1 \otimes x_i + x_i \otimes g, & \varepsilon(x_i) = 0, & S(x_i) = gx_i,\\
\Delta(g) = g \otimes g, & \varepsilon(g) = 1, & S(g) = g^{-1}.
\end{array} \]

\subsubsection{Dimension of DY cohomology groups}
\indent In \cite[\S 5]{GHS}, the DY cohomologies with trivial coefficients of the identity and forgetful functors have been computed for the category $B_k\text{-mod}$ thanks to Theorem \ref{thGHS} by finding an explicit $G$-projective resolution of $\mathbb{C}$. Here we present a much easier way to get these results, based on Corollaries~\ref{DYrelExt} and~\ref{relATensBsemisimple}, which moreover explains why the cohomologies for the identity and forgetful functors are isomorphic in the case of trivial coefficients (Proposition \ref{propDYCohomologyBk}).

\smallskip

\indent Recall that $\mathrm{sVect}$ is the category of super-vector spaces, which objects are $\mathbb{Z}_2$-graded $\mathbb{C}$-vector spaces $V = V_0 \oplus V_1$ and which morphisms are $\mathbb{C}$-linear maps respecting the gradings. The tensor product in $\mathrm{sVect}$ is defined by
\[ (V \otimes W)_0 = (V_0 \otimes W_0) \oplus (V_1 \otimes W_1), \qquad (V \otimes W)_1 = (V_1 \otimes W_0) \oplus (V_0 \otimes W_1). \]
The symmetry $c_{V,W} : V \otimes W \to W \otimes V$ is defined by
\[ c_{V,W}(v \otimes w) = (-1)^{|v||w|} w \otimes v \]
where $v$ and $w$ are homogeneous elements and $|v| = j$ if $v \in V_j$ ($j \in \{0,1\}$).

\smallskip

\indent $\mathrm{sVect}$ is equivalent as a tensor category to $\mathbb{C}[\mathbb{Z}_2]\text{-}\mathrm{mod}$. Let $H$ be a Hopf algebra in $\mathrm{sVect}$ and write $\mathbb{C}[\mathbb{Z}_2] = \mathbb{C}\langle g | g^2=1 \rangle$. Then $g$ acts on $H$ by $g \triangleright h = (-1)^{|h|}h$ and by assumption $H$ is a $\mathbb{C}[\mathbb{Z}_2]$-module-algebra. Hence the smash product $H \# \mathbb{C}[\mathbb{Z}_2]$ is an algebra in $\mathrm{Vect}$. With $\Delta(g) = g \otimes g$ then it actually is a Hopf algebra (in $\mathrm{Vect}$) which we denote by $H \rtimes \mathbb{C}[\mathbb{Z}_2]$, and called the bosonization of $H$. We have an equivalence of tensor categories
\begin{equation}\label{bosonization}
H\text{-}\mathrm{sVect} \cong (H \rtimes \mathbb{C}[\mathbb{Z}_2])\text{-}\mathrm{mod}
\end{equation}
defined by $V_0 = \ker(g - \mathrm{id}), V_1 = \ker(g+\mathrm{id})$, and where we denote $H\text{-}\mathrm{sVect}$ instead of $H\text{-}\mathrm{mod}_{\mathrm{sVect}}$ (which is the category of $H$-modules internal to $\mathrm{sVect}$).

\smallskip

\indent Now take 
\[ H = \Lambda\mathbb{C}^k = \langle x_1, \ldots, x_k \, | \, \forall \, i,j, \: x_ix_j = -x_jx_i \rangle\]
with coproduct $\Delta(x_i) = 1 \otimes x_i + x_i \otimes 1$. It is a Hopf algebra in $\mathrm{sVect}$.  Its bosonization $\Lambda\mathbb{C}^k \rtimes \mathbb{C}[\mathbb{Z}_2]$ is the Hopf algebra $B_k$ defined in \eqref{defBk}.

\smallskip

\indent We want to compute $H^n_{\mathrm{DY}}(B_k\text{-}\mathrm{mod})$. According to \eqref{DYandRelExtDHH}, this is isomorphic to $\Ext_{D(B_k),B_k}^n(\mathbb{C},\mathbb{C})$. Defining relations for $D(B_k)$ were obtained in \cite[App.\,C]{FGR}, however here we use its presentation from \cite[\S5]{GHS}:
\[ D(B_k) = \left\langle x_1, \ldots, x_k, y_1, \ldots, y_k, g, h \,\left|\, \forall\,i,j, \: 
\begin{array}{lll}
x_ix_j = -x_jx_i, & gx_i = - x_ig, & g^2 = 1,\\
y_iy_j = -y_jy_i, & hy_i = - y_ih, & h^2 = 1,\\
hx_i = -x_ih, & gy_i = - y_ig, & gh=hg,\\
x_iy_j + y_jx_i & \!\!\!\!\!\!\!\!\!\!\!= \delta_{i,j}(1-gh)& 
\end{array}
\right.\right\rangle \]
where $\delta_{i,j}$ is the Kronecker symbol. We see that
\[ \pi_{\pm} = \frac{1 \pm gh}{2} \]
are central orthogonal idempotents and that we have an isomorphism of algebras $D(B_k)\pi_+ \overset{\sim}{\to} B_{2k}$ given by
\[ x_i\pi_+ \mapsto x_i, \quad y_i\pi_+ \mapsto x_{k+i}, \quad g\pi_+ = h\pi_+ \mapsto g. \]
Hence the algebra map 
\[ \iota_+ : B_k \to D(B_k)\pi_+, \quad x \mapsto \iota(x)\pi_+ \]
(where $\iota : B_k \to D(B_k)$ is the canonical injection) can be viewed as a homomorphism $\iota_+ : B_k \to B_{2k}$; it is simply given by $\iota_+(x_i) = x_i$, $\iota_+(g)=g$. In particular it is injective and yields the resolvent pair $(B_{2k},B_k)$. Moreover
\begin{equation}\label{ExtDBB}
\Ext^n_{D(B_k),B_k}(\mathbb{C},\mathbb{C}) \cong \Ext^n_{B_{2k},B_k}(\mathbb{C},\mathbb{C}).
\end{equation}
Indeed, if $0 \leftarrow \mathbb{C} \leftarrow P_0 \leftarrow P_1 \leftarrow \ldots$ is a $\bigl( D(B_k),B_k \bigr)$-relatively projective resolution, then $0 \leftarrow \pi_+\mathbb{C} = \mathbb{C} \leftarrow \pi_+P_0 \leftarrow \pi_+P_1 \leftarrow \ldots$ is both a $\bigl( D(B_k),B_k \bigr)$-relatively projective resolution and a $(B_{2k},B_k)$-relatively projective resolution.

\smallskip

\indent Let $U : B_k\text{-}\mathrm{mod} \to \mathrm{Vect}$ be the forgetful functor.
\begin{proposition}\label{propDYCohomologyBk}
$H^n_{\mathrm{DY}}(B_k\text{-}\mathrm{mod}) \cong H^n_{\mathrm{DY}}(U)$.
\end{proposition}
\begin{proof}
Note that $\Lambda\mathbb{C}^{k+l} = \Lambda\mathbb{C}^k \otimes \Lambda\mathbb{C}^l$ in $\mathrm{sVect}$ and that $\mathrm{sVect}$ is semisimple. Hence, using \eqref{ExtDBB}, \eqref{bosonization}, Corollary \ref{DYrelExt}, Corollary \ref{relATensBsemisimple}, the fact that $B_k \cong (B^*_k)^{\mathrm{op}}$ and Lemma \ref{DYForgetfulExt}, we get
\begin{align*}
H^n_{\mathrm{DY}}(B_k\text{-}\mathrm{mod}) &\cong \Ext^n_{D(B_k),B_k}(\mathbb{C},\mathbb{C}) \cong \Ext^n_{B_{2k},B_k}(\mathbb{C},\mathbb{C}) \cong \Ext^n_{\Lambda\mathbb{C}^{2k}\text{-}\mathrm{sVect}, \Lambda\mathbb{C}^k\text{-}\mathrm{sVect}}(\mathbb{C}^{1|0},\mathbb{C}^{1|0})\\
&\cong \Ext^n_{(\Lambda\mathbb{C}^k \otimes \Lambda\mathbb{C}^k)\text{-}\mathrm{sVect}, \Lambda\mathbb{C}^k\text{-}\mathrm{sVect}}(\mathbb{C}^{1|0} \boxtimes \mathbb{C}^{1|0}, \mathbb{C}^{1|0} \boxtimes \mathbb{C}^{1|0}) \cong \Ext^n_{\Lambda\mathbb{C}^k\text{-}\mathrm{sVect}}(\mathbb{C}^{1|0}, \mathbb{C}^{1|0})\\
& \cong \Ext^n_{B_k}(\mathbb{C},\mathbb{C}) \cong \Ext^n_{(B^*_k)^{\mathrm{op}}}(\mathbb{C},\mathbb{C}) \cong H^n_{\mathrm{DY}}(U)
\end{align*}
where $\mathbb{C}^{1|0} = \mathbb{C} \oplus 0$ (no non-zero elements in odd degree). 
\end{proof}
Hence we have $H^n_{\mathrm{DY}}(B_k\text{-}\mathrm{mod}) \cong H^n_{\mathrm{DY}}(U) \cong \Ext^n_{(B^*_k)^{\mathrm{op}}}(\mathbb{C},\mathbb{C})$, from which it is not difficult to compute that (see \cite[Rem.\,5.11]{GHS})
\[ \dim\big( H^n_{\mathrm{DY}}(B_k\text{-}\mathrm{mod}) \big)\cong
\begin{cases}
\binom{k+n-1}{n} & \text{ if } n \text{ is even.}\\
0 & \text{ if } n \text{ is odd.}
\end{cases} \]

\subsubsection{Explicit cocycles from allowable exact sequences}\label{explicitCocyclesBk}
Recall from \eqref{defBarEta} and \eqref{DYandRelExtDHH} that we have isomorphisms of graded algebras (for the Yoneda product)
\[ \Ext^{\bullet}_{D(B_k),B_k}(\mathbb{C}, \mathbb{C}) \cong \YExt^{\bullet}_{D(B_k),B_k}(\mathbb{C}, \mathbb{C}) \cong H^{\bullet}_{\mathrm{DY}}(B_k\text{-}\mathrm{mod}). \]
Recall from \S \ref{subsectionProductOnDY} that the Yoneda product on $H^{\bullet}_{\mathrm{DY}}(B_k\text{-}\mathrm{mod}) = \bigoplus_{n \geq 0} H^n_{\mathrm{DY}}(B_k\text{-}\mathrm{mod})$ is simply given by $f \circ g = (-1)^{|f| |g|} g \otimes f$, for homogeneous elements $f$ and $g$. Here we will describe these spaces and their algebra structure; in particular we will use the method of \S \ref{methodDYCocyclesThanksToSequences} to find a basis of explicit DY cocycles.

\smallskip

\indent For $\epsilon \in \{\pm\}$, let $\mathbb{C}_{\epsilon}$ be the $D(B_k)$-module defined by $x_i \cdot 1 = y_i \cdot 1 = 0$, $g \cdot 1 = h \cdot 1 = \epsilon$. For $1 \leq i \leq n$, let $Y_i^{\epsilon}$ be the $2$-dimensional $D(B_k)$-module with basis $(u_{\epsilon}, u_{-\epsilon})$ and defined by
\[ \begin{array}{llll}
x_j u_{\epsilon} = 0, & y_j u_{\epsilon} = \delta_{i,j}u_{-\epsilon}, & g u_{\epsilon} = \epsilon u_{\epsilon}, & h u_{\epsilon} = \epsilon u_{\epsilon}, \\
x_j u_{-\epsilon} = 0, & y_j u_{-\epsilon} = 0, & g u_{-\epsilon} = -\epsilon u_{-\epsilon}, & h u_{-\epsilon} = -\epsilon u_{-\epsilon}.
\end{array} \]
Its subquotient structure can be depicted by 
\[ \xymatrix@C=.2em@R=.2em{
& \mathbb{C}_{\epsilon} \ar[dd]^{y_i} \\
Y_i^{\epsilon} = & \\
& \mathbb{C}_{-\epsilon}
} \]
so that it is a non-trivial extension of $\mathbb{C}_{\epsilon}$ by $\mathbb{C}_{-\epsilon}$, thus yielding the non-trivial exact sequence
\[ S_i^{\epsilon} = \bigl(0 \longrightarrow \mathbb{C}_{-\epsilon} \longrightarrow Y_i^{\epsilon} \longrightarrow \mathbb{C}_{\epsilon} \longrightarrow 0\bigr).\]
Moreover this short exact sequence splits when we look it in $B_k\text{-}\mathrm{mod}$, so it is allowable and can be seen as an element of $\YExt^1_{D(B_k),B_k}(\mathbb{C}_{\epsilon}, \mathbb{C}_{-\epsilon})$. We also define
\[ S_{i,j} = S_i^- \circ S_j^+ = \bigl( 0 \longrightarrow \mathbb{C} \longrightarrow Y_i^- \longrightarrow Y_j^+ \longrightarrow \mathbb{C} \longrightarrow 0 \bigr) \in \YExt^2_{D(B_k),B_k}(\mathbb{C}, \mathbb{C}). \]

\begin{proposition}\label{lemmaCongruenceBk}
1. Let $V = \mathbb{C}^k$ and $\xi_1, \ldots, \xi_k$ be a basis of $V^*$. There is an isomorphism of graded algebras
\[ \Ext^{\bullet}_{D(B_k),B_k}(\mathbb{C}, \mathbb{C}) \cong \bigoplus_{n \in 2\mathbb{N}} S^n(V^*) = \mathbb{C}[\xi_i\xi_j]_{1 \leq i \leq j \leq k} \]
where the latter is the subalgebra of even degree polynomials in the symmetric algebra $S^{\bullet}(V^*)$.

\smallskip

\noindent 2. Let $\overline{\eta}$ be the isomorphism $\YExt^{\bullet}_{D(B_k),B_k}(\mathbb{C}, \mathbb{C}) \overset{\sim}{\to}\Ext^{\bullet}_{D(B_k),B_k}(\mathbb{C}, \mathbb{C})$ (see \eqref{defBarEta}). Then, under the isomorphism of item 1, we have $\overline{\eta}(S_{i,j}) = \xi_i\xi_j$. It follows that $S_{i,j} \equiv S_{j,i}$ and that for $n$ even the sequences
\[ S_{i_1, i_2} \circ S_{i_3, i_4} \circ \ldots \circ S_{i_{n-1},i_n} \qquad (1 \leq i_1 \leq \ldots \leq i_n \leq k) \]
form a basis of $\YExt^n_{D(B_k),B_k}(\mathbb{C}, \mathbb{C})$.

\smallskip

\noindent 3. The explicit DY cocycle associated to $S_{i,j}$ is $(S_{i,j})_{\mathrm{DY}} = x_j \otimes x_ig$, which is an element of $\mathcal{Z}\bigl( \Delta(B_k) \bigr) \cong C^2_{\mathrm{DY}}(B_k\text{-}\mathrm{mod})$. It follows that for $n$ even, the elements
\[ x_{i_1} \otimes x_{i_2}g \otimes x_{i_3} \otimes x_{i_4}g \otimes \ldots \otimes x_{i_{n-1}} \otimes x_{i_n}g \qquad (1 \leq i_1 \leq \ldots \leq i_n \leq k) \]
form a basis of $H^n_{\mathrm{DY}}(B_k\text{-}\mathrm{mod})$.
\end{proposition}
\begin{proof}
1. Let us recall the relatively projective resolution of $\mathbb{C}$ constructed in \cite[\S 5]{GHS}, which is based on the Koszul resolution for the exterior algebra. Let $f_{\pm} = \frac{1 \pm h}{2}$, which are orthogonal idempotents in $B_k^*$. Define $\mathcal{C}_{\pm} = B_k^*f_{\pm}$ to be the $D(B_k)$-module where $B_k^*$ acts by multiplication, $x_i$ acts by $0$ for all $i$ and $g$ acts as $h$. Endow $S^n(V) \otimes \mathcal{C}_{\pm}$ with the $D(B_k)$-module structure given by $a\cdot (P \otimes w) = P \otimes (a \cdot w)$ and let $d_n : S^n(V) \otimes \mathcal{C}_{\pm} \to S^{n-1}(V) \otimes \mathcal{C}_{\mp}$ be the $D(B_k)$-linear map defined by
\[ d_n(t_{i_1} \ldots t_{i_n} \otimes f_{\pm}) = \sum_{j=1}^n t_{i_1} \ldots \widehat{t}_{i_j} \ldots t_{i_n} \otimes y_{i_j}f_{\mp} \]
where $(t_i)_{1 \leq i \leq k}$ is the basis of $V$ dual to $(\xi_i)_{1 \leq i \leq k}$, and $\widehat{t}_{i_j}$ means that $t_{i_j}$ is removed from the monomial. Then
\[ \ldots \overset{d_3}{\longrightarrow} S^2(V) \otimes \mathcal{C}_+ \overset{d_2}{\longrightarrow} S^1(V) \otimes \mathcal{C}_- \overset{d_1}{\longrightarrow} S^0(V) \otimes \mathcal{C}^+ \overset{d_0}{\longrightarrow} \mathbb{C} \longrightarrow 0 \]
is a relatively projective resolution. Since $\Hom_{D(B_k)}(\mathcal{C_+},\mathbb{C}) \cong \mathbb{C}$ (with basis element $f_+ \mapsto 1$) and $\Hom_{D(B_k)}(\mathcal{C_-},\mathbb{C}) = 0$, we have
\[ \Ext_{D(B_k),B_k}^n(\mathbb{C}, \mathbb{C}) =
\begin{cases}
\Hom_{D(B_k)}\bigl( S^n(V) \otimes \mathcal{C}_+, \mathbb{C} \bigr) & \text{if } n \text{ is even,}\\
0 & \text{if } n \text{ is odd.}
\end{cases}\]
For $n$ even, let us define (recall that $f_+$ is the generator of $\mathcal{C}_+$)
\begin{equation}\label{isoExtDBkSym}
\foncIso{F}{S^n(V^*)}{\Hom_{D(B_k)}\bigl( S^n(V) \otimes \mathcal{C}_+, \mathbb{C} \bigr)}{\psi}{\bigl( P \otimes f_+ \mapsto \langle \psi, P \rangle\bigr)}
\end{equation}
where we use the standard pairing $\langle -,- \rangle : S^n(V^*) \otimes S^n(V) \to \mathbb{C}$ given by $\langle \psi_1 \ldots \psi_n, v_1 \ldots v_n \rangle = \sum_{\sigma \in \mathfrak{S}_n}\psi_1(v_{\sigma(1)}) \ldots \psi_n(v_{\sigma(n)})$, with $\psi_i \in V^*$ and $v_i \in V$ for all $i$. The map $F$ clearly gives an isomorphism of graded vector spaces, it remains to show that it is compatible with the Yoneda product. So let $\varphi \in S^n(V^*)$ and $\psi \in S^m(V^*)$; according to \eqref{diagramLiftYoneda}, in order to compute $F(\varphi) \circ F(\psi)$ we have to fill the diagram
\[ \xymatrix@C=2.5em@R=2.5em{
S^{m+n}(V) \otimes \mathcal{C}_+ \ar@{-->}[d]^{\widetilde{F(\psi)}_n} \ar[r] & \: \ldots \: \ar[r] & S^{m+1}(V) \otimes \mathcal{C}_- \ar@{-->}[d]^{\widetilde{F(\psi)}_1} \ar[r] & S^m(V) \otimes \mathcal{C}_+ \ar@{-->}[d]^{\widetilde{F(\psi)}_0} \ar[rd]^{F(\psi)} & & \\
S^n(V) \otimes \mathcal{C}_+ \ar[r] & \: \ldots \: \ar[r] & S^1(V) \otimes \mathcal{C}_- \ar[r] & S^0(V) \otimes \mathcal{C}_+ \ar[r] & \mathbb{C} \ar[r] & 0
} \]
One checks that $\widetilde{F(\psi)}_l(v_1 \ldots v_{m+l} \otimes f_{\pm}) = \sum_{\sigma \in \mathfrak{S}_{m+l}}\langle\psi, v_{\sigma(1)}\ldots v_{\sigma(m)}\rangle v_{\sigma(m+1)} \ldots v_{\sigma(m+l)} \otimes f_{\pm}$ does the job for each $l$. Hence
\begin{align*}
&\bigl(F(\varphi) \circ F(\psi)\bigr)(v_1 \ldots v_{m+n} \otimes f_+) = F(\varphi)\bigl( \widetilde{F(\psi)}_n(v_1 \ldots v_{m+n} \otimes f_+) \bigl)\\
& = \sum_{\sigma \in \mathfrak{S}_{m+n}} \langle\psi, v_{\sigma(1)}\ldots v_{\sigma(m)}\rangle \langle\varphi, v_{\sigma(m+1)}\ldots v_{\sigma(m+n)}\rangle = \langle \varphi\psi, v_1 \ldots v_{m+n} \rangle = F(\varphi\psi)(v_1 \ldots v_{m+n} \otimes f_+).
\end{align*}
2. According to \eqref{IsoYExtExtOnCocycles}, we have to fill the diagram
\[ \xymatrix{
S^2(V) \otimes \mathcal{C}_+ \ar[r] \ar@{-->}[d]^{\eta_2 = \overline{\eta}(S_{i,j})} & S^1(V) \otimes \mathcal{C}_- \ar[r] \ar@{-->}[d]^{\eta_1} & S^0(V) \otimes \mathcal{C}_+ \ar[r] \ar@{-->}[d]^{\eta_0} & \mathbb{C} \ar[r] \ar@{=}[d] & 0\\
\mathbb{C} \ar[r] & Y_i^- \ar[r] & Y_j^+ \ar[r] & \mathbb{C} \ar[r] & 0
} \]
The solution is easily found: let $(u_-, u_+)$ be the basis of $Y_i^-$ and $(u'_+, u'_-)$ be the basis of $Y_j^+$, then
\[ \eta_0(f_+) = u'_+, \qquad \eta_1(t_a \otimes f_-) = \delta_{j,a}u_-, \qquad \overline{\eta}(S_{i,j})(t_at_b \otimes f_+) = \delta_{i,a}\delta_{j,b} + \delta_{j,a}\delta_{i,b} \]
(and all extended by $D(B_k)$-linearity), where $(t_i)_{1 \leq i \leq k}$ is the basis of $V$ dual to $(\xi_i)_{1 \leq i \leq k}$. Finally, note that
\[ \overline{\eta}(S_{i,j})(t_at_b \otimes f_+) = \langle \xi_i\xi_j, t_a t_b \rangle = F(\xi_i\xi_j)(t_at_b \otimes f_+) \]
where $F$ is the isomorphism $S^2(V^*) \overset{\sim}{\to} \Ext_{D(B_k),B_k}^2(\mathbb{C}, \mathbb{C})$ defined in \eqref{isoExtDBkSym}.
\\Since $\xi_i\xi_j = \xi_j\xi_i$, we have $S_{i,j} \equiv S_{j,i}$ (which means that $S_{i,j} = S_{j,i}$ as elements of $\YExt^2_{D(B_k),B_k}(\mathbb{C}, \mathbb{C})$). Finally, since the monomials $\xi_{i_1} \ldots \xi_{i_n}$ with $1 \leq i_1 \leq \ldots \leq i_n \leq k$ are a basis of $S^n(V^*)$, the same is true for the corresponding elements in $\YExt^n_{D(B_k),B_k}(\mathbb{C}, \mathbb{C})$ for $n$ even.

\smallskip

\noindent 3. Recall the definition of the DY cocycle associated to an allowable exact sequence in \eqref{defSDY}, as well as its compatibility with the products $\circ$ explained in \eqref{SDYYoneda}. Let us first compute $\eta(S_i^-)$ and $\eta(S_j^+)$. For $S_i^-$, we have to fill the commutative diagram of $D(B_k)$-modules
\[\xymatrix{
\bigl((B_k^*)^{\otimes 2} \otimes \mathbb{C}_-\bigr)_{\mathrm{coad}} \ar[r]^{\quad d_1} \ar@{-->}[d]^{\eta_1 = \eta(S_i^-)} & \bigl(\!B_k^* \otimes \mathbb{C}_-\bigr)_{\mathrm{coad}} \ar[r]^{\qquad \mathrm{act}} \ar@{-->}[d]^{\eta_0} & \mathbb{C}_- \ar[r] \ar@{=}[d] &0\\
\mathbb{C}_+ \ar[r] & Y^-_i \ar[r] & \mathbb{C}_- \ar[r] & 0
}\]
Let $ f_{\pm} = \frac{\varepsilon \pm h}{2}$, then the elements $y_{i_1} \ldots y_{i_l}f_{\pm}$ (with $i_1 < \ldots < i_l$) are a basis of $(B_k^*)^{\mathrm{op}}$. The solution is easily found:
\begin{align*}
&\eta_0(f_- \otimes 1_-) = v_-, \: \eta_0(y_if_- \otimes 1_-) = v_+, \: \text{and } 0 \text{ on the other basis elements,}\\
&\eta\big(S_i^-\big)(\varepsilon \otimes y_if_- \otimes 1_-) = -1_+, \: \text{and } 0 \text{ on the other basis elements of the form } \varepsilon \otimes  (\ldots) \otimes 1_-.
\end{align*}
$\eta\big(S_i^-\big)$ can be extended to all the elements by $(B_k^*)^{\mathrm{op}}$-linearity but we do not need this. One similarly finds
\[ \eta\big(S_j^+\big)(\varepsilon \otimes y_jf_+ \otimes 1_+) = -1_-, \: \text{and } 0 \text{ on the other basis elements of the form } \varepsilon \otimes  (\ldots) \otimes 1_+. \]
Now, let $(x_{i_1} \ldots x_{i_l})^*, (x_{i_1} \ldots x_{i_l}g)^*$ be the elements of the dual basis of the monomial basis of $B_k$. Recall from \cite[\S 5]{GHS} that by definition $y_i = x_i^* - (x_ig)^*$, $h = 1^* - g^*$. It follows that $x_i^* = y_i f_+$ and $(x_ig)^* = -y_if_-$. With this we can compute the DY cocycles associated to $\eta(S_i^-), \eta(S_j^+)$ through the isomorphism \eqref{isoBarGammaInverse}:
\begin{align*}
&\big(S_i^-\big)_{\mathrm{DY}} = \widetilde{\Gamma}^{-1}\eta\big(S_i^-\big) : \mathbb{C}_- \to B_k \otimes \mathbb{C}_+, \quad 1_- \mapsto x_ig \otimes 1_+,\\
&\big(S_j^+\big)_{\mathrm{DY}} = \widetilde{\Gamma}^{-1}\eta\big(S_j^+\big) : \mathbb{C}_+ \to B_k \otimes \mathbb{C}_-, \quad 1_+ \mapsto -x_j \otimes 1_-.
\end{align*}
Finally, thanks to the product \eqref{YonedaProductHomH}, we get $(S_{i,j})_{\mathrm{DY}}(1_+) = \big(S_i^-\big)_{\mathrm{DY}} \circ \big(S_j^+\big)_{\mathrm{DY}}(1_+) = x_j \otimes x_ig \otimes 1_+$, which we identify with $x_j \otimes x_ig$. For the last claim, note that
\[ x_{i_1} \otimes x_{i_2}g \otimes \ldots \otimes x_{i_{n-1}} \otimes x_{i_n}g = (S_{i_n, i_{n-1}})_{\mathrm{DY}} \circ \ldots \circ (S_{i_2, i_1})_{\mathrm{DY}} = \bigl( S_{i_n, i_{n-1}} \circ \ldots \circ S_{i_2, i_1} \bigr)_{\mathrm{DY}} \]
and recall item 2.
\end{proof}

\subsection{Taft algebras}\label{sectionExampleTaftAlgebra}
\indent Let $q \in \mathbb{C}$ be a primitive $n$-th root of unity ($n \geq 2$). The Taft algebra $T_q$ is
\[ T_q = \bigl\langle x, g \, | \, gx = qxg, \: x^n=0,\: g^n=1 \bigr\rangle. \]
It is a $n^2$-dimensional Hopf $\mathbb{C}$-algebra with basis $(x^ig^j)_{0 \leq i,j \leq n-1}$ and with Hopf structure given by
\[\begin{array}{lll}
\Delta(x) = 1 \otimes x  + x \otimes g, & S(x) = -xg^{-1}, & \varepsilon(x) =0,\\
\Delta(g) = g \otimes g, & S(g) = g^{-1}, & \varepsilon(g) = 1.
\end{array}\]
Let $X_d^{(s)} = \mathrm{span}(v^{(s)}_0, \ldots, v^{(s)}_{d-1})$ ($1 \leq d \leq n$, $0 \leq s \leq n-1$) be the $T_q$-module defined by
\[ gv^{(s)}_i = q^{s+i}v^{(s)}_i, \quad xv^{(s)}_i = v^{(s)}_{i+1}, \quad xv^{(s)}_{d-1} = 0. \]
Any indecomposable $T_q$-module is isomorphic to some $X_d^{(s)}$.

\subsubsection{Dimension of DY cohomology groups}
\indent We will compute the DY cohomology of $T_q\text{-}\mathrm{mod}$ thanks to \eqref{DYandRelExtDHH}. The first task is to describe the Drinfeld double $D(T_q) = (T_q^*)^{\mathrm{op}} \otimes T_q$. In order to determine $(T_q^*)^{\mathrm{op}}$, we use the matrix coefficients of $X_2^{(0)}$ in its basis $(v^{(0)}_0, v^{(0)}_1)$. The representation matrix has the form
\[ \rho_{X_2^{(0)}} = \begin{pmatrix}
\varepsilon & 0\\
y & h
\end{pmatrix} \]
where $\varepsilon$ is the counit. It is not difficult to show that
\begin{equation}\label{pairingTaft}
\langle y^kh^l, x^ig^j\rangle = q^{lj}\, (k)_q! \, \delta_{i,k}
\end{equation}
(where $(k)_q = \frac{1 - q^k}{1 - q}$ and $(k)_q! = (1)_q (2)_q \ldots (k)_q$) and it easily follows that the monomials $y^kh^l$ form a basis of $(T_q^*)^{\mathrm{op}}$. Moreover, we have the presentation
\[ (T_q^*)^{\mathrm{op}} = \bigl\langle y, h \, | \, hy = qyh, \: y^n=0,\: h^n=1 \bigr\rangle. \]
The Hopf structure is given by
\[\begin{array}{lll}
\Delta(y) = y \otimes 1  + h \otimes y, & S(y) = -h^{-1}y, & \varepsilon(y) =0,\\
\Delta(h) = h \otimes h, & S(h) = h^{-1}, & \varepsilon(h) = 1.
\end{array}\]
The exchange relations in $D(T_q)$ are easily computed and we find
\[ D(T_q) = \left\langle x, y, g, h \,\left|\,
\begin{array}{llll}
gx = qxg, & x^n=0, & g^n=1,&\\
hy = qyh, & y^n=0,& h^n=1,&\\
hx = q^{-1}xh, & gy = q^{-1}yg, &gh=hg, & [x,y] = h - g
\end{array}
\right.\right\rangle. \]

\indent Let $\mathcal{V}^{(s,t)}$ (with $0 \leq s,t \leq n-1$) be the $D(T_q)$-module generated by a vector $v^{(s,t)}$ such that
\[ xv^{(s,t)} = 0, \quad gv^{(s,t)} = q^sv^{(s,t)}, \quad hv^{(s,t)} = q^tv^{(s,t)} \]
(formally it is an induced representation). It has the basis $v^{(s,t)}_i = y^iv^{(s,t)}$ ($0 \leq i \leq n-1$) with actions
\[ xv^{(s,t)}_i = (i)_q(q^t - q^{s+1-i})v^{(s,t)}_{i-1}, \quad yv^{(s,t)}_i = v^{(s,t)}_{i+1}, \quad gv^{(s,t)}_i = q^{s-i}v^{(s,t)}_i, \quad hv^{(s,t)}_i = q^{t+i}v^{(s,t)}_i \]
where we used that $xy^i = y^ix + (i)_qy^{i-1}\bigl(h - q^{1-i}g \bigr)$. This gives the familiar ladder structure:
\begin{align}
&\xymatrix{
v^{(s,t)}_0 \ar@<0.5ex>[r]^{\quad y} & \ar@<0.5ex>[l]^{\quad x} \:\: \ldots \!\!\!\!\!\!\!\!\!\!\!\! &\ar@<0.5ex>[r]^{\!\!\!\!\!\!\!y} & \ar@<0.5ex>[l]^{\!\!\!\!\!\!\!x} v^{(s,t)}_{s-t} \ar[r]^{\!\!\!y} & v^{(s,t)}_{s-t+1} \ar@<0.5ex>[r]^{\quad y} & \ar@<0.5ex>[l]^{\quad x} \:\: \ldots \!\!\!\!\!\!\!\!\!\!\!\! &\ar@<0.5ex>[r]^{\!\!\!\!\!\!\!y} & \ar@<0.5ex>[l]^{\!\!\!\!\!\!\!x} v^{(s,t)}_{n-1}
} \quad \text{if } t \leq s,\label{pictureVermaModuleDTq}\\
&\xymatrix{
v^{(s,t)}_0 \ar@<0.5ex>[r]^{\quad y} & \ar@<0.5ex>[l]^{\quad x} \:\: \ldots \!\!\!\!\!\!\!\!\!\!\!\! &\ar@<0.5ex>[r]^{\!\!\!\!\!\!\!y} & \ar@<0.5ex>[l]^{\!\!\!\!\!\!\!x} v^{(s,t)}_{n+s-t} \ar[r]^{\!\!\!y} & v^{(s,t)}_{n+s-t+1} \ar@<0.5ex>[r]^{\qquad y} & \ar@<0.5ex>[l]^{\qquad x} \:\: \ldots \!\!\!\!\!\!\!\!\!\!\!\! &\ar@<0.5ex>[r]^{\!\!\!\!\!\!\!y} & \ar@<0.5ex>[l]^{\!\!\!\!\!\!\!x} v^{(s,t)}_{n-1}
} \quad \text{if } t \geq s+1. \nonumber
\end{align}

\begin{lemma}\label{relProjObjectsForTaft}
We have $\mathcal{F}\bigl( X_1^{(s)} \bigr) = \bigoplus_{t=0}^{n-1} \mathcal{V}^{(s, t)}$, where $\mathcal{F}$ is the induction functor described in \eqref{inductionFunctorForDH}; hence the modules $\mathcal{V}^{(s, t)}$ are relatively projective.
\end{lemma}
\begin{proof}
We have $\mathcal{F}\big(X_1^{(s)}\big) = (T_q^*)^{\mathrm{op}} \otimes X_1^{(s)}$ with the $D(T_q)$-action from \eqref{inductionFunctorHDH}. For $0 \leq t \leq n-1$, let $\varphi_t = \frac{1}{n}\sum_{j = 0}^{n-1}q^{-jt}h^j$. Then $h\varphi_t = q^t\varphi_t$ and $(y^i\varphi_t)_{0 \leq i,t\leq n-1}$ is a basis of $(T_q^*)^{\mathrm{op}}$. Let $v^{(s)}$ be the basis vector of $X_1^{(s)}$. Since
\begin{align*}
&x(\varphi_t \otimes v^{(s)}) = \varphi_{t-1} \otimes xv^{(s)} = 0, \\
&g(\varphi_t \otimes v^{(s)}) =  \varphi_t \otimes gv^{(s)} = q^s(\varphi_t \otimes v^{(s)}),\\
&h(\varphi_t \otimes v^{(s)}) =  h\varphi_t \otimes v^{(s)} = q^t(\varphi_t \otimes v^{(s)}),
\end{align*}
the subspace $\mathrm{vect}\bigl(y^i\varphi_t \otimes v^{(s)}\bigr)_{0 \leq i \leq n-1}$ is isomorphic to $\mathcal{V}^{(s,t)}$. The last claim is due to Lemma \ref{lemmaBasicPropertiesRelProj}.
\end{proof}
\noindent Note in particular that $\mathcal{V}^{(0,0)}$ is the relatively projective cover of $\mathbb{C}$ (Definition \ref{defRelProjCover}).

\medskip

\indent Thanks to \eqref{pictureVermaModuleDTq}, we see that $\mathcal{V}^{(n-1,1)}$ and $\mathcal{V}^{(0,0)}$ fit into short exact sequences
\[ 0 \longrightarrow \mathbb{C} \overset{\alpha}\longrightarrow \mathcal{V}^{(n-1,1)} \overset{\beta}{\longrightarrow} \mathcal{K} \longrightarrow 0, \qquad 0 \longrightarrow \mathcal{K} \overset{\gamma}{\longrightarrow} \mathcal{V}^{(0,0)} \overset{\delta}{\longrightarrow} \mathbb{C} \longrightarrow 0 \]
where $\mathcal{K} \cong \mathrm{span}\bigl( v^{(0,0)}_2, \ldots, v^{(0,0)}_{n-1} \bigr)$ is a simple $D(T_q)$-module. These sequences split under the forgetful functor $\mathcal{U} : D(T_q)\text{-}\mathrm{mod} \to T_q\text{-}\mathrm{mod}$, indeed $\mathcal{U}(\mathcal{V}^{(0,0)}) \cong  \mathcal{U}(\mathcal{V}^{(n-1,1)}) \cong \mathbb{C} \oplus X^{(n-1)}_{n-1}$.
It follows that the sequence of relatively projective $D(T_q)$-modules
\begin{equation}\label{relativeResolutionTaft}
0 \longleftarrow \mathbb{C} \overset{\delta}{\longleftarrow} \mathcal{V}^{(0,0)} \overset{\gamma\beta}{\longleftarrow} \mathcal{V}^{(n-1,1)} \overset{\alpha\delta}{\longleftarrow} \mathcal{V}^{(0,0)} \overset{\gamma\beta}{\longleftarrow} \mathcal{V}^{(n-1,1)} \overset{\alpha\delta}{\longleftarrow} \ldots
\end{equation}
is an allowable exact sequence, and thus a relatively projective resolution of $\mathbb{C}$. For trivial coefficients, the resulting sequence is
\begin{align*}
&\Hom_{D(T_q)}(\mathcal{V}^{(0,0)},\mathbb{C}) \longrightarrow \Hom_{D(T_q)}(\mathcal{V}^{(n-1,1)},\mathbb{C}) \longrightarrow \Hom_{D(T_q)}(\mathcal{V}^{(0,0)},\mathbb{C}) \longrightarrow \ldots\\
=\:& \mathbb{C} \longrightarrow 0 \longrightarrow \mathbb{C} \longrightarrow  \ldots
\end{align*}
and we get
\begin{equation}\label{dimDYTaft}
H^k_{\mathrm{DY}}(T_q\text{-}\mathrm{mod}) \cong \Ext^k_{D(T_q),T_q}(\mathbb{C},\mathbb{C}) = \begin{cases}
\mathbb{C} & \text{if } k \text{ is even}\\
0 & \text{if } k \text{ is odd}
\end{cases}
\end{equation}

\begin{remark}
The complex \eqref{relativeResolutionTaft} is in particular a projective resolution of $\mathbb{C}$ in $(T_q^*)^{\mathrm{op}}\text{-}\mathrm{mod}$. Since
\[ \Hom_{D(T_q)}(\mathcal{V}^{(0,0)},\mathbb{C}) = \Hom_{(T_q^*)^{\mathrm{op}}}(\mathcal{V}^{(0,0)},\mathbb{C}), \quad \Hom_{D(T_q)}(\mathcal{V}^{(n-1,1)},\mathbb{C}) = \Hom_{(T_q^*)^{\mathrm{op}}}(\mathcal{V}^{(n-1,1)},\mathbb{C}), \]
it holds by Lemma \ref{DYForgetfulExt} that $H^k_{\mathrm{DY}}(T_q\text{-}\mathrm{mod}) \cong \Ext^k_{(T_q^*)^{\mathrm{op}}}(\mathbb{C},\mathbb{C}) \cong H^k_{\mathrm{DY}}(U)$, where $U : T_q\text{-}\mathrm{mod} \to \mathrm{Vect}_{\mathbb{C}}$ is the forgetful functor.
\end{remark}

\subsubsection{Explicit cocycles from allowable exact sequences}\label{explicitCocyclesTaft}
Here we apply the method described in \S \ref{methodDYCocyclesThanksToSequences}. Recall from the previous paragraph that we have two allowable short exact sequences:
\[ S_1 = \bigl(0 \longrightarrow \mathbb{C} \longrightarrow \mathcal{V}^{(n-1,1)} \longrightarrow \mathcal{K} \longrightarrow 0 \bigr), \qquad S_2 = \bigl(0 \longrightarrow \mathcal{K} \longrightarrow \mathcal{V}^{(0,0)} \longrightarrow \mathbb{C} \longrightarrow 0\bigr). \]
We denote by $S$ their Yoneda product:
\begin{equation}\label{def2FoldSequenceTaftAlgebra}
S = S_1 \circ S_2 = \big( 0 \longrightarrow \mathbb{C} \longrightarrow \mathcal{V}^{(n-1,1)} \longrightarrow \mathcal{V}^{(0,0)} \longrightarrow \mathbb{C} \longrightarrow 0 \big).
\end{equation}
For the next proposition, recall again from \S \ref{subsectionProductOnDY} that $H^{\bullet}_{\mathrm{DY}}(F) = \bigoplus_{k \geq 0} H^k_{\mathrm{DY}}(F)$ is an algebra for the product $f \circ g = (-1)^{|f||g|}g \otimes f$ (where $f,g$ are homogeneous elements of degree $|f|, |g|$ respectively).
\begin{proposition}
1. The $2$-fold exact sequence $S$ in \eqref{def2FoldSequenceTaftAlgebra} is not congruent to $0$ and thus is a basis of $\YExt^2_{D(T_q),T_q}(\mathbb{C}, \mathbb{C})$.
\\2. The explicit DY cocycle associated to $S$ is
\[ S_{\mathrm{DY}} = -\sum_{i=1}^{n-1} \frac{1}{(i)_q!(n-i)_q!}x^i \otimes x^{n-i}g^i \]
which is an element in $\mathcal{Z}\bigl( \Delta(T_q) \bigr) \cong C^2_{\mathrm{DY}}(T_q\text{-}\mathrm{mod})$. It is a basis of $H^2_{\mathrm{DY}}(T_q\text{-}\mathrm{mod})$.
\\3. The DY cocycle $(S_{\mathrm{DY}})^{\circ k} = (S_{\mathrm{DY}})^{\otimes k}$ is a basis of $H^{2k}_{\mathrm{DY}}(T_q\text{-}\mathrm{mod})$ for all $k$ and it follows that $H^{\bullet}_{\mathrm{DY}}(T_q\text{-}\mathrm{mod}) \cong \mathbb{C}[X^2]$ as a graded $\mathbb{C}$-algebra.
\end{proposition}
\begin{proof}
1. We show that $\Ext^1_{D(T_q), T_q}\big(\mathbb{C}, \mathcal{V}^{(n-1,1)}\big) = 0$; then according to Corollary \ref{remarkCriterion2FoldSequences} it follows that $S = S_1 \circ S_2 \not\equiv 0$. Recall the relatively projective resolution of $\mathbb{C}$ in \eqref{relativeResolutionTaft} and consider
\[ \Hom_{D(T_q)}\big(\mathcal{V}^{(0,0)}, \mathcal{V}^{(n-1,1)} \big) \overset{d_1^*}{\longrightarrow} \Hom_{D(T_q)}\big(\mathcal{V}^{(n-1,1)}, \mathcal{V}^{(n-1,1)} \big) \overset{d_2^*}{\longrightarrow} \Hom_{D(T_q)}\big(\mathcal{V}^{(0,0)}, \mathcal{V}^{(n-1,1)} \big) \]
It is easy to see that $\Hom_{D(T_q)}\big(\mathcal{V}^{(n-1,1)}, \mathcal{V}^{(n-1,1)} \big) = \mathbb{C}\,\mathrm{id}$ and thus $\ker(d_2^*) = 0$, whence the claim.
\\2. Let us determine $\eta(S_1), \eta(S_2)$ . For $S_1$ we must fill
\[\xymatrix{
\bigl((T_q^*)^{\otimes 2} \otimes \mathcal{K}\bigr)_{\mathrm{coad}} \ar[r]^{\quad d_1} \ar@{-->}[d]^{\eta_1 = \eta(S_1)} & \bigl(T_q^* \otimes \mathcal{K}\bigr)_{\mathrm{coad}} \ar[r]^{\qquad \mathrm{act}} \ar@{-->}[d]^{\eta_0} & \mathcal{K} \ar[r] \ar@{=}[d] &0\\
\mathbb{C} \ar[r] & \mathcal{V}^{(n-1,1)} \ar[r] & \mathcal{K} \ar[r] & 0
}\]
It is straightforward to find a solution:
\[ \eta_0(y^i\varphi_j \otimes s_k) = \delta_{j,k+1} \delta_{i+k \leq n-1} v^{(n-1,1)}_{i+k}, \qquad \eta(S_1)(\varepsilon \otimes y^i\varphi_j \otimes s_k) = -\delta_{j,k+1}\delta_{n-i,k+1}. \]
where $\varphi_j = \frac{1}{n}\sum_{l=0}^{n-1}q^{-lj}h^l$. For $S_2$ one finds similarly:
\[ \eta_0(y^i\varphi_j \otimes 1) = \delta_{j,0}v_i^{(0,0)}, \qquad \eta(S_2)(\varepsilon \otimes y^i \varphi_j \otimes 1) = -\delta_{i > 0}\delta_{j,0}x_{i-1} \]
The morphisms $\eta(S_1), \eta(S_2)$ can be easily extended from these values by the $(T^*_q)^{\mathrm{op}}$-linearity, but we do not need this. From \eqref{pairingTaft}, the dual basis $\bigl( (x^ig^j)^* \bigr)_{0 \leq i,j \leq n-1}$ of the monomial basis $\big(x^ig^j\big)_{0 \leq k, l \leq n-1}$ is given by $(x^ig^j)^* = \frac{1}{(i)_q!} y^i \varphi_j$. With this we can compute the DY cocycles associated to $\eta(S_1), \eta(S_2)$ through the isomorphism \eqref{isoBarGammaInverse}:
\begin{align*}
&(S_1)_{\mathrm{DY}} = \widetilde{\Gamma}^{-1}(\eta(S_1)) : s_k \mapsto \sum_{i,j=0}^{n-1} x^ig^j \otimes \frac{1}{(i)_q!}\eta(S_1)(\varepsilon \otimes y^i\varphi_k \otimes s_k) = \frac{-1}{(n-k-1)_q!}x^{n-k-1}g^{k+1} \otimes 1,\\
&(S_2)_{\mathrm{DY}} = \widetilde{\Gamma}^{-1}(\eta(S_2)) : 1 \mapsto \sum_{i,j=0}^{n-1} x^ig^j \otimes \frac{1}{(i)_q!}\eta(S_2)(\varepsilon \otimes y^i\varphi_k \otimes 1) = -\sum_{i=1}^{n-1} x^i \otimes \frac{1}{(i)_q!}s_{i-1}.
\end{align*}
Then $S_{\mathrm{DY}} = (S_1)_{\mathrm{DY}} \circ (S_2)_{\mathrm{DY}}$ and the formula \eqref{YonedaProductHomH} for the product  $\circ$ gives the result.
\\3. Since $\mathcal{V}^{(0,0)}$ and $\mathcal{V}^{(n-1,1)}$ are relatively projective objects (Lemma \ref{relProjObjectsForTaft}), the sequences $S_1, S_2$ yield for $k \geq 1$ the connecting isomorphisms described in \eqref{connectingMorphismIsoDY}:
\[ \begin{array}{ccccc}
H^k_{\mathrm{DY}}(T_q\text{-}\mathrm{mod}, \mathbb{C}, \mathbb{C}) & \to & H^{k+1}_{\mathrm{DY}}(T_q\text{-}\mathrm{mod}, \mathcal{K}, \mathbb{C}) & \to & \multicolumn{1}{l}{H^{k+2}_{\mathrm{DY}}(T_q\text{-}\mathrm{mod}, \mathbb{C}, \mathbb{C})}\\
\multicolumn{1}{r}{g} & \mapsto & (-1)^l g \circ (S_1)_{\mathrm{DY}} & \mapsto & - g \circ (S_1)_{\mathrm{DY}} \circ (S_2)_{\mathrm{DY}}
\end{array} \]
We know that $H^0_{\mathrm{DY}}(T_q\text{-}\mathrm{mod}) = \mathbb{C}$, so by induction we deduce that $(S_{\mathrm{DY}})^{\otimes k}$ is not equal to $0$ in $H^{2k}_{\mathrm{DY}}(T_q\text{-}\mathrm{mod})$. Since this space has dimension $1$ by \eqref{dimDYTaft}, $(S_{\mathrm{DY}})^{\otimes k}$ is a basis element. For the last claim just note that $(S_{\mathrm{DY}})^{\otimes k} \circ (S_{\mathrm{DY}})^{\otimes l} = (S_{\mathrm{DY}})^{\otimes (k+l)}$.
\end{proof}

\subsection{Restricted quantum group $\bar U_{\mathbf{i}}(\mathfrak{sl}_2)$}\label{sectionExampleBarUiSl2}
\indent Let $\mathbf{i} = \sqrt{-1}$ and $\bar U_{\mathbf{i}} = \bar U_{\mathbf{i}}(\mathfrak{sl}_2)$ be the $\mathbb{C}$-algebra generated by $E,F,K$ modulo
\begin{equation}\label{relationsBarUi}
KE = -EK, \quad KF = -FK, \quad EF - FE = -\frac{\mathbf{i}}{2}(K - K^3), \quad E^2 = F^2 = 0, \quad K^4 = 1.
\end{equation}
The Hopf structure is given by
\begin{equation}\label{HopfStructureUqsl2}
\begin{array}{lll}
\Delta(E) = 1 \otimes E + E \otimes K, & \Delta(F) = F \otimes 1 + K^{-1} \otimes F, & \Delta(K) = K \otimes K,\\
\varepsilon(E) = 0, & \varepsilon(F) = 0, & \varepsilon(K) = 1,\\
S(E) = -EK^{-1}, & S(F) = -KF, & S(K) = K^{-1}.
\end{array}
\end{equation}
We also recall that this Hopf algebra is not factorisable, it admits even no braiding~\cite{GR17}. We thus can't apply the results of Section~\ref{sectionFactoHopfAlg} and proceed via representation theory of its Drinfeld double which we calculate first.

\indent Let $\rho : \bar U_i \to \mathrm{End}(\mathbb{C}^2)$ be the fundamental representation defined by
$\rho(E) = 
\left(\begin{smallmatrix}
0 & 1\\
0 & 0
\end{smallmatrix}\right)$, 
$\rho(F) = 
\left(\begin{smallmatrix}
0 & 0\\
1 & 0
\end{smallmatrix}\right)$, 
$\rho(K) = 
\left(\begin{smallmatrix}
\mathbf{i} & 0\\
0 & -\mathbf{i}
\end{smallmatrix}\right)$.
Write  $\rho = 
\left(\begin{smallmatrix}
a & b\\
c & d
\end{smallmatrix}\right)$, so that $a,b,c,d \in \bar U_{\mathbf{i}}^*$. Then $(\bar U_{\mathbf{i}}^*)^{\mathrm{op}}$ is generated by $b,c,d$ with relations
\begin{equation}\label{relationsBarUiDual}
bc = cb, \:\:\:\: db = -\mathbf{i}bd, \:\:\:\: dc = -\mathbf{i}cd, \:\:\:\: b^2=c^2=0, \:\:\:\: d^4=1.
\end{equation}
The element $a$ is not required as a generator due to the $q$-determinant relation $da + \mathbf{i}cb = 1$, which gives $a = d^3 + \mathbf{i}bcd^3$. The exchange relations of the Drinfeld double are easily computed: $D(\bar U_{\mathbf{i}})$ is the $\mathbb{C}$-algebra generated by $E, F, K, b, c, d$ modulo the relations \eqref{relationsBarUi}--\eqref{relationsBarUiDual} above and
\[\begin{array}{llll}
Ea = \mathbf{i}aE + cK, & Eb = -\mathbf{i}bE + a - dK, & Ec = \mathbf{i}cE, & Ed = -\mathbf{i}dE + c,\\
Fa = \mathbf{i}aF + \mathbf{i}bK^{-1}, & Fb = \mathbf{i}bF, & Fc = -\mathbf{i}cF + \mathbf{i}a - \mathbf{i}dK^{-1}, & Fd = -\mathbf{i}dF + ib,\\
Ka = aK, & Kb = -bK, & Kc = -cK, & Kd = dK.
\end{array}\]
It is useful to introduce the orthogonal idempotents
\begin{equation}\label{FourierTransformsBarUiDual}
\varphi_l = \frac{1}{4}\sum_{j=0}^3 \mathbf{i}^{-jl}d^j
\end{equation}
where $l \in \mathbb{Z}$ is taken modulo $4$. They satisfy $\varepsilon(\varphi_l) = \delta_{l,0}$ and
\begin{equation*}
d\varphi_l = \mathbf{i}^l \varphi_l, \qquad E\varphi_l = \varphi_{l+1}E + \frac{\mathbf{i}^{-l}}{2}c(\varphi_l + \varphi_{l+2}), \qquad F\varphi_l = \varphi_{l+1}F + \frac{\mathbf{i}^{1-l}}{2}b(\varphi_l + \varphi_{l+2}).
\end{equation*}
Moreover $\bigl(b^ic^j\varphi_l\bigr)_{0 \leq i, j \leq 1, \, 0 \leq l \leq 3}$ is a basis of $(\bar U_{\mathbf{i}}^*)^{\mathrm{op}}$.

\subsubsection{Dimension of DY cohomology groups}\label{subsubsectionDYcohomologyBarUi}
\indent It is too difficult to find a relatively projective resolution in this example. Instead, we will use the dimension formulas from Corollary \ref{corollaryDimensionDYGroups}. So let us determine the relatively projective cover $R_{\mathbb{C}}$ of the trivial $D(\bar U_{\mathbf{i}})$-module $\mathbb{C}$; recall from Remark \ref{remarkHowToComputeRelProjCover} that $R_{\mathbb{C}}$ is the minimal direct summand of $G(\mathbb{C})$ on which the counit $\varepsilon : G(\mathbb{C}) = (\bar U^*_{\mathbf{i}})_{\mathrm{coad}} \to \mathbb{C}$ does not vanish. We find that 
\[ (\bar U^*_{\mathbf{i}})_{\mathrm{coad}} = \bigl\langle b^ic^j\varphi_1 \bigr\rangle_{0 \leq i, j \leq 1} \oplus \bigl\langle b^ic^j\varphi_3 \bigr\rangle_{0 \leq i, j \leq 1} \oplus \bigl\langle b^ic^j\varphi_0, \, b^ic^j\varphi_2 \bigr\rangle_{0 \leq i, j \leq 1} \]
as a $D(\bar U_{\mathbf{i}})$-module, where $\langle \ldots \rangle$ means linear span. The direct summands $\bigl\langle b^ic^j\varphi_1 \bigr\rangle_{0 \leq i, j \leq 1}$ and \\$\bigl\langle b^ic^j\varphi_3 \bigr\rangle_{0 \leq i, j \leq 1}$ both are simple; the counit vanishes on them because $\varepsilon(\varphi_l) = \delta_{l,0}$. So the relevant direct summand is $R_{\mathbb{C}} = \bigl\langle b^ic^j\varphi_0, \, b^ic^j\varphi_2 \bigr\rangle_{0 \leq i, j \leq 1}$. It is indecomposable and its generating elements are a basis of eigenvectors for $d$ on which the $(\bar U_{\mathbf{i}}^*)^{\mathrm{op}}$-action is obvious (given by multiplication) and where the $\bar U_{\mathbf{i}}$-action is entirely determined by
\[\begin{array}{l}
E \cdot \varphi_0 = \frac{1}{2}(c\varphi_0 + c\varphi_2), \quad E \cdot \varphi_2 = -\frac{1}{2}(c\varphi_0 + c\varphi_2), \\[.5em]
F \cdot \varphi_0 = \frac{\mathbf{i}}{2}(b\varphi_0 + b\varphi_2), \quad F \cdot \varphi_2 = -\frac{\mathbf{i}}{2}(b\varphi_0 + b\varphi_2),\\[.5em]
\forall \, x, \: K\cdot x = d^2 \cdot x.
\end{array} \]
The action can be schematized by
\[ \xymatrix{
 & *+[F]{ \varphi_0, \:\: \varphi_2 } \ar[ld]_{F,b} \ar[rd]^{E,c}& \\
*+[F]{ b\varphi_0, \:\: b\varphi_2 } \ar[rd]_{E,c} & & *+[F]{ c\varphi_0, \:\: c\varphi_2 } \ar[ld]^{F,b}\\  
 & *+[F]{ bc\varphi_0, \:\: bc\varphi_2 } &
}\]

\smallskip

\indent Let $K_{\mathbb{C}} = \ker(\varepsilon_{|R_{\mathbb{C}}})$, which is generated by $\varphi_2$. By definition we have the allowable short exact sequence $0 \longrightarrow K_{\mathbb{C}} \longrightarrow R_{\mathbb{C}} \longrightarrow \mathbb{C} \longrightarrow 0$. It is not difficult to show that 
\[ R_{\mathbb{C}}^{\vee} \cong R_{\mathbb{C}} \quad \text{and} \quad K_{\mathbb{C}}^{\vee} \cong R_{\mathbb{C}}/\langle bc\varphi_2 \rangle. \]
Note that $H_{\mathrm{DY}}^0(\bar U_{\mathbf{i}}\text{-}\mathrm{mod}) \cong \mathbb{C}$ (as always for identity functor with trivial coefficients) and $H_{\mathrm{DY}}^1(\bar U_{\mathbf{i}}\text{-}\mathrm{mod}) = 0$ by Proposition \ref{propH1Equals0}. For $H_{\mathrm{DY}}^2(\bar U_{\mathbf{i}}\text{-}\mathrm{mod})$, Corollary \ref{corollaryDimensionDYGroups} and an easy computation gives
\begin{equation*}
\dim H_{\mathrm{DY}}^2(\bar U_{\mathbf{i}}\text{-}\mathrm{mod}) = \dim \Hom_{D(\bar U_{\mathbf{i}})}(K_{\mathbb{C}}, K_{\mathbb{C}}^{\vee}) - \dim \Hom_{D(\bar U_{\mathbf{i}})}(R_{\mathbb{C}}, K_{\mathbb{C}}^{\vee}) = 2 - 2 = 0.
\end{equation*}

 For higher cohomology groups, the computation of the Hom spaces appearing in   Corollary \ref{CoroDimExtWithHom} becomes  more difficult because it requires to determine the subquotient structure
 of $(K_{\mathbb{C}}^{\vee})^{\otimes (n-1)}$. However, using that $\Hom_{\mathcal{C}}(V,W) \cong \Hom_{\mathcal{C}}(\mathbf{1},W \otimes V^{\vee})$ in any tensor category $\mathcal{C}$ and the fact that $R_{\mathbb{C}}$ is self-dual we get
\begin{equation}\label{dimensionFormulaDYcohomologyWithInvForBarUi}
\dim H_{\mathrm{DY}}^n(\bar U_{\mathbf{i}}\text{-}\mathrm{mod}) = \dim \mathrm{Inv}\bigl( (K_{\mathbb{C}}^{\vee})^{\otimes n} \bigr) - \dim \mathrm{Inv}\bigl( (K_{\mathbb{C}}^{\vee})^{\otimes (n-1)} \otimes R_{\mathbb{C}} \bigr) + \dim \mathrm{Inv}\bigl( (K_{\mathbb{C}}^{\vee})^{\otimes (n-1)} \bigr)
\end{equation}
where $\mathrm{Inv}(M) = \Hom_{D(\bar U_{\mathbf{i}})}(\mathbb{C},M)$ is the subspace of invariant elements in a $D(\bar U_{\mathbf{i}})$-module $M$. Computing the dimension of these invariant subspaces simply amounts to find the number of solutions of a homogeneous linear system, which we did with a program written in GAP4\footnote{\url{https://github.com/mfaitg/DY\_cohomology/blob/main/dimension\_formula\_DY\_barUi.g}}. For $n=3$ and $n=4$ we find:
\begin{equation}\label{dimH3H4BarUi}
\dim H_{\mathrm{DY}}^3(\bar U_{\mathbf{i}}\text{-}\mathrm{mod}) = 10 - 9 + 2 = 3, \qquad \dim H_{\mathrm{DY}}^4(\bar U_{\mathbf{i}}\text{-}\mathrm{mod}) = 44 - 54 + 10 = 0.
\end{equation}

\subsubsection{Explicit cocycles}\label{sectionExplicitCocyclesBarUi}
\indent Here we address the problem of finding explicit cocycles in $H^3_{\mathrm{DY}}(\bar U_{\mathbf{i}}\text{-}\mathrm{mod})$. Thanks to two allowable $3$-fold exact sequences we constructed two explicit DY $3$-cocycles with the method of \S \ref{methodDYCocyclesThanksToSequences}, as we now explain. For $\lambda, \mu \in \{ \pm 1, \pm \mathbf{i} \}$, let $\mathcal{W}(\lambda, \mu)$ be the $D(\bar U_{\mathbf{i}})$-module with basis $(w, bw, Fw, Fbw)$ where $w$ satisfies
\[ Ew = 0, \:\: cw = 0, \:\: Kw = \lambda w, \:\: dw = \mu w \]
(formally it is the induced representation $\mathbb{C}\langle F,b \rangle \otimes_{\mathbb{C}\langle E,c,K,d \rangle} \mathbb{C}_{\lambda, \mu}$). Define
\begin{equation}\label{basisV}
 w_{\lambda, \mu} = w, \:\:\:\: w_{-\lambda, \mathbf{i}\mu} = Fw - \frac{\mathbf{i}}{2\mu}bw, \:\:\:\: w_{-\lambda, -\mathbf{i}\mu} = bw, \:\:\:\: w'_{\lambda, \mu} = Fbw.
\end{equation}
These elements form a basis of weight vectors for the actions of $K$ and $d$ and the subscripts indicate the weights for $K$ and $d$ respectively. The actions of the other generators are:
\begin{align*}
&Ew_{\lambda, \mu} = 0, \quad Fw_{\lambda,\mu} = w_{-\lambda, \mathbf{i}\mu} + \dfrac{\mathbf{i}}{2\mu}w_{-\lambda, -\mathbf{i}\mu}, \quad bw_{\lambda, \mu} = w_{-\lambda, -\mathbf{i}\mu}, \quad cw_{\lambda, \mu} = 0,\\
&Ew_{-\lambda, \mathbf{i}\mu} = \dfrac{\mathbf{i}}{2}(\lambda^{-1} - \mu^2)w_{\lambda, \mu}, \quad Fw_{-\lambda, \mathbf{i}\mu} = \dfrac{-\mathbf{i}}{2\mu}w'_{\lambda, \mu}, \quad bw_{-\lambda, \mathbf{i}\mu} = -\mathbf{i}w'_{\lambda, \mu}, \quad cw_{-\lambda, \mathbf{i}\mu} = \mu(\mu^2 - \lambda^{-1})w_{\lambda, \mu},\\
&Ew_{-\lambda, -\mathbf{i}\mu} = \mu(\mu^2 - \lambda)w_{\lambda, \mu}, \quad Fw_{-\lambda, -\mathbf{i}\mu} = w'_{\lambda, \mu}, \quad bw_{-\lambda, -\mathbf{i}\mu} = 0, \quad cw_{-\lambda, -\mathbf{i}\mu} = 0,\\
&Ew'_{\lambda, \mu} = \mathbf{i}\dfrac{\mu^2 - \lambda^{-1}}{2}w_{-\lambda, -\mathbf{i}\mu} - \mu(\lambda - \mu^2)w_{-\lambda, \mathbf{i}\mu}, \quad Fw'_{\lambda, \mu} = 0, \quad bw'_{\lambda, \mu} = 0, \quad cw'_{\lambda, \mu} = \mathbf{i}\mu(\mu^2 - \lambda^{-1})w_{-\lambda, -\mathbf{i}\mu}
\end{align*}
From these formulas it is not difficult to show that $\mathcal{W}(\lambda, \mu)$ is simple if and only if $\lambda \neq \mu^2$.\footnote{The complete list of simple $D(\bar U_{\mathbf{i}})$-modules consists actually of the $\mathcal{W}(\lambda, \mu)$ with $\lambda \neq \mu^2$ together with the $\mathbb{C}_{\mu}$, which is the $1$-dimensional module defined by $K\cdot 1 = \mu^2$, $c \cdot 1 = \mu$ and the other generators act by $0$.} For $\lambda = \mu^2$ we write $\mathcal{W}(\mu)$ instead of $\mathcal{W}(\mu^2, \mu)$ and we denote the basis of weight vectors by $w_{\mu}, w_{\mathbf{i}\mu}, w_{\mathbf{i}\mu}, w'_{\mu}$ since $K$ acts as $d^2$. We also notice that $E$ and $c$ act by $0$ on $\mathcal{W}(\mu)$, so that the action is depicted by:
\[\xymatrix@R=1em{
\mathbb{C}_{\mu} = \langle w_{\mu} \rangle \ar[d]^{F,b} \\
\mathbb{C}_{\mathbf{i}\mu} \oplus \mathbb{C}_{-\mathbf{i}\mu} = \langle w_{\mathbf{i}\mu}, w_{-\mathbf{i}\mu} \rangle \ar[d]^{F,b}\\
\mathbb{C}_{\mu} = \langle w'_{\mu} \rangle
}\]
where $\langle \ldots \rangle$ means linear span. Consider the following modules:
\[ \begin{array}{ll}
M = \mathcal{W}(\mathbf{i})/\langle w'_{\mathbf{i}}\rangle = \langle \overline{w}_{\mathbf{i}}, \overline{w}_{-1}, \overline{w}_1 \rangle, & N = \langle w_1, w_{-1}, w'_{-\mathbf{i}}\rangle \subset \mathcal{W}(-\mathbf{i}),\\[.2em]
V_{\mathbf{i}} = \langle w_{\mathbf{i}}, w'_{-1}\rangle \subset \mathcal{W}(-1), & V_{-1} = \mathcal{W}(-1)/V_{\mathbf{i}} = \langle \overline{w}_{-1}, \overline{w}_{-\mathbf{i}} \rangle.
\end{array} \]
They give three allowable short exact sequences:
\begin{align*}
&S_1 = \big( 0 \to \mathbb{C} \to M \to V_{\mathbf{i}} \to 0 \big), \quad S_2 = \big( 0 \to V_{\mathbf{i}} \to \mathcal{W}(-1) \to V_{-1} \to 0 \big), \\
&\qquad \qquad \qquad \quad S_3 = \big( 0 \to V_{-1} \to N \to \mathbb{C} \to 0 \big),
\end{align*}
such that $S_1 \circ S_2 \circ S_3$ is a $3$-fold sequence from $\mathbb{C}$ to $\mathbb{C}$. Let us compute $\widetilde{\Gamma}^{-1}\eta(S_i)$. For instance for $S_3$ we must first fill the following diagram:
\[\xymatrix{
\bigl((\bar U_{\mathbf{i}}^*)^{\otimes 2} \otimes \mathbb{C}\bigr)_{\mathrm{coad}} \ar[r]^{\quad d_1} \ar@{-->}[d]^{\eta_1 = \eta(S_1)} & \bigl(\bar U_{\mathbf{i}}^* \otimes \mathbb{C}\bigr)_{\mathrm{coad}} \ar[r]^{\qquad \mathrm{act}} \ar@{-->}[d]^{\eta_0} & \mathbb{C} \ar[r] \ar@{=}[d] &0\\
V_{-1} \ar[r] & N \ar[r] & \mathbb{C} \ar[r] & 0
}\]
($\mathrm{act}$ is equal to $\varepsilon$ in this case: $\mathrm{act}(\psi \otimes 1) = \varepsilon(\psi)$). Let us give a few details on how to find a solution. Recall the elements $\varphi_l$ defined in \eqref{FourierTransformsBarUiDual}. Since $\mathrm{act}(\varphi_1 \otimes 1) = \mathrm{act}(\varphi_3 \otimes 1) = 0$, we set $\eta_0(\varphi_1 \otimes 1) = \eta_0(\varphi_3 \otimes 1) = 0$. On the other hand, due to the weights and to the commutativity of the right square we necessarily have $\eta_0(\varphi_0 \otimes 1) = w_1$ and $\eta_0(\varphi_2 \otimes 1) = \lambda w_{-1}$ for some scalar $\lambda$. Then we note that
\[ \eta_0(F \cdot (\varphi_2 \otimes 1)) = \eta_0(\varphi_3 \otimes F \cdot 1) - \frac{\mathbf{i}}{2}\eta_0(b(\varphi_0 + \varphi_2) \otimes 1) = -\frac{1}{2}w'_{-\mathbf{i}}, \quad F \cdot \eta_0(\varphi_2 \otimes 1) = \lambda w'_{-\mathbf{i}} \]
which reveals that $\lambda = -\frac{1}{2}$. Hence
\[ \eta_0\big(b^jc^k\varphi_l \otimes 1\big) = b^jc^k\eta_0\big(\varphi_l \otimes 1\big) = \delta_{j,0}\delta_{k,0}\delta_{l,0}w_1 - \mathbf{i}\delta_{j,1}\delta_{k,0}\delta_{l,0}w'_{-\mathbf{i}} - \frac{1}{2}\delta_{j,0}\delta_{k,0}\delta_{l,2}w_{-1}. \]
It follows that
\[ \eta_0d_1\big( \varepsilon \otimes b^jc^k\varphi_l \otimes 1 \big) = \eta_0\big( \delta_{j,0}\delta_{k,0}\delta_{l,0} \varepsilon \otimes 1 - b^jc^k\varphi_l \otimes 1 \big) = \frac{1}{2}\delta_{j,0}\delta_{k,0}(\delta_{l,2} - \delta_{l,0})w_{-1} +  \mathbf{i}\delta_{j,1}\delta_{k,0}\delta_{l,0}w_{-\mathbf{i}} \]
and a solution is simply given by
\[ \eta(S_1)\big( \varepsilon \otimes b^jc^k\varphi_l \otimes 1 \big) = \frac{1}{2}\delta_{j,0}\delta_{k,0}(\delta_{l,2} - \delta_{l,0})\overline{w}_{-1} +  \mathbf{i}\delta_{j,1}\delta_{k,0}\delta_{l,0}\overline{w}_{-\mathbf{i}}. \]
Recall from \eqref{isoBarGammaInverse} that to compute $\widetilde{\Gamma}^{-1}$ we need a pair of dual bases. It is not difficult to find the dual basis of the monomial basis of $\bar U_{\mathbf{i}}$:
\[ (K^l)^* = \varphi_{-l}, \quad (EK^l)^* = \mathbf{i}^lb\varphi_{-l}, \quad (FK^l)^* = \mathbf{i}^{-l}c\varphi_{1-l}, \quad (EFK^l)^* = bc\varphi_{1-l}. \]
(recall that the products on the right-hand sides are in $(\bar U_{\mathbf{i}}^*)^{\mathrm{op}}$). It follows that
\[ (S_3)_{\mathrm{DY}} = \widetilde{\Gamma}^{-1}\eta(S_3) : \mathbb{C} \to \bar U_{\mathbf{i}} \otimes V_{-1}, \quad 1 \mapsto -\boldsymbol{e} \otimes \overline{w}_{-1} + \mathbf{i}E \otimes \overline{w}_{-\mathbf{i}} \]
where $\boldsymbol{e} = \frac{1 - K^2}{2}$. One similarly finds
\begin{align*}
&(S_1)_{\mathrm{DY}} = \widetilde{\Gamma}^{-1}\eta(S_1) : V_{\mathbf{i}} \to \bar U_{\mathbf{i}} \otimes \mathbb{C}, \quad w_{\mathbf{i}} \mapsto \mathbf{i}EK^3 \otimes 1, \quad w'_{-1} \mapsto -\boldsymbol{e} \otimes 1,\\
&(S_2)_{\mathrm{DY}} = \widetilde{\Gamma}^{-1}\eta(S_2) : V_{-1} \to \bar U_{\mathbf{i}} \otimes V_{\mathbf{i}}, \quad \overline{w}_{-1} \mapsto EK^2 \otimes w_{\mathbf{i}}, \quad \overline{w}_{-\mathbf{i}} \mapsto -\frac{\mathbf{i}}{2}K\boldsymbol{e} \otimes w_{\mathbf{i}} - EK \otimes w_{-1}.
\end{align*}
Thanks to \eqref{SDYYoneda} and to the formula \eqref{YonedaProductHomH} for the product $\circ$ we find
\[ (S_1 \circ S_2 \circ S_3)_{\mathrm{DY}} = (S_1)_{\mathrm{DY}} \circ (S_2)_{\mathrm{DY}} \circ (S_3)_{\mathrm{DY}} = \mathbf{i}c_1 \]
where $c_1$ is given in Proposition \ref{propDYbarUi} below.

\smallskip

\indent In a completely analogous fashion, we can consider the $D(\bar U_{\mathbf{i}})$-module $\overline{\mathcal{W}}(\lambda,\mu)$ with basis $(w, cw$, $Ew, Ecw)$ where $w$ satisfies
\[ Fw = 0, \:\: bw = 0, \:\: Kw = \lambda w, \:\: dw = \mu w \]
and if we reproduce the above reasoning with $\overline{\mathcal{W}}(\lambda,\mu)$ instead of $\mathcal{W}(\lambda,\mu)$, we will discover the cocycle $c_2$ in Proposition \ref{propDYbarUi} below (up to a scalar).

\smallskip

\indent Unfortunately, we have not been able to show that the cocycles $c_1, c_2$ so obtained are a free family in $H^3_{\mathrm{DY}}(\bar U_{\mathbf{i}}\text{-}\mathrm{mod})$, or to find another explicit $3$-cocycle. Recall from \eqref{dimH3H4BarUi} that $\dim\bigl(H^3_{\mathrm{DY}}(\bar U_{\mathbf{i}}\text{-}\mathrm{mod})\bigr) = 3$; with the help of a program written in GAP4\footnote{\url{https://github.com/mfaitg/DY\_cohomology/blob/main/basis\_H3\_DY\_barUi.g}}, we get the following:

\begin{proposition}\label{propDYbarUi}
A basis of explicit cocycles for $H^3_{\mathrm{DY}}(\bar U_{\mathbf{i}}\text{-}\mathrm{mod})$ is
\begin{align*}
& c_1 = \boldsymbol{e}  \otimes  EK^2  \otimes  EK^3 - E  \otimes  K\boldsymbol{e}  \otimes  EK^3 - E  \otimes  EK  \otimes  \boldsymbol{e},\\[.3em]
& c_2 = \boldsymbol{e}  \otimes  FK^3  \otimes  F - FK  \otimes  K\boldsymbol{e}  \otimes  F - FK  \otimes  FK^2  \otimes \boldsymbol{e},\\[.3em]
& c_3 = \boldsymbol{e}  \otimes  EK^2  \otimes  FK^2 - \boldsymbol{e}  \otimes  FK^3  \otimes  EK^3 + EK^2  \otimes  K\boldsymbol{e}  \otimes  FK^2 \\
& \qquad + FK  \otimes  K\boldsymbol{e}  \otimes  EK^3 - EK^2 \otimes  FK^2  \otimes \boldsymbol{e} + FK  \otimes  EK  \otimes \boldsymbol{e}
\end{align*}
where $\boldsymbol{e} = \frac{1 - K^2}{2}$.
\end{proposition}
\begin{proof}
We use the description of the DY complex given in \eqref{DYComplexExplicitFormWithCentralizers}--\eqref{differentialDYHmodForTrivialCoeffs}. It is a purely linear problem. In GAP4, we represent $\bar U_{\mathbf{i}}$ as a $16$-dimensional vector space over $\mathsf{Qi} := \mathsf{GaussianRationals}$, together with a unit tensor $\mathsf{un} : \mathsf{Qi} \to \bar U_{\mathbf{i}}$ and matrices $\mathsf{m}_E, \mathsf{m}_F, \mathsf{m}_K$ for the left multiplication by $E,F,K$ respectively. From these matrices, we easily construct
\begin{itemize}
\item the multiplication tensor $\mathsf{m} : \bar U_{\mathbf{i}}^{\otimes 2} \to \bar U_{\mathbf{i}}$ (a matrix of size $16 \times 256$),
\item the comultiplication tensor $\mathsf{com} : \bar U_{\mathbf{i}} \to \bar U_{\mathbf{i}}^{\otimes 2}$ (a matrix of size $256 \times 16$),
\item the matrices $\mathsf{r}_E, \mathsf{r}_F, \mathsf{r}_K$ for the right multiplication by $E,F,K$ respectively.
\end{itemize}
From this data it is easy to determine the centralizers of the iterated coproduct and to define the matrices of the differentials $\delta^2, \delta^3$, recall~\eqref{differentialDYHmodForTrivialCoeffs}. Finally, one looks for $3$ basis elements of $\ker(\delta^3)$ which generate a supplementary to $\mathrm{im}(\delta^2) \subset \ker(\delta^3)$.
\end{proof}

\begin{remark}\label{remarkRelExtNot0UsualExt0BarUi}
\begin{enumerate}
\item From the results in \eqref{dimH3H4BarUi}, the category $\bar U_{\mathbf{i}}\text{-}\mathrm{mod}$ admits a $3$-parameter family of deformations of its trivial associator, up to soft tensor autoequivalence \cite{davydov}, and each of them can be extended to the next order due to the fact that $H^4_{\mathrm{DY}} = 0$ \cite[Prop. 3.21]{BD}. In Hopf algebra terms, Proposition \ref{propDYbarUi} means that any infinitesimal coassociator on $\bar U_{\mathbf{i}}$ is gauge equivalent to $\lambda_1c_1 + \lambda_2 c_2 + \lambda_3c_3$ for certain $\lambda_i \in \mathbb{C}$.
\item
We have seen two methods to compute the DY cohomology using a computer program: at the end of \S \ref{subsubsectionDYcohomologyBarUi} and in the previous proof.  If one is just interested in the dimensions of the DY cohomology groups, the method explained at the end of \S \ref{subsubsectionDYcohomologyBarUi} is much faster, although it requires to do a bit of representation theory in order to determine the relatively projective cover of $\mathbb{C}$.

\item  By \eqref{DYandRelExtDHH} and \eqref{dimH3H4BarUi}, we get $\Ext^3_{D(\bar U_{\mathbf{i}}), \bar U_{\mathbf{i}}}(\mathbb{C}, \mathbb{C}) \cong \mathbb{C}^3$.  On the other hand by finding a usual projective resolution of $\mathbb{C}$ one can compute that $\Ext^3_{D(\bar U_{\mathbf{i}})}(\mathbb{C}, \mathbb{C}) = 0$. Hence the two allowable $3$-fold exact sequences constructed above form a free family in $\YExt^3_{D(\bar U_{\mathbf{i}}), \bar U_{\mathbf{i}}}(\mathbb{C}, \mathbb{C}) \cong \mathbb{C}^3$ but are equal to $0$ in \mbox{$\YExt^3_{D(\bar U_{\mathbf{i}})}(\mathbb{C}, \mathbb{C}) = 0$.}

\item
 It is straightforward to compute a usual projective resolution of $\mathbb{C}$ over $(U_{\mathbf{i}}^*)^{\mathrm{op}}$ by using the defining relations \eqref{relationsBarUiDual}. Thanks to Lemma \ref{DYForgetfulExt} we get $H^3_{\mathrm{DY}}(U) = 0$, where $U : \bar U_{\mathbf{i}}\text{-}\mathrm{mod} \to \mathrm{Vect}_{\mathbb{C}}$ is the forgetful functor. This means that the map $\overline{\iota}_3 : H^3_{\mathrm{DY}}(\bar U_{\mathbf{i}}\text{-}\mathrm{mod}) \to H^3_{\mathrm{DY}}(U)$ defined in Remark \ref{remarkConcreteDY} is equal to $0$. In other words, there exist elements $b_i \in \bar U_{\mathbf{i}}^{\otimes 2}$ such that $\iota_3(c_i) = \delta^2(b_i)$ for $i = 1,2,3$. This is in contrast to the previous two examples, the algebras $B_k$ and the Taft algebras, where the DY cohomologies of the identity functor equal those of the forgetful functor.
\end{enumerate}
\end{remark}

\subsection{Small quantum group $u_q(\mathfrak{sl}_2)$}\label{sectionSmallQuantumGroupH2}
Let $q$ be a primitive root of unity of odd order $p \geq 3$. The small quantum group $u_q = u_q(\mathfrak{sl}_2)$ is the $\mathbb{C}$-algebra generated by $E,F,K$ modulo
\[ KE = q^2EK, \quad KF = q^{-2}FK, \quad EF - FE = \frac{K - K^{-1}}{q - q^{-1}}, \quad E^p = F ^p = 0, \quad K^p = 1. \]
The Hopf structure is given by the same formulas as in \eqref{HopfStructureUqsl2}.

\smallskip

\indent The Hopf algebra $u_q$ is known to be factorizable \cite[Cor.\,A.3.3]{lyu}, so the results of \S\ref{sectionFactoHopfAlg} apply and we will not need to compute the Drinfeld double of $u_q$.
\begin{proposition}
$H^2_{\mathrm{DY}}\bigl(u_q(\mathfrak{sl}_2)\text{-}\mathrm{mod}\bigr) = 0$.
\end{proposition}
\begin{proof}
Let $Q_1$ be as in \S\ref{sectionFactoHopfAlg}, $K_1 = \ker(\varepsilon : Q_1 \to \mathbb{C})$ and $T = Q_1/(\mathbb{C} \boxtimes \mathbb{C})$. Then by Corollary \ref{dimFormulaDYFactoHopfAlg}:
\[ \dim H^2_{\mathrm{DY}}(u_q\text{-}\mathrm{mod}) = \dim \Hom_{u_q \otimes u_q}(K_1,T) - \dim \Hom_{u_q \otimes u_q}(Q_1,T). \]
Recall that there are $p$ isomorphism classes of simple $u_q$-modules, denoted by $\mathscr{X}(s)$ with $1 \leq s \leq p$. $\mathscr{X}(s)$ has dimension $s$ and is generated by a vector $v_s$ such that $Ev_s = 0$ and $Kv_s = q^{s-1}v_s$. In particular $\mathscr{X}(1) = \mathbb{C}$. By the same analysis as in \cite[\S 4.4.2]{FGST}, the structure of $Q_1$ can be depicted by diagrams similar to the ones given below \cite[Prop. 4.4.2]{FGST} (one has to take $s=1$ and to erase the signs on the simple modules). Then $K_1$ is obtained by ignoring the node $\mathscr{X}(1) \boxtimes \mathscr{X}(1)$ at the top of the diagrams of $Q_1$ and the arrows out of it. Using the diagrammatic representation of the structure of these modules, we easily find
\[ \dim \Hom_{u_q \otimes u_q}(K_1,T) = 2, \qquad \dim \Hom_{u_q \otimes u_q}(Q_1,T) = 2. \qedhere \]
\end{proof}

\indent For higher cohomology groups, the computation of the Hom spaces appearing in Corollary \ref{CoroDimExtWithHom} becomes much more difficult because it requires to find the structure of $K_1^{\otimes (n-1)}$. If the root of unity $q$ has a small order, one can use a computer with the same strategy as in \eqref{dimensionFormulaDYcohomologyWithInvForBarUi}, but with the formula of Corollary \ref{dimFormulaDYFactoHopfAlg}. For instance if $q^3=1$, 
\begin{align*}
\dim H^3_{\mathrm{DY}}(u_q\text{-}\mathrm{mod}) &= \dim \mathrm{Inv}\bigl( T^{\otimes 3} \bigr) - \dim \mathrm{Inv}\bigl( T^{\otimes 2} \otimes Q_1 \bigr) + \dim \mathrm{Inv}\bigl( T^{\otimes 2} \bigr)\\
&= 8 - 9 + 2 = 1
\end{align*}
where $\mathrm{Inv}(V) = \Hom_{u_q \otimes u_q}(\mathbb{C} \boxtimes \mathbb{C}, V)$ and $T = Q_1/(\mathbb{C} \boxtimes \mathbb{C})$. For the first equality we used that $Q_1$ is self-dual and $K_1^{\vee} \cong T$ (Proposition \ref{propPropertiesQ1Facto}) while the second equality is obtained by a program\footnote{\url{https://github.com/mfaitg/DY_cohomology/blob/main/dimension_formula_DY_uj.g}} written in GAP4.

\end{document}